\documentclass[a4paper,11pt]{article}
\pdfoutput=1

\usepackage[english]{bricearticle}

\title{Bi-Lagrangian structures and Teichmüller theory}
\author{Brice Loustau\footnote{Heidelberg University, Mathematisches Institut. 69120 Heidelberg, Germany. \newline
E-mail: \url{bloustau@mathi.uni-heidelberg.de}} ~ 
and Andrew Sanders\footnote{Heidelberg University, Mathematisches Institut. 69120 Heidelberg, Germany. \newline
E-mail: \url{asanders@mathi.uni-heidelberg.de}}}
\date{} % No date in title page

\hypersetup{
    pdftitle={Bi-Lagrangian structures and Teichmüller theory},
    pdfauthor={Brice Loustau and Andrew Sanders},
}

\begin{document}

\pdfbookmark[1]{Title page, abstract}{Title}
\hypersetup{pageanchor=false}
\begin{titlepage}
\maketitle
\thispagestyle{empty}

\bigskip \bigskip
\bigskip \bigskip

\begin{abstract}

In this paper, we review or introduce several differential structures on manifolds in the general setting of real and complex differential geometry, and apply this study to Teichmüller theory. We focus on bi-Lagrangian \ie{} para-Kähler structures, 
which consist of a symplectic form and a pair of transverse Lagrangian foliations. We prove in particular that the complexification of a real-analytic Kähler manifold has a natural complex bi-Lagrangian structure. We then specialize to moduli spaces of geometric structures on closed surfaces, and show that old and new geometric features of these spaces are formal consequences of the general theory. We also gain clarity on several well-known results of Teichmüller theory by deriving them from pure differential geometric machinery.

\bigskip \bigskip
\bigskip \bigskip


\bigskip

\noindent \textbf{2000 Mathematics Subject Classification:} 
Primary: 
53C15; % General geometric structures on manifolds (almost complex, almost product structures, etc.)
Secondary: 
% 30F10 $\cdot$ % Compact Riemann surfaces and uniformization
% 30F35 $\cdot$ % Fuchsian groups and automorphic functions
% 30F40 $\cdot$ % Kleinian groups
30F60 $\cdot$ % Teichmüller theory
% 32G15 $\cdot$ % Moduli of Riemann surfaces, Teichmüller theory
32Q15 $\cdot$ % Kähler manifolds
% 32Q30 $\cdot$ % Uniformization
53B05 $\cdot$ % Linear and affine connections
% 53B21 $\cdot$ % Methods of Riemannian geometry
53B35 $\cdot$ % Hermitian and Kählerian structures
% 53C25 $\cdot$ % Special Riemannian manifolds (Einstein, Sasakian, etc.)
53C26 $\cdot$ % Hyper-Kähler and quaternionic Kähler geometry, "special'' geometry
% 53C30 $\cdot$ % Homogeneous manifolds
% 53C55 $\cdot$ % Hermitian and Kählerian manifolds
% 53C56 $\cdot$ % Other complex differential geometry
53D05 $\cdot$ % Symplectic manifolds, general
% 53D12 $\cdot$ % Lagrangian submanifolds; Maslov index
53D30 $\cdot$ % Symplectic structures of moduli spaces
57M50 % Geometric structures on low-dimensional manifolds

\end{abstract}

\end{titlepage}
\hypersetup{pageanchor=true}

% \newgeometry{top=0.09\paperheight, bottom=0.11\paperheight}
\thispagestyle{empty}
\setcounter{tocdepth}{2}
\pdfbookmark[1]{Contents}{Contents}  % For hyperref thumbnail in pdf file
{\tableofcontents}
\addtocounter{page}{1}
% \restoregeometry

\cleardoublepage\phantomsection % So that the hyperref thumbnail is correct
\section*{Introduction}
\addcontentsline{toc}{section}{Introduction}

A Lagrangian foliation of a symplectic manifold $(M, \omega)$ is a foliation by Lagrangian submanifolds, \ie{} maximally isotropic submanifolds for $\omega$. These play a major role in geometric quantization where they are known as real polarizations (see \eg{} \cite{MR1806388}).
Weinstein \cite{MR0286137} showed that each leaf of a Lagrangian foliation admit a natural affine structure given by a flat torsion-free connection
called the \emph{Bott connection}.
The Bott connection is only defined tangentially to the foliation; it can always be extended to a full connection but in general there is no preferred extension. However, if $(M, \omega)$ is given \emph{two} transverse Lagrangian foliations $\mathcal{F}_1$ and $\mathcal{F}_2$, there exists a unique 
symplectic connection $\nabla$ on $M$
which extends the Bott connection of each foliation. This was first proved by He{\ss} \cite{MR607691, HessPhD}, who studied symplectic connections
in polarized symplectic manifolds as a tool for geometric quantization. 
In this scenario $(\omega, \mathcal{F}_1,\mathcal{F}_2)$ is called a bi-Lagrangian structure.

Bi-Lagrangian structures can alternatively be described as para-Kähler structures via an approach that may be more appealing to complex geometers.
Note that on a smooth manifold $M$, the data of an ordered pair of transverse distributions $L_1, L_2 \subset \upT M$ of the same dimension is equivalent to a traceless endomorphism $F$ of $\upT M$ such that $F^2 = \mathbf{1}$, called an \emph{almost para-complex} structure.
As for an almost complex structure, 
there is a notion of integrability for $F$, which amounts to the involutivity of both distributions $L_1$ and $L_2$.
By the Frobenius theorem, this is equivalent to $L_1$ and $L_2$ being the tangent bundles to two transverse foliations $\mathcal{F}_1$ and $\mathcal{F}_2$. The condition that $\mathcal{F}_1$ and $\mathcal{F}_2$
are both Lagrangian is equivalent to $g(u,v) \coloneqq \omega(Fu, v)$ being symmetric, in which case it is a pseudo-Riemannian metric of signature $(n,n)$, called a neutral metric. Therefore, a bi-Lagrangian structure
 $(\omega, \mathcal{F}_1,\mathcal{F}_2)$ is equivalent to a \emph{para-Kähler structure} $(g, F, \omega)$.

Thus, bi-Lagrangian structures lie at the intersection of symplectic, para-complex, and neutral pseudo-Riemannian geometry. They also relate
to affine differential geometry via the Bott connection. Para-complex geometry is perhaps the least heavily investigated of these fields, 
although it has long been on mathematicians' radars (\cite{MR0056351}, \cite{MR0048893}). It shares many similarities with complex geometry, 
with the algebra of para-complex numbers playing the role of complex numbers (the imaginary unit squares to $1$ instead of $-1$), with some important differences nonetheless. 
Interested readers may refer to \cite{MR1386154} and \cite[Chapter 15]{MR2605651}  for expositions on para-complex geometry.

In this paper we are interested in the analog of bi-Lagrangian structures in the complex setting, which is essentially absent from the existing literature.
Hence, we review complex symplectic structures and complex Lagrangian foliations, introduce complex Bott connections and complex bi-Lagrangian connections, \etc{} The algebra of bicomplex numbers naturally appears when 
one introduces para-complex structures in the complex setting.
This is the algebra over the reals generated by three units $i$, $j$, $f$, such that $i^2 = j^2 = -f^2 = -1$ and $ij = ji = f$. 
The fact that any two of the three units are sufficient to generate this algebra enables several characterizations of complex bi-Lagrangian structures,
encapsulated by a \emph{bicomplex Kähler} structure.

In section \ref{sec:BiLagrangianComplexification} we focus on the complexification of a
real-analytic Kähler manifold, and show that it has a canonical complex bi-Lagrangian structure (\autoref{thm:BiLagComplexificationKahler}). 
Recall that any real-analytic manifold $N$ admits a complexification, unique up to germ. When $N$ is a complex manifold, one can define its complexification as $N^c \coloneqq N \times \overline{N}$, where $N$ is diagonally embedded. In particular, there is a canonical pair $(\mathcal{F}_1, \mathcal{F}_2)$ of transverse foliations in $N^c$, which amounts to a bicomplex structure (\autoref{thm:BicomplexInComplexification}).
When $N$ is additionally Kähler, its symplectic structure extends as a complex symplectic structure on $N^c$ which gives a complex bi-Lagrangian structure. This structure can be identified to the bicomplex Kähler structure complexifying the Kähler structure (\autoref{thm:BicomplexKahlerInComplexification}).
After giving a quick application (\autoref{corollary:CKA} and \autoref{corollary:CKB}), we conclude this section by explicitly working out the example of $\CP^1$ 
with the Fubini-Study metric. (NB: \autoref{fig:diagram} below summarize somes of the interrelationships between all the structures discussed so far.)

Delaying for the moment sections \ref{sec:TeichmullerBackground} and \ref{sec:TeichmullerApplications} on Teichmüller theory, in the final section \ref{sec:HyperKahler} we push further 
the study of the complexification
of a Kähler manifold, constructing a natural almost hyper-Hermitian structure. This is the quaternionic analog of a Hermitian structure,
and is called hyper-Kähler when integrable. The Calabi metric \cite{MR543218} in $\upT^* \CP^n$, previously discovered by Eguchi-Hanson \cite{MR540896} for $n=1$, was the first nontrivial example of a hyper-Kähler structure on a noncompact manifold. Feix \cite{FeixThesis} and Kaledin \cite{MR1815021} independently
proved that, more generally, there is a canonical hyper-Kähler structure in the cotangent bundle of any Kähler manifold, though typically only defined
in a neighborhood of the zero section. It can be transported from the cotangent bundle to the complexification, 
although the resulting hyper-Kähler structure has never been properly characterized (\autoref{qu:FKComp}). We were initially hoping to recover
Feix-Kaledin in a more tangible way with our construction, but soon realized that our almost hyper-Hermitian structure is typically not integrable.
On the other hand, it is parallel and satisfies all the other sensible requirements for a hyper-Kähler extension of in the complexification (\autoref{thm:MainThmHK}). Moreover it is part of a \emph{biquaternionic structure} (\autoref{thm:BiQuaternionicHermitian}
and \autoref{cor:ParaHH}).

\bigskip

Let us now turn to the applications to Teichmüller theory, which were the initial motivation of our work.
Here we quote the first paragraph of the Foreword written by Papadopoulos--which we recommend reading \emph{in extenso}--in the Handbook of Teichmüller theory \cite{MR2284826} : 
\begin{quote}
In a broad sense, Teichmüller theory is the study of moduli spaces for geometric structures on surfaces. 
This subject makes important connections between several areas in mathematics, including low-dimensional topology, 
hyperbolic geometry, dynamical systems theory, differential geometry, algebraic topology, representations of discrete groups in Lie groups, symplectic geometry, topological quantum field
theory, string theory, and there are others.
\end{quote}
Complex geometry and (pseudo-)Riemannian geometry can surely be appended to Papadopoulos' list; these features of Teichmüller 
theory are the focus of the present paper.

Let $S$ be a closed oriented surface of negative Euler characteristic. The Teichmüller space of $S$ is the space of isotopy classes
of complex structures on $S$. This space and its complex-analytic structure were originally constructed and studied by Oswald Teichmüller in the 1930s.  
The naturality of its complex structure was fortified by the deformation theory of Kodaira-Spencer \cite{MR0112154} and the algebraic approach of 
Grothendieck \cite{GrothendieckTechniques}. The complex-analytic theory of Teichmüller space was then studied intensely until the 1970s by Lars Ahlfors, Lipman Bers, and others. We refer to \cite{MR3021551} for a historical introduction.

The story with which this paper is concerned properly started when Weil \cite{MR0124485} realized that the Petersson pairing of automorphic forms provides a Hermitian structure on $\TS$ 
now known as the Weil-Petersson metric. Ahlfors \cite{MR0204641} soon proved that this Hermitian metric is Kähler for the natural complex structure on $\TS$, 
unleashing the possibility of a rich symplectic geometry\footnote{Weil already claimed this in \cite{MR0124485}, but in lieu of a proof he wrote that it is ``stupid computation''.}.

In the 1970s and 1980s, Teichmüller theory was deeply influenced 
by the singular geometric vision of William Thurston, and the emphasis partly shifted from complex analysis to hyperbolic geometry. 
It was eventually understood that the symplectic geometry of Teichmüller space is more intrinsic on the deformation space of hyperbolic 
structures on $S$, which we call the Fricke-Klein space $\FS$, and which is in bijection with $\TS$ by the Poincaré uniformization theorem.

The symplectic theory of $\TS$ and $\FS$ was beautifully developed by Wolpert in the 1980s \cite{MR657237, MR690844, MR796909}.
Wolpert showed that given a simple closed curve $\gamma$ on $S$, 
the Hamiltonian flow of the hyperbolic length function $l_\gamma \colon \FS \to \R$ is the flow on $\FS$ that consists in twisting the hyperbolic structure along $\gamma$. 
Wolpert's work shows that given a pants decomposition of $S$, the length functions of the pants curves define an integrable Hamiltonian system 
whose action-angle variables are the famous Fenchel-Nielsen coordinates relative to the pants decomposition.

The full extent of the naturality of the symplectic structure of deformation spaces relative to closed surfaces was established by Goldman \cite{MR762512}, following 
Atiyah-Bott \cite{MR702806}. Goldman showed that there is a natural symplectic structure 
on the character variety $\mathcal{X}(\pi_1(S), G)$ for any semisimple Lie group $G$, where the character variety is the space of closed conjugacy classes of group homomorphisms $\pi_1(S) \to G$. A consequence of Goldman's work is that 
the Fricke-Klein space $\FS$, the quasi-Fuchsian space $\QFS$, and 
the deformation space of complex projective structures $\CPS$ all enjoy natural real or complex symplectic structures, inherited from the character variety for $G = \PSL_2(\C)$.

The complex symplectic geometry of $\CPS$ was carefully studied by Loustau in \cite{MR3352248, MR3323643}.
We recall that a complex projective structure on $S$ is defined by an atlas on $S$ whose charts take values in $\CP^1$, and whose transition functions
are projective linear transformations. Since such an atlas is, in particular, a holomorphic atlas, there is a forgetful projection $p \colon \CPS \to \TS$.
One of the main results of \cite{MR3352248} is that the Schwarzian parametrization of $\CPS$ relative to any Lagrangian section of $p$
provides a symplectomorphism $\CPS \to \upT^* \TS$, generalizing a result of Kawai \cite{MR1386110}. 
Loustau also showed that the symplectic theory of $\FS$
developed by Wolpert holomorphically extends to quasi-Fuchsian space,  extending results of Platis \cite{MR1866841} and Goldman \cite{MR2094117}.

The present paper is in part a continuation of \cite{MR3352248} in that we apply general machinery from symplectic geometry 
to obtain clarifications of results in Teichmüller theory, which up to this point only had proofs relying on idiosyncratic features of the spaces in consideration. 
Furthermore, we discover new geometric features of these deformation spaces.

One simple yet key observation is that quasi-Fuchsian space $\QFS$ is the complexification of Teichmüller space $\TS$.
Without precisely defining $\QFS$ in this introduction, let us just remember that it is a complex manifold and that there is a ``simultaneous uniformization'' biholomorphism
$\TS \times \overline{\TS} \to \QFS$, whose restriction to the diagonal coincides with the uniformization map $\TS \isomap \FS$.
If we recall that $\TS \times \overline{\TS}$ is the canonical complexification of $\TS$, we see that simultaneous uniformization
is ``just'' the complexification of uniformization.
Since $\TS$ is a real-analytic Kähler manifold for the Weil-Petersson metric, its complexification $\QFS$ is a candidate for the general
theory of section \ref{sec:BiLagrangianComplexification}. Thus $\QFS$ admits a complex bi-Lagrangian structure, and we show that the canonical transverse Lagrangian foliations are the well-known foliations of $\QFS$ by Bers slices (\autoref{thm:QFBL}). Other features of this complex bi-Lagrangian structure are yet to be explored, such as its holomorphic Riemannian structure which we partly elucidate in \autoref{thm:QFHoloMetric}.
Note that since the hyper-Hermitian construction of section \ref{sec:HyperKahler} is different from the Feix-Kaledin hyper-Kähler structure, our bi-Lagrangian and hyper-Hermitian picture of quasi-Fuchsian space is not the same as the hyper-Kähler structure beautifully constructed by Donaldson in \cite{MR2039989} and further described in \cite{HodgePHD} and \cite{Trautwein2018}.

Let us conclude this introduction by listing the other applications of section \ref{sec:TeichmullerApplications}:
\begin{itemize}
 \item The Weil-Petersson metric can be defined on Fricke-Klein space $\FS$ via pure complex symplectic geometry, without using the uniformization theorem (\autoref{subsec:WPredef}).
 \item The remarkable affine structure in the fibers of the forgetful projection $p \colon \CPS \to \TS$, which is classically described via the Schwarzian derivative,
 coincides with the Bott affine structure in the leaves of a Lagrangian foliation (\autoref{thm:SchwarzianAffine}).
 \item The family of affine structures on Teichmüller space provided by the Bers embeddings
 coincides with the family of affine structures on a real-analytic Kähler manifold induced by the bi-Lagrangian structure in its complexification (\autoref{thm:TSAffine}).
 \item The derivative of the Bers embedding at the origin is equal to -1/2 times the musical isomorphism induced by the Weil-Petersson metric (\autoref{thm:DerivativeBersEmbedding}).
\end{itemize}

\bigskip

\noindent \emph{A note to the reader:} The segment of the paper dealing with general differential geometry (sections \ref{sec:LagrangianAffine},
\ref{sec:BiLagrangian}, \ref{sec:BiLagrangianComplexification}, \ref{sec:HyperKahler}) and the segment on Teichmüller theory (sections \ref{sec:TeichmullerBackground}, \ref{sec:TeichmullerApplications}) are quite dissimilar in flavor and may 
appeal to different audiences. We tried to make both expositions self-contained while trying to keep a reasonable length;
it is inevitable that parts of the paper will feel too detailed and others too condensed to different readers.

\phantomsection % So that the hyperref link to the Acknowledgments is correct
\subsection*{Acknowledgments}
\addcontentsline{toc}{section}{Acknowledgments}

The authors wish to thank
Charles Boubel,
David Dumas,
Jonah Gaster,
Bill Goldman,
Nigel Hitchin,
John Loftin,
Curt McMullen,
Jean-Marc Schlenker, and
Leon Takhtajan
for valuable conversations and correspondences pertaining to this work.

% The first author thanks his postdoctoral advisor François Labourie for the period 2011-2014 at the Mathematics Department of the University of Paris XI.
% During this period he received funding from the European Research Council under the European Community’s seventh Framework Programme
% (FP7/2007-2013)/ERC grant agreement.
% 
% The second author thanks the University of Illinois at Chicago where he spent 2013–2016 as a postdoctoral fellow; 
% he is very grateful for the freedom and hospitality provided by the Mathematics Department. He also gratefully acknowledges partial support from the National Science Foundation Postdoctoral Research Fellowship 1304006.

Both authors gratefully acknowledge research support from the NSF Grant DMS1107367 \emph{RNMS: GEometric structures And Representation varieties} (the \emph{GEAR Network}).

\section{Lagrangian foliations and affine structures} \label{sec:LagrangianAffine}

In this section, we review basics of Lagrangian foliations, define affine structures and discuss Bott connections, both in the real and in the complex setting. 
Further reliable references for the real setting include \cite{MR0286137}, \cite{MR1038491} and \cite{MR3130569}.

\subsection{Lagrangian foliations}

Let $M$ be a smooth manifold. A \emph{symplectic structure} on $M$ is a $2$-form $\omega\in \Omega^{2}(M, \mathbb{R})$ which is 
closed ($\upd \omega = 0$) and nondegenerate as a bilinear form on $\upT M$.
A \emph{Lagrangian submanifold} is a smooth embedded submanifold $\iota : N \hookrightarrow M$ which is maximally isotropic, \ie{} $\iota^{*}\omega=0$ and  $\dim N = \dim M/2$. A \emph{Lagrangian foliation} of $M$ is a foliation $\mathcal{F}$ of $M$ with Lagrangian leaves.

If $M$ is a complex manifold, a \emph{complex symplectic structure} $\omega$ is a closed nondegenerate $(2,0)$-form.
This definition implies that $\omega$ is holomorphic: it is locally written $\omega = \omega_{ij} \upd z^i \wedge \upd z^j$ 
in a complex chart $(z^i)$ on $M$, where the $\omega_{ij}$ are holomorphic functions. Note that we use the Einstein notation that implies summation over repeated indices.
In this category, a \emph{complex Lagrangian submanifold} is a holomorphically embedded half-dimensional submanifold $\iota : N \hookrightarrow M$ such that $\iota^{*}\omega=0$.
Note that both the real and imaginary parts of $\omega$ are real symplectic structures on $M$, and a complex Lagrangian submanifold
is real Lagrangian for both $\Real(\omega)$ and $\Imag(\omega)$.
A \emph{complex Lagrangian foliation} is a foliation $\mathcal{F}$ of $M$ with complex Lagrangian leaves.

In either the smooth or complex setting, let $p: M \to B$ be a fiber bundle where $(M,\omega)$ is a symplectic manifold.
The fiber bundle $p:M\to B$ is called a \emph{Lagrangian fibration} if for every $b\in B,$ the fiber $p^{-1}(b)$ is Lagrangian in $M$. 
Therefore, the total space of any Lagrangian fibration has a Lagrangian foliation. 
Conversely, given a symplectic manifold with a Lagrangian foliation,
projection to the leaf space yields a Lagrangian fibration in any sufficiently small open set.

\begin{example}[Cotangent bundle] \label{ex:CotangentBundle}
A fundamental example of a symplectic manifold and Lagrangian fibration is the cotangent bundle
of any manifold. Let $N$ be a manifold either in the smooth or complex category. Let $M = \upT ^*N$ denote the total space of the cotangent bundle $p : \upT ^* N \to N$. 
NB: In the complex setting, we let $\upT ^*N$ denote the holomorphic cotangent bundle \ie{} the complex dual of $\upT ^{(1, 0)} N$.
Then $M$ admits a \emph{canonical $1$-form} $\xi$ and a \emph{canonical symplectic structure} $\omega$. These are defined by $\xi_{\alpha} = p^* \alpha$ for any $\alpha \in M$, and $\omega = \upd\xi$. 
Choosing local coordinates $(q^1, \dots, q^n)$ on $N$ defines natural coordinates $(q^1, \dots,  q^n, p_1, \dots, p_n)$ on $M = \upT ^{*}N$, by writing a generic covector $p_i \, \upd q^i \in \upT ^{*}N$. 
In these coordinates, the canonical
$1$-form and symplectic structure are given by $\xi = p_i \, \upd q^i$ and $\omega = \upd p_i \wedge \upd q^i$\footnote{Some authors' definition
of the canonical symplectic structure $\omega$ differs from ours by a minus sign.}.
One quickly checks that the bundle projection $p:\upT ^{*}N\rightarrow N$ is a Lagrangian fibration for the canonical symplectic structure.
\end{example}

Lagrangian foliations satisfy the following structure theorems, \emph{both in the real and in the complex categories}:

\begin{theorem}[Theorem of Darboux-Lie] \label{thm:Darboux-Lie}
Let $(M,\omega)$ be a symplectic manifold with a Lagrangian foliation $\mathcal{F}$. There exists local coordinates
$(q^1, \dots , q^n, p_1, \dots, p_n)$ near any point such that:
\begin{enumerate}[(i)]
\item $(q^1, \dots , q^n, p_1, \dots, p_n)$ are Darboux coordinates for the symplectic structure up to sign:
 \begin{equation}
 \omega = \upd p_i \wedge \upd q^i~.
 \end{equation}
 \item The leaves of the foliations are the level sets of $q = (q_1, \dots, q_n)$.
\end{enumerate}
\end{theorem}

\begin{remark}
The Darboux-Lie theorem implies that, given a Lagrangian submanifold $f: N\rightarrow M$ which is transverse to a Lagrangian foliation $\mathcal{F}$ of $M$, there is a canonical identification 
between a neighborhood of the zero section in $T^{*}N$ and a neighborhood of $f(N)\subset M.$
\end{remark}

\begin{theorem}[Weinstein's symplectic tubular neighborhood theorem \cite{MR0286137}]
Let $(M,\omega)$ be a symplectic manifold with a Lagrangian foliation $\mathcal{F}$ and let $\iota :N \hookrightarrow M$ be a Lagrangian submanifold transverse to $\mathcal{F}$. 
There exists a unique germ of a diffeomorphism $\phi: U\to V$ where $U$ is a neighborhood of $N$ in $M$
and $V$ is a neighborhood of the zero section in $\upT ^* N$ such that:
\begin{enumerate}[(i)]
\item $\phi$ is a fiber-preserving symplectomorphism.
\item $\phi \circ \iota$ is the zero section $N \to \upT^* N$.
\end{enumerate}
\end{theorem}

\subsection{Affine structures}

Let $M$ be a manifold of dimension $n$. Once again, we work either in the real or complex category.

\begin{definition} \label{def:AffineStructure}
An \emph{affine structure} on $M$ is equivalently the data of:
\begin{enumerate}[(i)]
 \item A compatible $(X,G)$-structure on $M$, where $X = \mathbb{A}^n$ is the standard $n$-dimensional affine space and $G = \Aff(\mathbb{A}^n)$
 is the group of affine transformations of $\mathbb{A}^n$.
 \item A flat torsion-free [complex] connection $\nabla$ in the tangent bundle $\upT M$.
\end{enumerate}
\end{definition}

We pause to quickly review the terminology appearing in the previous definition:
\begin{itemize}
 \item The standard $n$-dimensional affine space is $\mathbb{A}^n = k^n$ and its group of affine transformations is 
 $G = \Aff(k^n) \approx \GL_n(k) \ltimes k^n$, where $k = \R$ or $\C$.
 \item A \emph{compatible $(X,G)$-structure} on $M$ is given by
 an atlas of charts with values in $X$ with transition functions given by elements of $G$ acting on $X$.
 For more background on $(X,G)$-structures, see \autoref{subsubsec:GeomStrDef}.
 \item In the complex setting, we require the connection $\nabla$ to be a \emph{complex connection}, 
 meaning that the almost complex structure $I$ on $M$ is parallel: $\nabla I = 0$. Any flat torsion-free complex connection is in fact a \emph{holomorphic connection}, meaning that 
 $\nabla_{X} Y$ is a holomorphic vector field whenever $X$ and $Y$ are holomorphic vector fields.
\end{itemize}

The equivalence of the two definitions above is elementary and left out for brevity.
A third characterization exists in terms of a locally defined free and transitive action of the vector space $k^n$ on $M$, but
we shall not need it.

An $(X, G)$-structure is called \emph{complete} when it is given as the quotient of the model space $X$ by a discrete subgroup of $G$ (provided $X$
is simply connected), see \autoref{subsubsec:GeomStrDef} for details. 
In the case of affine structures, this is equivalent to the connection $\nabla$ being geodesically complete.

\subsection{Bott connection}

\subsubsection{Real Bott connection}

Let us first review Bott connections in the real setting: in what follows, $M$ is a smooth manifold.

We begin with the notion of a partial connection in a vector bundle along a distribution:
\begin{definition}
Let $V \to M$ be a vector bundle on $M$ and $L \subset \upT M$ a distribution on $M$, \ie{} a smooth subbundle of $\upT M$. A \emph{partial (linear) connection} 
in $V$ along $L$ is a map 
\begin{equation}
\begin{split}
\nabla: \Gamma(L) \times \Gamma(V) \to \Gamma(V) \\
(X, s) \mapsto \nabla_X s
\end{split}
\end{equation}
which is ${\mathcal C}^\infty(M, \R)$-linear in $X$ and a derivation in $s$, \ie{} satisfies the Leibniz rule $\nabla_X (f s) = X(f) s + f \nabla_X s$
for all $f \in {\mathcal C}^\infty(M, \R)$.
We use the notation $\Gamma$ above for the space of smooth sections.
\end{definition}

One can then define the Bott connection in the annihilator of an involutive distribution. We recall that a smooth
distribution $L \subset \upT M$ is called \emph{involutive} or \emph{integrable} if the space of sections of $L$ is stable under the Lie bracket. The Frobenius theorem states that a distribution is involutive if and only if it is the tangent space to a smooth foliation.

\begin{definition}
Let $L$ be an involutive distribution on $M$
and $L^\perp \subset \upT ^*M$ be the annihilator of $L$,
\ie{} the subbundle of $T^*M$ defined by $L_x^\perp := \{\alpha \in T_x^*M \colon L_x \subseteq \ker \alpha\}$.
The \emph{Bott connection}
in $L^\perp$ along $L$ is the partial connection defined by $\nabla_X \alpha =  {\mathcal L}_X \alpha$, 
where ${\mathcal L}$ is the Lie derivative.
\end{definition}

It is an elementary exercise to check that this defines a partial connection, and that it is flat:
\begin{proposition}
The Bott connection in $L^\perp$ along $L$ is flat.
\end{proposition}

Now let $\omega$ be a symplectic structure on $M$. Consider a Lagrangian involutive distribution $L$, or equivalently
a Lagrangian foliation ${\cal F}$ of $M$ ($L$ is the tangent distribution to the foliation ${\cal F}$).

\begin{definition}
The \emph{Bott connection in $L$} is the partial connection in $L$ along $L$ obtained from the Bott connection in 
$L^\perp$ along $L$ using the isomorphism $L \approx L^\perp$ given by symplectic duality
\end{definition}

Recall that the musical isomorphism $\flat : \upT M \to \upT ^*M$, $u \mapsto u^\flat$ is defined by $u^\flat = \omega(u, \cdot)$. 
It restricts to an isomorphism $L \to L^\perp$, which we called symplectic duality above.

\begin{proposition} \label{prop:RealBottConnection}
Let $(M, \omega)$ be a real symplectic manifold, let ${\cal F}$ be a Lagrangian foliation of $M$ and denote
by $\nabla$ the Bott connection as above.
\begin{enumerate}[(i)]
	\item $\nabla$ is characterized by the formula 
	\begin{equation}
	 X \cdot \omega(Y, Z)  = \omega(\nabla_X Y, Z) + \omega(Y, [X,Z])
	\end{equation}
	where $X$ and $Y$ are vector fields on $M$ tangent to the foliation and $Z$ is any vector field.
        \item $\nabla$ is flat and torsion-free.
\end{enumerate}
\end{proposition}

The Bott connection may extend to a ``full'' connection in $\upT M$, but there is no canonical way to choose such an extension in general. 
In the next section, we will see that given two transverse Lagrangian foliations, there exists a unique symplectic connection in $M$
that extends both Bott connections (\autoref{prop:BiLagrangianConnection1}).
On the other hand, the Bott connection restricts to a full connection in any leaf, which moreover is flat
and torsion-free by \autoref{prop:RealBottConnection}. It thus defines an affine structure in any leaf of the foliation according to \autoref{def:AffineStructure}:

\begin{theorem}\label{thm:RealLagrangianAffine}
The Bott connection equips every leaf of a Lagrangian foliation with a natural affine structure. 
\end{theorem}

\autoref{thm:RealLagrangianAffine} relates to the Liouville-Arnold theorem. Recall that a Lagrangian foliation is locally
a Lagrangian fibration. Choosing coordinates on the base defines an integrable Hamiltonian system on the total space.
The Liouville-Arnold theorem guarantees the existence of ``action-angle coordinates'' that are canonical for the foliation as in 
the Darboux-Lie \autoref{thm:Darboux-Lie}, and such that there is an $\R^n$-action in the level sets of the action coordinates.  These level sets are the leaves of the foliation, and the $\R^n$-action is an equivalent way to think about the affine structure.

\subsubsection{Complex Bott connection}

Let us now turn to the Bott connection in the complex setting. Let $M$ be a complex manifold. Assume that $M$ is equipped with a complex symplectic structure $\omega$ and that ${\cal F}$ is a complex 
Lagrangian foliation of $M$. Denote by $\omega = \omega_1 + i\omega_2$ the real and imaginary parts of $\omega$. These are real symplectic structures on $M$, and
${\cal F}$ is a Lagrangian foliation with respect to both $\omega_1$ and $\omega_2$.

The \emph{complex Bott connection} in the distribution $L = \upT {\cal F}$ is characterized as follows:

\begin{theorem} \label{thm:ComplexBottConnection}
 There is a unique partial connection $\nabla$ in $L$ along $L$ such that:
 \begin{enumerate}[(i)]
  \item $\nabla$ is a flat and torsion-free partial complex connection.
  \item $\nabla = \nabla^1 = \nabla^2$, where $\nabla^i$ is the Bott connection of $\omega_i$ in $L$ ($i = 1,2$).
 \end{enumerate}
\end{theorem}

The proof of this theorem easily follows from the properties of real Bott connections stated in
\autoref{prop:RealBottConnection}. One derives from \autoref{thm:ComplexBottConnection} the fundamental property:

\begin{theorem} \label{thm:ComplexLagrangianAffine}
The complex Bott connection equips every leaf of a complex Lagrangian foliation with a natural complex affine structure. 
\end{theorem}

\subsubsection{Fundamental example: cotangent bundles}

Let us work either in the smooth or complex setting in what follows.
Recall from \autoref{ex:CotangentBundle} that if $M = \upT ^* N$ is the total space of a cotangent bundle, then $M$ carries a canonical
symplectic structure $\omega$ and the bundle projection $p : M \to N$ is a Lagrangian fibration.
By \autoref{thm:RealLagrangianAffine} or \autoref{thm:ComplexLagrangianAffine},
each fiber is equipped with a natural affine structure. On the other hand, any fiber has an obvious affine structure as a vector space. These two affine structures coincide:

\begin{theorem} \label{thm:BottCotangent}
The Bott affine structure in any fiber of $\upT ^*N$ is the same as its affine structure as a vector space.
\end{theorem}

Since this theorem is key to proving \autoref{thm:SchwarzianAffine}, let us produce a complete proof.

\begin{proof}
We write the proof in the real setting, and the extension to the complex setting is an immediate application of \autoref{thm:ComplexBottConnection}.
By \autoref{prop:RealBottConnection}, The Bott connection ${\nabla}$ is characterized by 
$ \omega(\nabla_X Y, Z) = X \cdot \omega(Y,Z) - \omega(Y, [X, Z])$
for any vector fields $X, Y, Z$ on $M = \upT ^*N$, where $X$ and $Y$ are tangent to the foliation.

Let us work on some open set $\upT ^* U \subset M$ with coordinates $(q^i, p_i)$ as in \autoref{ex:CotangentBundle}. 
The canonical symplectic structure $\omega$ is given by 
$ \omega = \upd p_i \wedge \upd q^i$
% \begin{equation}
%  \omega = \upd p_i \wedge \upd q^i
% \end{equation}
and the vector fields
$X$, $Y$, $Z$ can be written (recall that we use the Einstein notation for writing tensors):
\begin{equation}
 X = a_i \dfrac{\partial}{\partial p_i} \qquad
 Y = b_i \dfrac{\partial}{\partial p_i} \qquad
 Z = c^i \dfrac{\partial}{\partial q^i} + d_i \dfrac{\partial}{\partial p_i}~.
\end{equation}
By direct computation:
\begin{equation} \label{eq:XOYZ}
 X \cdot \omega(Y,Z) - \omega(Y, [X, Z]) = a_i \, c^j \, \frac{\partial b_j}{\partial p_i}~.
\end{equation}
On the other hand, writing
$\nabla_X Y$ in coordinates as $\nabla_X Y = v_j \dfrac{\partial}{\partial p_j}$ gives
\begin{equation} \label{eq:ONBX}
 \omega(\nabla_X Y, Z) = v_j\, c^j~.
\end{equation}
Equating \eqref{eq:XOYZ} and \eqref{eq:ONBX} yields $v_j = a_i \dfrac{\partial b_j}{\partial y^i}$, which gives us the expression
of $\nabla_X Y$ in these coordinates:
\begin{equation}
 \nabla_X Y = a_i \,\dfrac{\partial b_j}{\partial p_i} \, \frac{\partial}{\partial p_j}~.
\end{equation}
This is the expression of the standard covariant derivative in $\R^n$ in the coordinates $(p_1, \dots p_n)$.
\end{proof}

\section{Bi-Lagrangian structures} \label{sec:BiLagrangian}

In this section, we review bi-Lagrangian and para-Kähler structures in the real setting and study properties of the bi-Lagrangian connection.
We then introduce and study complex bi-Lagrangian structures. A good reference for the real setting is \cite{MR2193747}.

\subsection{Real bi-Lagrangian structures}

In what follows, $M$ is a smooth manifold.

\begin{definition}
 A \emph{bi-Lagrangian structure} in $M$ is the data of a symplectic structure
 $\omega$ and an ordered pair $(\mathcal{F}_1, \mathcal{F}_2)$ of transverse Lagrangian foliations.
\end{definition}
Bi-Lagrangian structures have also been called \emph{bipolarizations}, as a Lagrangian foliation is sometimes
called a (real) \emph{polarization}, especially in mathematical physics. 
We shall soon see that a bi-Lagrangian structure is also the same as a \emph{para-Kähler structure} (\autoref{subsec:ParaK}).

A connection $\nabla$ in a symplectic manifold $(M,\omega)$ is called a \emph{symplectic connection} if it is torsion-free and preserves the symplectic structure: $\nabla \omega = 0$.
Symplectic connections always exist but are not unique, contrary to Riemannian connections.
However, the additional data of a bi-Lagrangian structure determines a unique connection $\nabla$ called the \emph{bi-Lagrangian connection} as follows.
\begin{proposition} \label{prop:BiLagrangianConnection1}
 Let $(M,\omega, \mathcal{F}_1, \mathcal{F}_2)$ be a bi-Lagrangian manifold. There exists a unique symplectic connection $\nabla$ on $M$ which
 extends both the Bott connections in $\mathcal{F}_1$ and in $\mathcal{F}_2$.
\end{proposition}

One can show that the bi-Lagrangian connection $\nabla$ is in fact the unique symplectic connection
in $M$ which satisfies the apparently weaker condition that it preserves both foliations, see \autoref{thm:BiLagrangianConnection}.

\subsection{Para-complex and para-Kähler structures} \label{subsec:ParaK}

The algebra of \emph{para-complex numbers}\footnote{Para-complex numbers are also sometimes called \emph{split-complex numbers}, \emph{hyperbolic numbers}
and a variety of other names that are listed in \cite{wiki:paracomplex}.} is the commutative algebra
\begin{equation}
\R[X]/(X^2-1) = \{a + fb : (a,b)\in\R^2, f^2 = 1\}~.
\end{equation}
Para-complex structures on manifolds are the para-complex analog of complex structures. We refer to \cite{MR1386154} for a survey of para-complex geometry.

\begin{definition} \label{def:ParaComplexStructure}
 Let $M$ be a smooth manifold. An \emph{almost para-complex structure} on $M$ is a smooth field of endomorphisms $F \in \End(\upT M)$
 such that $F^2 = \mathbf{1}$ and $\tr(F) = 0$.
\end{definition}

Call $L_1$ and $L_2$ the $+1$ and $-1$ eigendistributions of $F$, so that $\upT M = L_1 \oplus L_2$. 
The condition $\tr(F) = 0$ amounts to saying that $\dim L_1 = \dim L_2$. The almost para-complex structure $F$ is called \emph{integrable}
if $L_1$ and $L_2$ are involutive, in other words there are two transverse foliations $\mathcal{F}_1$ and $\mathcal{F}_2$ of $M$
such that $L_i$ is the tangent subbundle to $\mathcal{F}_i$. When $F$ is integrable, it is just called a \emph{para-complex structure}.
Clearly, any ordered pair of equidimensional transverse distributions $(L_1,L_2)$ conversely determines a unique almost para-complex 
structure $F$ on $M$.

Next we define para-Kähler structures, which are the para-complex analog of Kähler structures (compare with
\autoref{def:Kahler}).
\begin{definition}
 A \emph{para-Kähler} structure in a smooth manifold $M$ is the data of $(g, F, \omega)$, where:
 \begin{itemize}
  \item $g$ is a pseudo-Riemannian metric in $M$,
  \item $F$ is an (almost) para-complex structure in $M$,
  \item $\omega$ is an (almost) symplectic structure in $M$,
 \end{itemize}
 such that:
 \begin{enumerate}[(i)]
  \item $\omega(u,v) = g(F u, v)$ for any $u,v$ (\emph{compatibility condition})
  \item $F$ is parallel for the Levi-Civita connection of $g$: $\nabla F = 0$ (\emph{integrability condition}).
 \end{enumerate}
\end{definition}
It is easy to see that the signature of $g$ must be $(n,n)$ where $\dim M = 2n$, such a pseudo-Riemannian metric is called a \emph{neutral metric}.
As in the complex case, the integrability condition is equivalent to the simultaneous integrability of $F$ as a para-complex structure and closedness of $\omega$ as a $2$-form, 
so that a para-Kähler manifold is para-complex and symplectic manifold in addition to a pseudo-Riemannian manifold, and these three structures are mutually compatible.

If $(g, F, \omega)$ is a para-Kähler structure on $M$, then one quickly checks that the eigendistributions $L_1$ and $L_2$ of $F$ are isotropic for $g$ and $\omega$.
In particular, the two transverse foliations $\mathcal{F}_1$ and $\mathcal{F}_2$ defined by $F$ are Lagrangian for $\omega$ and thus define a bi-Lagrangian structure on $M$.
Conversely, it is clear that any bi-Lagrangian structure on $M$ determines a unique para-Kähler structure $(g, F, \omega)$.
\begin{proposition}
There is a 1-1 correspondence between bi-Lagrangian structures $(\omega, \mathcal{F}_1, \mathcal{F}_2)$ on $M$
and para-Kähler structures $(g, F, \omega)$ on $M$, where the tangent distributions to $\mathcal{F}_1$ and $\mathcal{F}_2$ 
are respectively the $+1$ and $-1$ eigendistributions of the para-complex structure $F$.
\end{proposition}

\subsection{Properties of the bi-Lagrangian connection} \label{subsec:PropertiesOfBLConnection}

Let $(M, \omega, \mathcal{F}_1, \mathcal{F}_2)$ be a bi-Lagrangian manifold. We call the neutral metric $g$ of the associated para-Kähler structure
the \emph{bi-Lagrangian metric}. Recall that there is a unique symplectic connection in $M$ extending both Bott connections in $\mathcal{F}_i$ called the 
\emph{bi-Lagrangian connection} (\autoref{prop:BiLagrangianConnection1}). The next theorem ensures that it can be alternatively 
defined as the Levi-Civita connection of $g$.

\begin{theorem} \label{thm:BiLagrangianConnection}
 Let $(M, \omega, \mathcal{F}_1, \mathcal{F}_2)$ be a bi-Lagrangian manifold and denote by $(g, F, \omega)$
 the associated para-Kähler structure. The bi-Lagrangian connection of $M$ is the unique torsion-free connection $\nabla$ satisfying the following equivalent conditions:
 \begin{enumerate}[(i)]
  \item $\nabla$ parallelizes $g$: $\nabla g = 0$.
  \item $\nabla$ parallelizes $\omega$ and $F$: $\nabla \omega = 0$ and $\nabla F = 0$.
  \item $\nabla$ parallelizes $\omega$ and preserves both foliations $\mathcal{F}_1$ and $\mathcal{F}_2$.
 \end{enumerate}
\end{theorem}

Knowing that the bi-Lagrangian connection $\nabla$ is the Levi-Civita connection of the bi-Lagrangian metric $g$ allows one to compute it explicitly. 
Let $\nabla^{B_i}$ denote the Bott connection in $\mathcal{F}_i$ ($i \in \{1,2\}$). Recall that we write  any vector
$u \in \upT M$ as $u = u_1 + u_2$ where $u_i$ is tangent to the foliation $\mathcal{F}_i$. Then for any vector fields $X$ and $Y$:
\begin{equation} \label{eq:BiLagrangianConnection}
 \nabla_{X} Y = {\nabla_{X_1}^{B_1}} Y_1 + {\nabla_{X_2}^{B_2}} Y_2 + [X_1, Y_2]_2 + [X_2, Y_1]_1~.
\end{equation}

One can see from \eqref{eq:BiLagrangianConnection} that both foliations are totally geodesic. Note that they are also flat since the Bott connections $\nabla^{B_i}$ are flat in restriction to the leaves. Next we compute the curvature tensor $R$ of the bi-Lagrangian connection, which is straightforward using \eqref{eq:BiLagrangianConnection}:
\begin{proposition} \label{prop:BLCurvature}
 Let $X$, $Y$, and $Z$ be vector fields on $M$.
 \begin{itemize}
  \item If $X = X_i$ and $Y = Y_i$ are both tangent to $\mathcal{F}_i$ where $i \in \{1,2\}$, then $R(X,Y) = 0$.
  \item If $X = X_1$ is tangent to $\mathcal{F}_1$ and $Y = Y_2$ is tangent to $\mathcal{F}_2$, then:
    \begin{equation}
    \begin{split}
   R(X,Y)Z = {} & \nabla_X [Y, Z_1]_1 - [[X,Y]_2, Z_1]_1 - \nabla_{[X,Y]_1} Z_1 - [Y, \nabla_X Z_1]_1\\
           &  -\nabla_Y [X, Z_2]_2 - [[X,Y]_1, Z_2]_2 - \nabla_{[X,Y]_2} Z_2 - [X, \nabla_Y Z_2]_2~.
    \end{split}
  \end{equation}
 \end{itemize}
\end{proposition}

By linearity and antisymmetry of the curvature tensor $R$ in $X$ and $Y$, \autoref{prop:BLCurvature} determines $R$ completely.
Note that $R(X,Y)$ always preserves both distributions in the sense
that $R(X,Y)Z$ is tangent to $\mathcal{F}_i$ whenever $Z$ is tangent to $\mathcal{F}_i$. Also observe that while we established that the restriction of 
$\nabla$ to a leaf of either foliations is flat, the fact that $R(X,Y) = 0$ whenever $X$ and $Y$ are tangent to one foliation is somewhat stronger. It implies 
that parallel transport with respect to $\nabla$ of any tensor on $M$
along a path contained in a leaf is independent of the path. We shall make use of that property in the construction of \autoref{thm:MainThmHK}, so let us record this 
precisely:  

\begin{corollary} \label{cor:PathIndependence}
 Let $(M, \omega, \mathcal{F}_1, \mathcal{F}_2)$ be a bi-Lagrangian manifold and denote by $\nabla$ the bi-Lagrangian connection. 
 Let $\gamma : [0,1] \to M$ be a smooth path
 contained in a leaf of either foliations $\mathcal{F}_i$ and let $T$ be a tensor on $M$ at $\gamma(0)$. Then the parallel transport of $T$ along $\gamma$ with respect
 to $\nabla$ does not depend on the choice of $\gamma$ in its homotopy class in the leaf rel.~endpoints.
\end{corollary}

The proof of \autoref{cor:PathIndependence} can be adapted  effortlessly from a standard proof
that the vanishing of the curvature of a connection is equivalent to the local path-independence of parallel transport, see \eg{}
\cite[Theorem 11.1]{VoronovLecture}.

\subsection{Complex bi-Lagrangian structures}

We now turn to complex bi-Lagrangian structures, which are the natural extension of real bi-Lagrangian structures in the complex setting.

Let $M$ be a complex manifold. We let $J$ denotes the (integrable) almost complex structure of $M$. 
We recall that an almost complex structure on a smooth manifold $M$ is a tensor field $J \in \End(\upT M)$ such that
$J^2 = -\mathbf{1}$. There is a notion of integrability for almost complex structures in terms of the vanishing of their so-called \emph{Nijenhuis tensor};
the Newlander-Nirenberg theorem says that $J$ is integrable if and only if $M$ can be given the structure of a complex manifold inducing $J$.

\begin{definition}
 A \emph{complex bi-Lagrangian structure} in $M$ is the data of a complex symplectic structure
 $\omega$ and an ordered pair $(\mathcal{F}_1, \mathcal{F}_2)$ of transverse complex Lagrangian foliations.
\end{definition}
Similarly to the real case, the data of a complex bi-Lagrangian structure in $M$ is equivalent to the data of 
a \emph{bicomplex para-Kähler structure} $(g, F, \omega)$, but before we define such structures, we need to review basic facts about holomorphic metrics.

\subsection{Holomorphic metrics} \label{subsec:HolomorphicMetrics}

A \emph{holomorphic (Riemannian) metric}
$g$ on $M$ is a holomorphic symmetric complex bilinear form on $M$ (\ie{} a holomorphic section of the symmetric product $S^2 \, \upT ^*M$) which is nondegenerate. 
The next proposition is elementary.

\begin{proposition}
Let $g$ be a holomorphic Riemannian metric on $M$, denote by $g = g_1 + i g_2$ its real and imaginary parts.
\begin{enumerate}[(i)]
 \item $g_1$ and $g_2$ are neutral metrics on $M$.
 \item $g_1(u,v) = g_2(J u, v)$ for any tangent vectors $u$, $v$ at a same point of $M$.
\end{enumerate}
\end{proposition}
% We recall that a neutral metric on $M$ is a pseudo-Riemannian metric of signature $(n,n)$, where $\dim_\R M = 2n$.

The ``fundamental theorem of Riemannian geometry'' holds for holomorphic metrics:
\begin{theorem} \label{thm:HolomorphicMetrics}
 Let $g$ be a holomorphic metric on $M$. There exists a unique torsion-free holomorphic connection
 $\nabla$ which parallelizes $g$, called its \emph{holomorphic Levi-Civita connection}.
 Moreover, $\nabla$ is the Levi-Civita connection of both the real and imaginary parts of $g$.
\end{theorem}

\begin{proof}
The proof of \autoref{thm:HolomorphicMetrics} is the same as the usual proof for Riemannian metrics: the fact that $\nabla$ is torsion-free and parallelizes $g$ implies that it
verifies the \emph{Koszul formula}:
\begin{equation}
\begin{split}
 2\,g(\nabla_X Y, Z) &=  X \cdot g(Y,Z) + Y \cdot  g(Z,X) - Z \cdot g(X,Y) \\
 &+ g([X,Y],Z) - g([Y,Z], X) - g([X,Z], Y) ~.
\end{split}
\end{equation}
This formula gives existence and uniqueness. The fact that $\nabla$ is the Levi-Civita connection of $g_1$ and $g_2$ is derived
by taking the real and imaginary parts of the formula. Since $g_1( \cdot , \cdot ) = g_2(J \cdot, \cdot)$,
the fact that $g_1$ and $g_2$ are parallel implies that $J$ is parallel, so that $\nabla$ is a complex connection.
Finally, one derives from the Koszul formula and from $g$ being holomorphic that $\nabla$ is a holomorphic.
\end{proof}

\subsection{Bicomplex Kähler structures} \label{subsec:Bicomplex}

The algebra of \emph{bicomplex numbers} $\mathbb{BC}$ is the unital associative algebra over the real numbers generated by three elements $i$, $j$, and $f$ satisfying
the \emph{bicomplex relations}:
\begin{equation}
\begin{gathered}
 i^2 = -1 \quad j^2 = -1 \quad f^2 = +1\\
 ij = ji = f~.
\end{gathered}
\end{equation}
The algebra of bicomplex numbers $\mathbb{BC}$ is a $4$-dimensional algebra over $\R$: a generic bicomplex number is written $q = a + ib + jc + kd$ with $(a,b,c,d) \in \R^4$.
One quickly sees that $\mathbb{BC}$ can be simply be described as $\C \otimes_\R \C$ by writing $q = (a + bi) + (c + di)j$.
We refer to \cite{MR3410909} for the reader interested to learn more about bicomplex numbers.

\begin{definition}
An \emph{bicomplex structure} on a smooth manifold $M$ is the data of $(I, J, K)$ where:
\begin{itemize}
 \item $I$ and $J$ are integrable almost complex structures.
 \item $F$ is an integrable para-complex structure.
 \item $I$, $J$, $F$ satisfy the bicomplex relations as above.
\end{itemize}
\end{definition}

Of course, the data of only two of the three structures $I$, $J$, $F$ is enough to determine the third one via the relation $IJ = F$. 
This allows us to give the following equivalent definitions
of a \emph{holomorphic} bicomplex structure on $M$, when the complex structure on $M$ corresponding to $J$ is already given:
\begin{definition}
Let $(M,J)$ be a complex manifold. A \emph{holomorphic bicomplex structure} 
on $M$ is equivalently the data of:
\begin{enumerate}[(i)]
 \item An integrable almost complex structure $I$ which is complex linear 
 as an endomorphism of $\upT M$ and holomorphic as a tensor field on $M$, and such that $\tr(IJ) = 0$.
 \item An integrable almost para-complex structure $F$ which is complex linear 
 as an endomorphism of $\upT M$ and holomorphic as a tensor field on $M$.
 \item An ordered pair $(\mathcal{F}_1,\mathcal{F}_2)$ of transverse holomorphic foliations of $M$ by half-dimensional complex submanifolds.
\end{enumerate}
\end{definition}
This definition is seemingly weaker than another natural definition, namely the data of an atlas with bicomplex-holomorphic transition functions 
(as in \eg{} \cite{MR2746980}), but they turn out to be equivalent. This follows from the fact that a holomorphic bicomplex manifold is locally biholomorphic
to a product of complex manifolds, which is the complex declination of the fact that an integrable para-complex structure induces a local product structure
(see \eg{} \cite{MR0066020} for details).

We are now ready to define bicomplex Kähler structures:
\begin{definition}
Let $M$ be a complex manifold. A \emph{(holomorphic) bicomplex Kähler structure} on $M$ is the data of $(g, F, \omega)$, where:
 \begin{itemize}
  \item $g$ is a holomorphic metric in $M$,
  \item $F$ is a para-complex structure in $M$ defining a (holomorphic) bicomplex structure, and
  \item $\omega$ is a complex symplectic structure in $M$
 \end{itemize}
 such that:
 \begin{enumerate}[(i)]
  \item $\omega(u,v) = g(F u, v)$ for any $u,v$ (\emph{compatibility condition})
  \item $F$ is parallel for the Levi-Civita connection of $g$: $\nabla F = 0$ (\emph{integrability condition}).
 \end{enumerate}
\end{definition}

As in the real setting, it is straightforward to show the following proposition.
\begin{proposition} \label{prop:BLBKC}
There is a 1-1 correspondence between complex bi-Lagrangian structures $(\omega, \mathcal{F}_1, \mathcal{F}_2)$ on $M$
and holomorphic bicomplex Kähler structures $(g, F, \omega)$ on $M$, where the tangent distributions to $\mathcal{F}_1$ and $\mathcal{F}_2$ 
are the $+1$ and $-1$ eigendistributions of the para-complex structure $F$.
\end{proposition}

\subsection{Complex bi-Lagrangian metric and connection} \label{subsec:ComplexBiLagrangianMetricAndConnection}

Let $M$ be a complex manifold equipped with a complex bi-Lagrangian structure $(\omega, \mathcal{F}_1, \mathcal{F}_2)$ and denote by $(g, F, \omega)$
the holomorphic bicomplex Kähler structure (\autoref{prop:BLBKC}). We call $g$ the \emph{complex bi-Lagrangian metric}
and its Levi-Civita connection $\nabla$ the \emph{complex bi-Lagrangian connection}.
\begin{theorem} \label{thm:ComplexBiLagConnection}
The complex bi-Lagrangian connection $\nabla$ is the unique torsion-free holomorphic connection in $M$ which satisfies the equivalent conditions:
\begin{enumerate}[(i)]
  \item $\nabla$ parallelizes $g$.
  \item $\nabla$ parallelizes $\omega$ and preserves both foliations $\mathcal{F}_1$ and $\mathcal{F}_2$.
  \item $\nabla$ parallelizes $g$ and the bicomplex structure $(I, J, F)$.
\end{enumerate}
 Moreover, $\nabla$ extends both the complex Bott connections in $\mathcal{F}_1$ and in $\mathcal{F}_2$.
\end{theorem}

Let us write $\omega$ in terms of its real and imaginary parts: $\omega = \omega_1 + i\omega_2$.
Note that $(\omega_1, \mathcal{F}_1, \mathcal{F}_2)$ and $(\omega_2, \mathcal{F}_1, \mathcal{F}_2)$ are both real
bi-Lagrangian structures in $M$, let us call them the real and imaginary parts of the complex bi-Lagrangian structure.
\begin{theorem}\label{thm:ComplexAndRealBL}
The bi-Lagrangian metrics for the real and imaginary parts of the complex bi-Lagrangian structure are respectively the real and imaginary parts of the complex bi-Lagrangian metric.
Moreover, the complex bi-Lagrangian connection is equal to the real bi-Lagrangian connection for both the real and imaginary parts of the complex bi-Lagrangian structure.
\end{theorem}

The proofs of \autoref{thm:ComplexBiLagConnection} and \autoref{thm:ComplexAndRealBL} are easily reduced to a combination of \autoref{thm:ComplexBottConnection}, \autoref{thm:BiLagrangianConnection}, 
\autoref{thm:HolomorphicMetrics}, and \autoref{prop:BLBKC}.

\section{Bi-Lagrangian structure in the complexification of a Kähler manifold} \label{sec:BiLagrangianComplexification}

In this section, we show that the complexification of a real-analytic Kähler manifold enjoys a natural complex bi-Lagrangian structure
and study some of its properties.

\subsection{Complexification of real-analytic manifolds} \label{subsec:Complexification}

Let us recall the essentials of complexification of real-analytic manifolds.
\begin{definition}
Let $M$ be a complex manifold and denote by $J$ its almost complex structure. A real-analytically embedded submanifold $\iota \colon N \hookrightarrow M$
is called \emph{maximal totally real} if for all $p\in N$:
\begin{equation}
\upT_p N \oplus J(\upT_p N) = \upT_p M~.
\end{equation}
In this case the embedding $\iota \colon N \hookrightarrow M$ is called a \emph{complexification of $N$}.  
\end{definition}

One can characterize complexifications of real-analytic manifolds in terms of fixed points of anti-holomorphic involutions:
\begin{proposition} \label{prop:AntiholoInvolution}
Let $N$ be a real-analytic submanifold of a complex manifold $M$. Then $M$ is a complexification of $N$ if and only if there exists an anti-holomorphic involution 
$\tau : U \to U$, where $U$ is a neighborhood of $N$ in $M$, such that $N$ is the set of fixed points of $\tau$.
\end{proposition}

The following theorem guarantees existence of a complexification of any real-analytic manifold 
and uniqueness up to restriction, so that the \emph{germ} of a complexification is unique:
\begin{theorem} \label{thm:Complexification}
Let $N$ be a real-analytic manifold.
\begin{enumerate}[(i)]
 \item There exists a complexification $\iota \colon N \rightarrow M$.
 \item Let $\iota_1 \colon : N \to M_1$ and $\iota_2 \colon N \to M_2$ be two complexifications. There exists a unique germ of a biholomorphism 
$\phi:U_1\rightarrow U_2$
where $U_i$ is a connected neighborhood of $\iota_i(N)$ in $M_i$, such that $\iota_2 = \phi \circ \iota_1$.
\end{enumerate}
\end{theorem}

A fundamental fact about complexification of a real-analytic manifold is that any analytic tensor field extends uniquely to (a germ of) a holomorphic tensor field in the complexification:
\begin{proposition} \label{prop:HolomorphicExtension}
 Let $N$ be a real-analytic manifold and let $\iota \colon N \to M$ be a complexification. Let $T$ be a real-analytic tensor field on $N$, then there exists a unique germ
 of a holomorphic tensor field $T^c$ in a neighborhood of $N$ in $M$ such that $\iota^* T^c = T$.
\end{proposition}
The proof of this proposition, in local coordinates, boils down to standard analytic continuation using power series. 
Let us clarify that by \emph{holomorphic tensor field}, we mean a holomorphic section of the tensor product of a finite number of copies of the holomorphic tangent bundle and its dual. 
Here are a few examples of this phenomenon:
\begin{itemize}
 \item Any real-analytic function on $N$ locally extends to a holomorphic function on $M$.
 \item Any real-analytic symplectic structure on $N$ locally extends to a complex symplectic structure.
 \item Any real-analytic Riemannian metric on $N$ locally extends to a holomorphic metric on $M$.
\end{itemize}

\subsection{Complexification of complex manifolds} \label{subsec:ComplexificationComplex}

Given a real-analytic manifold $N$, even though there exists an essentially unique complexification $M$, there is no canonical model for $M$ in general.
However, if $N$ happens to be a complex manifold then such a canonical complexification exists:

\begin{proposition} \label{prop:CanonicalComplexification}
 Let $(N, I_0)$ be a complex manifold. Let $N^c := N \times N$; equip $N^c$ with the integrable almost complex structure
 $J := I_{0} \oplus -I_{0}$ (in other words, $N^c = N \times \overline{N}$ according to a standard notation). Then the diagonal embedding $f : N \to N^c$ is a complexification  of $N$.
\end{proposition}

Let us call $N^c := N \times \overline{N}$ the \emph{canonical complexification} of $N$. 

Let $(z^1, \dots z^n)$ be local holomorphic coordinates in $N$,
denote $w^i = \overline{z^i}$ the conjugates in a copy of $N$. Then $(z^i, w^i)$ are local holomorphic
coordinates in $N^c$. The complexification map $f : N \to N^c$ is given by $f(z^1, \dots, z^n) = (z^1, \dots, z^n, \overline{z^1}, \dots, \overline{z^n})$.

Note that in this situation, the anti-holomorphic involution $\tau$ of \autoref{prop:AntiholoInvolution} is simply the map $\tau(x, y) = (y,x)$, defined everywhere
in $N^c = N \times \overline{N}$. In the coordinates $(z,w)$ above, it is given by $\tau(z,w) = (\overline{w}, \overline{z})$.

Remarkably, when $N$ is a complex manifold, any complexification $M$ admits two natural transverse holomorphic foliations $\mathcal{F}_1$
and $\mathcal{F}_2$ by half-dimensional complex submanifolds, at least in a neighborhood of $N$ in $M$.
When $M = N^c$ is given as the canonical complexification, these foliations are simply the vertical and horizontal foliations of the product  $N^c = N \times \overline{N}$.
In the coordinates $(z^i, w^i)$ above, the vertical foliation is given by $\{z = z_0\}$ where $z_0 \in \C^n$ is a constant, similarly the horizontal foliation is
given by $\{w = w_0\}$. Let us record this:
\begin{definition} \label{def:CanonicalFoliations}
 Let $N$ be a complex manifold and let $M$ be a complexification. The ordered pair $(\mathcal{F}_1, \mathcal{F}_2)$ of transverse holomorphic foliations of (a neighborhood of $N$ in) $M$
 described above will be called the \emph{canonical pair of foliations} of $M$.
\end{definition}
We insist that these foliations are not defined in any complexification of $N$ when $N$ is not equipped with a complex structure to begin with, 
in other words they depend on the complex structure on $N$.

\subsection{Complexification of Kähler manifolds}

We recall the definition of a Kähler manifold in order to fix notations and conventions:
\begin{definition} \label{def:Kahler}
 A \emph{Kähler} structure on a smooth manifold $N$ is the data of $(g, I, \omega)$, where:
 \begin{itemize}
  \item $g$ is a Riemannian metric in $N$,
  \item $I$ is an (integrable) almost complex structure, and
  \item $\omega$ is an (almost) symplectic structure in $N$
 \end{itemize}
 such that:
 \begin{enumerate}[(i)]
  \item $\omega(u, v) = g(Iu, v)$ for any tangent vectors $u, v$ (\emph{compatibility condition})
  \item $I$ is parallel for the Levi-Civita connection of $g$: $\nabla I = 0$ (\emph{integrability condition}).
 \end{enumerate}
\end{definition}
The integrability condition amounts to the simultaneous integrability of $I$ as a complex structure and closedness of $\omega$ as a $2$-form, so that a Kähler manifold is a complex manifold
and a symplectic manifold in addition to a Riemannian manifold, and the three structures are mutually compatible.

Now let $(N, g_0, I_0, \omega_0)$ be a \emph{real-analytic Kähler manifold}, \ie{} a Kähler manifold such that $g_0$ (and $I_0$ and $\omega_0$, automatically) are
real-analytic tensor fields for the real-analytic structure underlying the complex-analytic structure. Let $N \hookrightarrow M$ be a complexification of $N$. Denote by $J$ the (integrable) almost complex structure on $M$.
By \autoref{prop:HolomorphicExtension}, in a neighborhood $U$ of $N$ in $M$, the tensor fields $g_0$, $I_0$, and $\omega_0$ admit unique holomorphic extensions, namely:
\begin{itemize}
 \item The Riemannian metric $g_0$ extends to a holomorphic metric $g_0^c$ (see \autoref{subsec:HolomorphicMetrics} for the definition).
 \item The almost complex structure $I_0$ extends holomorphically to a complex endomorphism $I_0^c$ of $\upT M$ which squares to $-1.$
 \item The symplectic form $\omega_0$ extends to a complex symplectic form $\omega_0^c$.
\end{itemize}
In addition, $M$ enjoys two transverse holomorphic foliations $\mathcal{F}_1$ and $\mathcal{F}_2$ as we saw in
\autoref{def:CanonicalFoliations}. The next theorem shows that all this structure is encapsulated by a bi-Lagrangian structure.
\begin{theorem} \label{thm:BiLagComplexificationKahler}
 Let $(N, g_0,  I_0, \omega_0)$ be a real-analytic Kähler manifold. Let $M$ be a complexification of $N$.
 Then $M$ has a canonical complex bi-Lagrangian structure in a neighborhood of $N$. More precisely, in a sufficiently small connected open neighborhood of $N$
 in $M$, there exists a unique complex bi-Lagrangian structure $(\omega, \mathcal{F}_1, \mathcal{F}_2)$ such that :
 \begin{enumerate}[(i)]
  \item \label{theoremCKBLi} The complex symplectic structure $\omega$ is the holomorphic extension of $\omega_0$: $\omega = \omega_0^c$.
  \item \label{theoremCKBLii} The transverse complex Lagrangian foliations $\mathcal{F}_1$ and $\mathcal{F}_2$ are the canonical foliations of $M$ defined in \autoref{def:CanonicalFoliations}.
  \item \label{theoremCKBLiii} The complex bi-Lagrangian metric $g$ (see \autoref{subsec:ComplexBiLagrangianMetricAndConnection}) is  $-i$ times the holomorphic extension of $g_0$:
  \begin{equation} \label{eq:gBLgc}
  g = -i \, g_0^c~.
  \end{equation}
\end{enumerate}
\end{theorem}

\begin{proof}
It is clear that we can take a neighborhood $U$ of $N$ in $M$ small enough so that $g_0^c$, $I_0^c$, and $\omega_0^c$, as well as $\mathcal{F}_1$ and $\mathcal{F}_2$, are all well-defined
(and uniquely defined) in $U$. In order to show that $(\omega, \mathcal{F}_1, \mathcal{F}_2)$ is a complex bi-Lagrangian structure, all that is left to prove
is that $\mathcal{F}_1$ and $\mathcal{F}_2$ are isotropic for $\omega = \omega_0^c$. Let $(z^i)$ be a system of local holomorphic 
coordinates on $N$. Since the Kähler form $\omega_0$ is of type $(1,1)$, it can be written:
\begin{equation}
 \omega_0 = -\frac{1}{2i}\,h_{j\bar{k}}\,\dz^j \wedge \overline{\dz^k}
\end{equation}
where $h_{j\bar{k}}$ are real-analytic complex-valued functions on $N$ which satisfy $h_{k\bar{j}} = \overline{h_{j\bar{k}}}$.
Each of these functions admits a unique holomorphic extension $h_{j\bar{k}}^c$ (provided $U$ is small enough), obtained by holomorphically extending
both the real and imaginary parts of $h_{j\bar{k}}$. Denote by $(z^i, w^i)$ the local holomorphic coordinates
in $M$ as in \autoref{subsec:ComplexificationComplex}. The holomorphic extension $\omega_0^c$ of $\omega_0$ is simply given by
\begin{equation}
 \omega_0^c = -\frac{1}{2i}\,h_{j\bar{k}}^c\,\dz^j \wedge \dw^k \eqqcolon \omega \,.
\end{equation}
Indeed, by uniqueness of the holomorphic extension, it is enough to check that:
\begin{enumerate}[(a)]
 \item $\omega$ is a complex symplectic structure. This is clear, since $(z^1, \dots, z^n, w^1, \dots, w^n)$ are holomorphic coordinates
 and $h_{j\bar{k}}^c$ are holomorphic functions.
 \item $\omega$ restricts to $\omega_0$ on $N$. This is also clear, since $w^i = \overline{z^i}$ and $h_{j\bar{k}}^c = h_{j\bar{k}}$ at points of $N \subset M$.
\end{enumerate}
Now recall that the canonical foliations $\mathcal{F}_1$ and $\mathcal{F}_2$ of $M$ are defined by $\{z = \mathit{constant}\}$ and $\{w = \mathit{constant}\}$ respectively.
With the expression of $\omega_0^c$ above, it is immediate that both these foliations are isotropic for $\omega$. This concludes the proof of \ref{theoremCKBLi} and \ref{theoremCKBLii}.
We delay proving \ref{theoremCKBLiii} until the proof of \autoref{thm:BicomplexKahlerInComplexification}.
\end{proof}

Observe in particular that the canonical foliations $\mathcal{F}_1$ and $\mathcal{F}_2$ are not only isotropic for the holomorphic metric, but also totally geodesic and flat
(see \autoref{subsec:PropertiesOfBLConnection}). Let us give the following corollary as a simple example of consequence.

\begin{corollary} \label{corollary:CKA}
Let $(N, g_0, I_0, \omega_0)$ be a real-analytic Kähler manifold. Assume that the holomorphic extension $\omega = \omega_0^c$ exists everywhere in the canonical 
complexification $N^c = N \times \overline{N}$. Then $N$ admits a natural family of complex affine structures, parametrized by the points of $N$.
\end{corollary}

This is just a consequence of the fact that $N$ can be identified to any horizontal leaf in the product $N \times \overline{N}$, and the leaf space is parametrized by $N$. 
Note that these affine structures have no reason to be complete in general.

\begin{remark}
The existence of a flat connection in a compact complex manifold implies that all Chern classes of the tangent bundle vanish, 
in particular the Euler characteristic vanishes as it identifies with the top Chern class. The next corollary follows.
\end{remark}

\begin{corollary} \label{corollary:CKB}
 Let $N$ be a compact complex manifold which admits a Kähler form whose holomorphic extension exists everywhere in the canonical complexification $N^c = N \times \overline{N}$.
 Then $N$ has zero Euler characteristic.
\end{corollary}

\begin{remark}
Recently, paracomplex geometry has been applied to prove a famous conjecture of Chern on the vanishing Euler characteristic of closed affine manifolds with parallel volume
\cite{MR3665000}.
\end{remark}

\subsection{The bicomplex Kähler structure} \label{subsec:BicomplexKahlerInComplexification}

Let us know show how the bicomplex and bicomplex Kähler structures introduced in \autoref{subsec:Bicomplex} provide clarity
on the complexification of complex and Kähler manifolds.

\subsubsection{Bicomplex structure in the complexification of a complex manifold}

Let $(N,I_0)$ be a complex manifold and $(M, J)$ a complexification of $N$. Denote $I := I_0^c$ the holomorphic extension of $I_0$ (cf \autoref{prop:HolomorphicExtension}).
\begin{theorem} \label{thm:BicomplexInComplexification}
 The triple $(I, J, F:=IJ)$ is a holomorphic bicomplex structure in a neighborhood of $N$ in $M$. Moreover:
 \begin{enumerate}[(i)]
  \item $F$ is the integrable para-complex complex structure associated to the canonical pair of foliations $(\mathcal{F}_1, \mathcal{F}_2)$ of \autoref{def:CanonicalFoliations}.
  \item This holomorphic bicomplex structure is defined everywhere in the canonical complexification $M = N^c = N \times \overline{N}$.
 \end{enumerate}
\end{theorem}

\begin{proof}
The fact that $I$ and $J$ commute and that the $\pm1$-eigendistributions of $F = IJ$ are the vertical and horizontal foliations of $M$ is a straightforward calculation in the local
coordinates $(z^i, w^i)$ introduced in \autoref{subsec:ComplexificationComplex}. It follows that $(I, J, F)$ is a bicomplex structure,
moreover it is holomorphic because $I$ is holomorphic on $(M,J)$ (alternatively: the foliations $\mathcal{F}_i$ are holomorphic).
Since $J$ and the foliations $\mathcal{F}_i$ exist everywhere in $M = N^c$, so do $F$ and $I = -JF$.
\end{proof}

\subsubsection{Bicomplex Kähler structure in the complexification of a Kähler manifold}

Let $(N, g_0, I_0, \omega_0)$ be a real-analytic Kähler manifold and $(M, J)$ a complexification of $N$. 
Denote $I := I_0^c$, $ig := g_0^c$, and $\omega := \omega_0^c$ the holomorphic extensions of $I_0$, $g_0$, and $\omega_0$ respectively.

\begin{theorem} \label{thm:BicomplexKahlerInComplexification}
The triple $(g, F:= IJ, \omega)$ defines a holomorphic bicomplex Kähler structure in (a neighborhood of $N$ in) $M$.
Moreover, the associated complex bi-Lagrangian structure (cf \autoref{prop:BLBKC}) is the canonical bi-Lagrangian structure
$(\omega, \mathcal{F}_1$, $\mathcal{F}_2)$ (cf \autoref{thm:BiLagComplexificationKahler}).
\end{theorem}

\begin{proof} 
We already know that $(I, J, F)$ is a holomorphic bicomplex structure in $M$ whose associated pair of foliations is $(\mathcal{F}_1, \mathcal{F}_2)$
by \autoref{thm:BicomplexInComplexification}, and that $(\omega, \mathcal{F}_1, \mathcal{F}_2)$ is a complex bi-Lagrangian structure by \autoref{thm:BiLagComplexificationKahler}.
It remains to show that $g$ is the complex bi-Lagrangian metric, in other words that the identity $\omega(u,v) = g(F u, v)$ holds.
This is a satisfying computation:
\begin{align}
    g(Fu,v) &= -i g_0^c(Fu,v) && \text{(by definition of }g\text{)}\\
            &= -i g_0^c(-JIu, v) && \text{(since }F = -IJ = -JI\text{)}\\
            &= g_0^c(Iu, v) && \text{(since }g_0^c\text{ is }J\text{-complex linear)}\\
            &= \omega_0^c(u, v) && \text{(see argument below)}
\end{align}
The last line is justified by analytic continuation: since $g_0(I \cdot, \cdot) = \omega_0(\cdot, \cdot)$ holds in $N \hookrightarrow M$, the identity $g_0^c(I \cdot, \cdot) = \omega_0^c(\cdot, \cdot)$
must hold everywhere defined in $M$.
\end{proof}

\autoref{fig:diagram} features a partial overview of the relationships between all structures discussed so far.

\medskip
\begin{figure}[!ht]
\begin{center}
\begin{footnotesize}
\noindent \begin{tikzpicture}[scale = 0.99\textwidth/(7*(72.27/2.54))] % NB: 7 = width of picture, 72.27/2.54 = pts per cm
 \draw [rounded corners, fill=blue, fill opacity=0.2] (0,3) rectangle (2, 4) ;
 \draw (0.5, 3.6) node {$I_0$} ;
 \draw (0.5, 3.4) node {Complex} ;
 \draw [rounded corners, fill=yellow, fill opacity=0.2] (1,3) rectangle (3, 4) ;
 \draw (2.5, 3.6) node {$g_0$} ;
 \draw (2.5, 3.4) node {Riemannian} ;
 \draw [rounded corners, fill=red, fill opacity=0.2] (1,3) rectangle (2, 5) ;
 \draw (1.5, 4.6) node {$\omega_0$} ;
 \draw (1.5, 4.4) node {Symplectic} ;
 \draw (1.5, 3.5) node {Kähler} ;
 
 \draw[->,>=latex] (1.2, 2.7) to[bend right] (1.9, 1.8);
 \draw (1.9, 2.3) node {Complexification} ;
  
 \draw [rounded corners, fill=blue, fill opacity=0.2] (4,3) rectangle (6, 4) ;
 \draw (4.5, 3.6) node {$F$} ;
 \draw (4.5, 3.4) node {Para-complex} ;
 \draw [rounded corners, fill=yellow, fill opacity=0.2] (5,3) rectangle (7, 4) ;
 \draw (6.5, 3.6) node {$g$} ;
 \draw (6.5, 3.4) node {Neutral} ;
 \draw (6.5, 3.25) node {(pseudo-Riem.)} ;
 \draw [rounded corners, fill=red, fill opacity=0.2] (5,3) rectangle (6, 5) ;
 \draw (5.5, 4.4) node {Symplectic} ;
 \draw (5.5, 3.7) node {Para-Kähler} ;
 \draw (5.5, 3.5) node {$\Leftrightarrow$} ;
 \draw (5.5, 3.3) node {Bi-Lagrangian} ;
 \draw (5.5, 4.6) node {$\omega$} ;

 \draw[<->,>=latex] (5.8, 2.7) to[bend left] (5.1, 1.8);
 \draw (4.9, 2.3) node {Real vs complex setting} ;

 \draw [rounded corners, fill=blue, fill opacity=0.2] (2,0) rectangle (4, 1) ;
 \draw (2.5, 0.75) node {$I$, $J$, $F = IJ$} ;
 \draw (2.5, 0.6) node {($I = I_0^c$)} ;
 \draw (2.5, 0.4) node {Bicomplex} ;
 \draw [rounded corners, fill=yellow, fill opacity=0.2] (3,0) rectangle (5, 1) ;
 \draw (4.5, 0.6) node {$g$ $(= -ig_0^c)$} ;
 \draw (4.5, 0.4) node {Holomorphic} ;
 \draw (4.5, 0.25) node {Riemannian} ;
 \draw [rounded corners, fill=red, fill opacity=0.2] (3,0) rectangle (4, 2) ;
 \draw (3.5, 1.6) node {$\omega$ $(= \omega_0^c$)} ;
 \draw (3.5, 1.4) node {Complex} ;
 \draw (3.5, 1.25) node {symplectic} ;
  \draw (3.5, 0.8) node {Bicomplex} ;
 \draw (3.5, 0.65) node {Kähler} ;
 \draw (3.5, 0.5) node {$\Leftrightarrow$} ;
 \draw (3.5, 0.35) node {Complex} ;
 \draw (3.5, 0.2) node {bi-Lagrangian} ;
\end{tikzpicture}
\end{footnotesize} 
\end{center}
\caption{\label{fig:diagram}Partial overview of interrelationships between complex, para-complex, symplectic, and (pseudo-)Riemannian metric structures on real and complex manifolds.}
\end{figure}
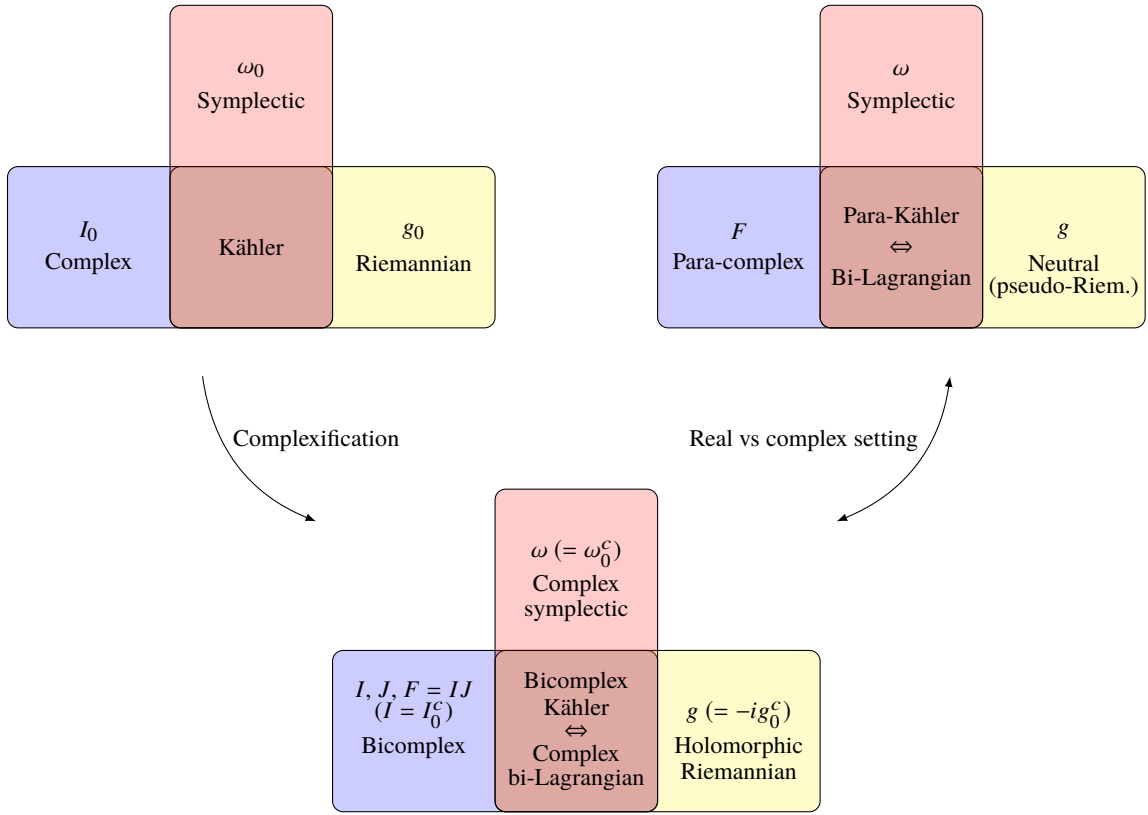
\medskip

\subsection{\texorpdfstring{Example: $\CP^1$}{Example: CP1}} \label{subsec:CP1BL}

 Let $N = \CP^1$ be the complex projective line. Let $N^c = \CP^1 \times \overline{\CP^1}$ denote the canonical complexification of $N$.
 Let $z$ denote the usual complex coordinate in the affine patch $U := \CP^1 \setminus  \left\{[1:0]\right\}$, and let $w = \overline{z}$ in a copy of $U$.
 Then $(z,w)$ are holomorphic coordinates in $U \times \overline{U} \subset N^c$. The manifold $N = \CP^1$ sits inside $N^c$ as the totally real locus $w = \overline{z}$.
 The canonical foliations in $U \times \overline{U}$ are the vertical and horizontal foliations of $U \times \overline{U} \approx \C \times \overline{\C}$,
 given by $\{z = \mathit{constant}\}$ and $\{w = \mathit{constant}\}$.

\subsubsection{Complexification of the Fubini-Study Kähler structure}
The complex projective line $\CP^1$ has a natural \emph{Fubini-Study} Kähler structure inherited from the flat Kähler structure of $\C^2$. Indeed,
$\CP^1$ can be described as the quotient of $S^3 \subset \C^2$ by the isometric action of $\U(1)$ by multiplication (this yields the Hopf fibration $S^1 \to S^3 \to S^2$).
In the affine coordinate $z$, the Fubini-Study Kähler structure $(g_0, I_0, \omega_0)$ is given by
\begin{equation}
 g_0 = \frac{\dz\,\dzbar}{\left(1 + |z|^2\right)^2} \qquad
 I_0 = i \, \dz \otimes \dd{z} -i \, \dzbar \otimes \dd{\overline{z}} \qquad
 \omega_0 = \frac{i \, \dz \wedge \dzbar}{2\left(1 + |z|^2\right)^2}\,.
\end{equation}

The holomorphic extensions of these tensor fields in $N^c$ are the holomorphic metric
$g_0^c$, the complex linear endomorphism $I_0^c$ and the complex symplectic form $\omega_0^c$ given by:
\begin{equation}
 g_0^c = \frac{\dz\,\dw}{\left(1 + z w\right)^2}\qquad
 I_0^c = i \, \dz \otimes \dd{z} - i \,  \dw \otimes \dd{w}\qquad
 \omega_0^c = \frac{i \, \dz \wedge \dw}{2\left(1 + z w\right)^2}\,.
\end{equation}
Observe that $g_0^c$ and $\omega_0^c$ are only defined in a neighborhood of $N = \{\overline{z} = w\}$ in $N^c$: they are singular
at points where $1 + zw = 0$.

\subsubsection{Bi-Lagrangian structure and connection}

The canonical foliations of $N^c$ are the vertical and horizontal foliations $\mathcal{F}_1$ and $\mathcal{F}_2$ of $\CP^1 \times \overline{\CP^1}$.
Let us call $\left(\mathcal{F}^{z_0}_1\right)$ (for ${z_0} \in \overline{\CP^1}$) and $\left(\mathcal{F}^{w_0}_2\right)$ (for ${w_0} \in {\CP^1}$) the leaves of the foliations, 
respectively given by
$\mathcal{F}^{z_0}_1 = \{(z, w) \in N^c :z = z_0\}$ and $\mathcal{F}^{w_0}_2 = \{(z, w) \in N^c : w = w_0\}$.
It is clear that these leaves are complex Lagrangian for $\omega_0^c$.

It is easy to check in this scenario that the complex bi-Lagrangian metric $g$ is equal to $-i g_0^c$
as predicted by \autoref{thm:BiLagComplexificationKahler} \ref{theoremCKBLiii}, indeed:
$g = \omega_0^c(F \cdot, \cdot)$ where 
$F = \mathrm{pr}_{1} - \mathrm{pr}_{2}$, that is:
\begin{equation}
 g = \frac{-i \, \dz\,\dw}{\left(1 + z w\right)^2} = -i\,g_0^c~.
\end{equation}

We now turn to the complex bi-Lagrangian connection $\nabla$, which we compute as the Levi-Civita connection of $g$ using
\emph{Cartan's structural equations}. 
We choose a $g$-orthonormal frame $(E_1, E_2)$ in the holomorphic tangent bundle of $N^c$ and compute its dual coframe $(\chi^1, \chi^2)$:
\begin{equation}
\begin{aligned}
 E_1 &= \alpha\left(\dd{z} + \dd{w}\right) \\
 E_2 &= i\alpha\left(\dd{z} - \dd{w}\right)
\end{aligned} \qquad
\begin{aligned}
 \chi^1 &= \dfrac{\dz + \dw}{2\alpha}\\
 \chi^2 &= \dfrac{\dz - \dw}{2i\alpha}
\end{aligned}
\end{equation}
with $\alpha \coloneqq \frac{1+i}{\sqrt{2}} (1 + z w)$.

The connection $1$-forms $\omega_2^1$ and $\omega_1^2 = -\omega_2^1$ must satisfy 
$\upd \chi^1 = \chi^2 \wedge \omega_2^1$ and
$\upd \chi^2 = \chi^1 \wedge \omega_1^2$ according to Cartan's first structural equations, which yields
\begin{equation}
\omega_2^1 = \frac{-2i}{1 + zw}(w\dz - z\dw) ~. 
\end{equation}
Note that $d\omega_2^1 = 4 \omega_0^c$ as expected from Cartan's second structural equation, knowing that the Fubini-Study metric
has constant sectional curvature $4$.
The connection $\nabla$ is then determined by $\nabla E_1 = \omega_1^2 \otimes E_2$ and $\nabla E_2 = \omega_2^1 \otimes E_1$. In the $(z,w)$ coordinates,
$\nabla$ is thus given by:
\begin{equation} \label{eq:BLCCP1}
\begin{aligned}
 \nabla \dd{z} &= \dfrac{-2 w}{1 + z w} \, \dz \otimes \dd{z}\\
 \nabla \dd{w} &= \dfrac{-2 z}{1 + z w} \,\dw \otimes \dd{w}
\end{aligned}\qquad
\begin{aligned}
 \nabla \dd{\overline{z}} &= \overline{\nabla \dd{z}} = \dfrac{-2 \overline{w}}{1 + \overline{z} \overline{w}} \, \dzbar \otimes \dd{\overline{z}}\\
 \nabla \dd{\overline{w}} &= \overline{\nabla \dd{w}} = \dfrac{-2 \overline{z}}{1 + \overline{z} \overline{w}} \, \dwbar \otimes \dd{\overline{w}}~.
\end{aligned}
\end{equation}

\subsubsection{Geodesics and affine structure in the vertical leaves}

Let $L = \mathcal{F}^{z_0}_1$ be a vertical leaf. Let us compute the geodesics in $L$ for the Bott connection. These are also
geodesics for the bi-Lagrangian connection $\nabla$, since $\nabla$ restricts to the Bott connection in $L$. Let $\gamma(t) = (z_0, w(t))$ be a path in $L$. 
It is a geodesic if and only if $\nabla_t \gamma(t) = 0$. Using the expression of $\nabla$ found above,
this reduces to the equation $w''(t) - \dfrac{2 z_0 w'(t)^2}{1 + z_0 w(t)} = 0$. This ODE is quickly solved noting that
it is rewritten $f''(t) = 0$, where $f(t) = \dfrac{1}{1 + z_0 w(t)}$. The next proposition follows.

\begin{proposition}
 The geodesic $\gamma(t)$ in $L$ (or in $N^c$) with initial value $\gamma(0) = (z_0, w_0)$ and initial
 tangent vector $\gamma'(0) = a \dd{w} + \overline{a} \dd{\overline{w}}$ is given by $\gamma(t) = (z_0, w(t))$
 with $w(t) = \dfrac{at + w_0(1 + z_0 w_0)}{-z_0 a t + 1 + z_0 w_0}$.
\end{proposition}

We can then proceed to describe the complex affine structure in the leaf. 
The exponential map at a point $(z_0, w_0)$ 
for the Bott connection in $L$ identifies (an open set of) the tangent space $\upT _{(z_0, w_0)} L$ with (an open set of) $L$
as affine spaces. We choose $w_0 = 0$ and identify $\upT _{(z_0, 0)} L \approx \C$ via 
$a \dd{w} + \overline{a} \dd{\overline{w}} \mapsto a$. The affine structure in $L$ is then given by,
for $(z_0, w) \in L$ and $a \in \C$: 
\begin{equation}
 (z_0, w) \ast a = (z_0, f_a(w))
\end{equation}
 where $f_a$ is the projective linear transformation (homography) associated to the matrix
\begin{equation}
 M_a = \left(\begin{array}{cc}
        1 + a z_0 & a\\
        -a {z_0}^2 & 1 - az_0
       \end{array}\right)~.
\end{equation}
One checks that 
$M_a = P_{z_0} \left(\begin{array}{cc}
        1 & a\\
        0 & 1
       \end{array}\right) P_{z_0}^{-1}$ where $P_{z_0} = \left(\begin{array}{cc}
        z_0 & 1\\
        -{z_0}^2 & 0
       \end{array}\right)$. The next proposition ensues.

\begin{proposition}
 The affine structure in the leaf $L = \mathcal{F}^{z_0}_1 \approx \CP^1$, in the coordinate $w$, is the affine structure identifying $\CP^1$
 (minus a point) with the affine linear space $\C$ via the map
 \begin{equation}
 \begin{split}
  \CP^1 &\to \C\\
  w & \mapsto \frac{-1}{z_0(1 + z_0 w)}~.
  \end{split}
\end{equation}
\end{proposition}

\begin{remark}
The affine structure is singular where $1 + z_0w = 0$, which is not surprising since the bi-Lagrangian structure is
singular there. In particular, the bi-Lagrangian connection is incomplete.
\end{remark}

\section{Background on Teichmüller theory} \label{sec:TeichmullerBackground}

In this section, we provide a brief introduction to Teichmüller theory and the theory of deformation spaces associated to surfaces. 
Our main purpose here is to establish the context and notations that we will rely on 
in the next section, where we study these deformation spaces using symplectic geometry. The reader who is familiar with Teichmüller theory and the aforementioned deformation spaces may safely skip this section and return to it whenever necessary.

Throughout this section, $S$ is a connected, oriented, smooth, closed surface of negative Euler characteristic. We let denote $\pi$ its fundamental group with respect to some fixed basepoint.

\subsection{Geometric structures and character varieties}

We start by introducing the notion of locally homogeneous geometric structures on manifolds, their attendant deformation spaces, 
and the corresponding holonomy maps to character varieties. We refer to \cite{MR2827816} for a nice survey (also see \cite{MR1435975, MR957518, MR3289710}).

\subsubsection{Geometric structures and their deformation space} \label{subsubsec:GeomStrDef}

Let $X = G/H$ be a homogeneous space. More precisely, assume $X$ is a real-analytic manifold on which
a Lie group $G$ acts analytically and transitively. The pair $(X,G)$ is called a \emph{(Klein) geometry}. 
An \emph{$(X,G)$-structure} on a smooth manifold $M$
is given by an atlas of charts mapping open sets of $M$ into $X$ such that the transition 
functions coincide with the action of elements of $G$. Such atlases are always assumed compatible
with the smooth structure of $M$ and the orientation if $M$ is oriented.

The group $\Diff(M)$ of (orientation-preserving) diffeomorphisms of $M$ has a right action
on the set of $(X,G)$-structures on $M$ by pulling back atlases. Let $\Diff_{0}(M)<\Diff(M)$ denote its identity component,
consisting of diffeomorphisms isotopic to the identity. 
The \emph{deformation space of $(X,G)$-structures} on $M$ is defined as the quotient:
\begin{equation}
 \mathcal{D}_{(X,G)}(M) = \bigslant{\{(X,G)\text{-structures on }M\}}{\Diff_0(M)}\,.
\end{equation} 

This space comes with a natural topology (see \cite{MR957518}). In general it may be pathological (\eg{} the deformation space
of affine structures on a $2$-torus is not Hausdorff \cite{MR2181958}), but in this paper we will study geometries for which
the deformation space is a nonsingular real-analytic manifold.

\begin{remark}
 The \emph{mapping class group} $\pi_0(\Diff(M)) =  \Diff(M)/\Diff_0(M)$ has a natural action on $\mathcal{D}_{(X,G)}(M)$, and the quotient
 space $\mathcal{M}_{(X,G)}(M) = \{(X,G)\text{-structures on }M\}/\Diff(M)$
 is the \emph{moduli space of $(X,G)$-structures} on $M$.
\end{remark}

\begin{remark} \label{remark:Diff1vsDiff0}
Some authors prefer to use the group $\Diff_1(M)$ of \emph{homotopically trivial} diffeomorphisms instead of $\Diff_0(M)$ in the definition of the deformation space, see
\cite[Remark 3.19]{MR3289710}. When $M = S$ is a surface, $\Diff_0(S)$ and $\Diff_1(S)$ coincide. In general 
$\Diff_0(M)$ is always a normal subgroup of $\Diff_1(M)$. See \cite{HatcherMO} for more details.
% An argument in favor of using $\Diff_1(M)$ is that the holonomy map \eqref{eq:HolonomyMap1} in restriction to the space of complete $(X,G)$-structures becomes an embedding in that case
% provided that the natural homomorphism $\Diff(M)/\Diff_1(M) \to \Out(\pi_1(M))$ is injective: see \cite[Theorem 3.25]{MR3289710}. On the other hand, we like using
% $\Diff_0(M)$ because of the natural action of the mapping class group $\pi_0(\Diff(M))$ on the deformation space.
\end{remark}

\subsubsection*{Holonomy} \label{subsubsec:Holonomy}

A common way to construct $(X,G)$-structures is to find discrete subgroups of $G$ acting freely and properly on open sets $\Omega \subseteq X$.
The quotient $M \coloneqq \Omega / \Gamma$ inherits a unique $(X,G)$-structure such that the covering map
$\Omega \to M$ is a morphism of $(X,G)$-structures. Such $(X,G)$-structures are called \emph{embedded}.

Assume $M$ is a fixed smooth manifold and denote by $\pi$ its fundamental group. 
Any embedded $(X,G)$-structure on $M$ can be recovered from a group homomorphism
$\rho \colon \pi \to G$ with discrete image $\Gamma < G$
and a covering map $f : \tilde{M} \to \Omega \subseteq X$ that is \emph{$\rho$-equivariant}, \ie{} equivariant for the action of $\pi$ on the universal
cover $\tilde{M}$ by deck transformations and on $X$ via $\rho \colon \pi \to G$. The quotient map 
$\bar{f} \colon M \to \Omega / \Gamma$ identifies $M$ to $\Omega / \Gamma$ as $(X,G)$-manifolds.

More generally, any $(X,G)$-structure on $M$ can be described by a \emph{developing map} $f : \tilde{M} \to X$ and a \emph{holonomy representation}
$\rho \colon \pi \to G$, where $f$ is a $\rho$-equivariant local diffeomorphism. The developing map is essentially a global chart on $\tilde{M}$; conversely any $(X,G)$-structure has a developing map that can be reconstructed from any local chart 
by analytic continuation \cite{MR1435975}.
The pair $(f,\rho)$ for a given $(X,G)$-structure is unique
up to the action of $G$ by post-composition on $f$ and conjugation on $\rho$. 
Embedded $(X,G)$-structures are those whose developing map is a covering onto its image, and are called \emph{complete}
when it is surjective.

The group $\Diff(M)$ has a natural right action on the set $\Dev_{(X,G)}(M)$ of all development pairs $(f,\rho)$ 
by precomposition on $f$. This action commutes with the left action of $G$ on such pairs by post-composition on $f$ and conjugation on $\rho$.
Consequently, assigning to a marked $(X, G)$-structure the conjugacy class of its holonomy representation is a well-defined operation that induces 
the \emph{holonomy map}
\begin{equation} \label{eq:HolonomyMap1}
  \mathit{hol} \colon \mathcal{D}_{(X,G)}(M) \to \Hom(\pi, G)/G\,.
\end{equation}

\begin{remark}
After defining character varieties, we also call \emph{holonomy map} the map
$\mathcal{D}_{(X,G)}(M) \to {\cal X}(\pi, G)$ obtained by composing
\eqref{eq:HolonomyMap1} with the surjection $\Hom(\pi,G)/G \to {\cal X}(\pi, G)$.

\end{remark}

One expects the holonomy map to be a local homeomorphism, but this can fail in some instances, see \autoref{remark:holonomy}. 
However, we shall see that it is the case for the geometries considered in this paper.

\begin{remark} \label{remark:holonomy}
% Let us expand on the topological behavior of the holonomy map in general for the curious reader.
% A consequence of the previous discussion is that $\mathit{hol}$ appears in the following commutative diagram:
% \begin{center}
% \begin{tikzcd}[column sep=4em, row sep=3em]
%   \Dev_{(X,G)}(M) \arrow[d, "/G"] \arrow[dr, "/\Diff_0(M)"] \arrow[drr, bend left, "\widehat{\widehat{\mathit{hol}}}"] \\
%   \{(X,G)\text{-struct. on }M\} \arrow[dr, "/\Diff_0(M)"] & 
%   \bigslant{\Dev_{(X,G)}(M)}{\Diff_0(M)} \arrow[r, "\widehat{\mathit{hol}}"] \arrow[d, "/G"] & \Hom(\pi,G) \arrow[d, "/G"]\\
%   & \mathcal{D}_{(X,G)}(M) \arrow[r, "\mathit{hol}"] &  \Hom(\pi, G)/G
% \end{tikzcd}
% \end{center}
% Note that the map $\widehat{\widehat{\mathit{hol}}}$ is just $(f,\rho) \mapsto \rho$. 
The best result about the holonomy map in general is that its $G$-equivariant lift 
\begin{equation}
 \widehat{\mathit{hol}} \colon \bigslant{\Dev_{(X,G)}(M)}{\Diff_0(M)} \to \Hom(\pi,G)
\end{equation}
is a local homeomorphism (\emph{Deformation theorem}, \label{DeformationTheorem} see \cite{MR2827816}). 
The quotient map $\mathit{hol}$ is still open, but local injectivity may fail in general (see \cite{MR1043223}, \cite{MR3289710}).
However, in restriction to the \emph{stable} locus of $(X,G)$-structures with irreducible holonomy (see \autoref{subsubsec:CharVar}),
the action of $G$ is proper, so that $\mathit{hol}$ descends to a local homeomorphism at least away from the fixed points of the action.
\end{remark}

\subsubsection{Character varieties} \label{subsubsec:CharVar}

The quotient $\Hom(\pi, G)/G$ is usually very pathological: $\Hom(\pi, G)$ can already be singular, and the action of $G$ on it by conjugation is generally neither proper nor free. It is often preferable to work with a Hausdorff quotient called the \emph{character variety}, which moreover has an algebraic structure when $G$ is algebraic. Let us briefly discuss this and refer to \cite{MR2931326}
for more details.

Let $\pi$ be a finitely generated group and let $G$ be a
reductive complex algebraic group\footnote{A reductive algebraic group is assumed affine by definition, equivalently it admits a faithful linear representation. A connected affine algebraic group over $\C$ is reductive if and only if it has a reductive Lie algebra and its center
is of multiplicative type. It can also be characterized as the complexification of a compact real algebraic group.}, 
for instance $G = \PSL_n(\C)$. Let $\Hom(\pi, G)$ denote the space of group homomorphisms from $\pi$ to $G$, abusively called \emph{representations}. 
It is an affine algebraic set on which $G$ acts by conjugation, and the character variety 
${\cal X}(\pi, G)$ is defined as the algebraic quotient or \emph{GIT quotient}\footnote{
Let us quickly recall how this works. Let $M$ be an affine algebraic set and $G$ an algebraic group acting on $M$. 
Let $\C[M]$ denote the coordinate ring of regular functions on $M$
and $\C[M]^{G} \subset \C[M]$ the subalgebra of $G$-invariant functions.
$\C[M]^{G}$ is finitely generated if $G$ is reductive by Nagata's theorem solving Hilbert's 14th problem. Therefore $\C[M]^{G}$ is the coordinate ring of an affine algebraic set, namely $\Spec \C[M]^{G}$. This affine set is the GIT quotient $M/\!/G$.}:
$\cX(\pi, G) = \bigdoubleslant{\Hom(\pi, G)}{G}$.

\begin{remark}
The character ``variety'' ${\cal X}(\pi, G)$ is an affine algebraic set, but it is not necessarily irreducible when $\Hom(\pi, G)$ itself is not.
\end{remark}

As a topological space, ${\cal X}(\pi, G)$ is the largest Hausdorff quotient of $\Hom(\pi, G)/G$, \ie{}
% \begin{equation} \label{eq:CharacterVariety1}
%   {\cal X}(\pi, G) = \bigslant{\Hom(\pi, G)}{\mathord{\sim}}
%  \end{equation}
${\cal X}(\pi, G) = \Hom(\pi, G) / \mathord{\sim}$
where the equivalence relation $\sim$ identifies representations whose orbit closures intersect.
In fact this may be taken as a definition of the character variety
when $G$ is merely a real Lie group. When $G$ is a reductive complex algebraic group, 
any equivalence class contains a unique closed orbit, and representations with closed orbits are precisely the completely reducible ones\footnote{By
definition, a \emph{completely reducible} representation $\rho : \pi \to G$ is one for which $H \coloneqq \rho(\pi)$ is a completely reducible subgroup of $G$, meaning that 
its Zariski closure is a reductive subgroup. Equivalently, for every parabolic subgroup $P$ containing $H$, there is a Levi factor of $P$ containing $H$.}.
Such representations are the \emph{polystable} locus $\Hom^{\textnormal{ps}}(\pi, G) \subset \Hom(\pi, G)$ in the language of GIT. Consequently,
the character variety as a topological space is
$  {\cal X}(\pi, G) = \Hom(\pi, G) / \mathord{\sim} = \Hom^{\textnormal{ps}}(\pi, G) / G$.
% \begin{equation}
%   {\cal X}(\pi, G) = \bigslant{\Hom(\pi, G)}{\mathord{\sim}} = \bigslant{\Hom^{\textnormal{ps}}(\pi, G)}{G}\,.
%  \end{equation}
In particular, there is a natural surjective map $\Hom(\pi,G)/G \to {\cal X}(\pi, G)$ which restricts to a bijection, in fact a homeomorphism,
$\Hom^{\textnormal{ps}}(\pi, G)/G \isomap {\cal X}(\pi, G)$.

\begin{remark}
\emph{Stable} representations in the language of GIT are the irreducible representations $\rho : \pi \to G$\footnote{A representation
$\rho : \pi \to G$ is called \emph{irreducible} when $\rho(\pi)$ is not contained
in any parabolic subgroup of $G$. Technically, irreducible representations are the stable points of $\Hom(\pi, G)$ for the action of $G/Z(G)$,
they are only (``properly'') stable for the action of $G$ when $G$ has finite center.}. When $\pi$ is the fundamental group of a closed surface
of negative Euler characteristic,
They are smooth points of $\Hom(\pi, G)$, so they project to either smooth points or orbifold points in ${\cal X}(\pi, G)$,
depending on whether their centralizer in $G$ is equal to the center $Z(G)< G$ or a finite extension of it.
It is known that there are no orbifold points for $G = \GL_n(\C)$ or $\SL_n(\C)$.
\end{remark}

\subsection{\texorpdfstring{Fricke-Klein space $\FS$ and the character variety ${\cal X}(\pi, \PSL_2(\R))$}
{Fricke-Klein space F(S) and the character variety X(S,PSL(2,R))}} \label{subsec:Fricke}

We now specialize to $2$-dimensional hyperbolic geometry, \ie{} the Klein geometry $(X,G)$ where $X = \mathbb{H}^2$ is the hyperbolic plane and 
$G = \Isom^+(\H^2) \approx \PSL_2(\R)$.
An $(X,G)$-structure on $S$ is called a \emph{hyperbolic structure}, and the deformation space
\begin{equation}
 \FS = \bigslant{\{\text{hyperbolic structures on }S\}}{\Diff_0(M)}
\end{equation} 
is called the \emph{Fricke-Klein space} of $S$.
It is topologically a cell of dimension $6g -6$ where $g$ is the genus of $S$,
and has a natural real-analytic structure (see \eg{} \cite{MR2284826} for several proofs).
 
In this setting, all $(X,G)$-structures are automatically complete and we have:
\begin{theorem}[Goldman \cite{MR2630832}]
 The holonomy map $\mathit{hol}: \FS \to {\cal X}(\pi, \PSL_2(\R))$ is an embedding. Its image in the character variety is
 a connected component, consisting precisely of the equivalence classes of all discrete and faithful group homomorphisms $\rho \colon \pi \to G$\footnote{To be accurate,
 the discrete and faithful representations form \emph{two} connected components of the real character variety.
 They correspond to holonomies of hyperbolic structures on $S$ equipped with either of its two possible orientations.}.
\end{theorem}

\begin{remark}
It is a general fact that if $G$ preserves a Riemannian metric on $X$ (\ie{} $G$ acts on $X$ with compact stabilizers) and $M$ is closed, then
all $(X,G)$-structures are complete. Furthermore, under some topological restrictions on $M$ (see \autoref{remark:Diff1vsDiff0}), the holonomy map
$\mathit{hol} \colon \mathcal{D}_{(X,G)}(M) \to {\cal X}(\pi, G)$ is an embedding. Goldman showed that 
in the specialization being discussed here, the image of $\FS$ in the character variety is the connected component with maximal Euler class.
\end{remark}

\begin{remark}
 The character ``variety'' ${\cal X}(\pi, \PSL_2(\R))$ does not come with an algebraic structure
 since $\PSL_2(\R)$ is not a complex reductive algebraic group.
 There is of course a natural ``inclusion''\footnote{This ``inclusion'' is actually a $2:1$ map, essentially because $\PSL_2(\R) \neq \PGL_2(\R)$.} 
 map ${\cal X}(\pi, \PSL_2(\R)) \to {\cal X}(\pi, \PSL_2(\C))$,
 but the image of ${\cal X}(\pi, \PSL_2(\R))$ in ${\cal X}(\pi, \PSL_2(\C))$ is only a real \emph{semi}-algebraic subset. 
In the case where $\pi$ is the fundamental group of a closed surface, it is also a real-analytic subvariety
 by virtue of the nonabelian Hodge correspondence\footnote{This real-analytic correspondence 
 between the character variety and the moduli space of Higgs bundles is due to Hitchin and Donaldson (\cite{MR887284}, \cite{MR887285}) and to Corlette and Simpson 
 (\cite{MR965220}, \cite{MR1159261, MR1179076}) in the more general case where $\pi$ is the fundamental group of a smooth Kähler variety and $G$ is a reductive complex algebraic group.}.
\end{remark}

\subsection{\texorpdfstring{Teichmüller space $\TS$ and the Weil-Petersson metric}{Teichmüller space T(S) and the Weil-Petersson metric}} \label{subsec:TSWP}

The \emph{Teichmüller space} of $S$ is the deformation space of complex structures on $S$. It is defined similarly to a deformation space
of $(X,G)$-structures, with the pseudogroup of invertible holomorphic functions
between open sets of $\C$ playing the role of the Lie group $G$ acting on $X = \C$. Good references for Teichmüller theory and the Weil-Petersson metric include \cite{MR2245223, MR2284826, MR2641916}.

A \emph{complex structure} on $S$ is given by an atlas of charts mapping open sets of $S$ into $\C$ such that the transition functions are holomorphic. We typically denote by $X$ the surface $S$ equipped with a complex structure, $X$ is called a Riemann surface. Complex structures on $S$ are always assumed compatible with its given smooth structure and orientation.
The group $\Diff(S)$ of orientation-preserving diffeomorphisms of $S$ acts on complex structures on $S$ by pulling back atlases. 
The \emph{Teichmüller space of $S$} is defined as the quotient $\TS \coloneqq \{\text{complex structures on }S\}/\Diff_0(S)$.

Kodaira-Spencer deformation theory (see \cite{MR0112154}, \cite{MR0276999}) shows that 
$\TS$ is a complex manifold with holomorphic tangent space $\upT _X \TS = \check{H}^1(X,\Theta_X)$,
where $\Theta_X$ is the sheaf of holomorphic vector fields on $X$. Via Serre duality,
one can identify the holomorphic cotangent space as $ \upT _X^* \TS = H^0(X, K_X^2)$ where $K_X$ is the canonical bundle
on $X$. In other words $\upT _X^* \TS$ is the space of \emph{holomorphic quadratic differentials} on $X$, that is tensors on $X$ of the form
$\phi =  \varphi(z)dz^2$ with $\varphi$ holomorphic. A quick application of the Riemann-Roch theorem shows that $\TS$ has complex dimension 
$3g - 3$ where $g$ is the genus of $S$.

An immediate consequence of the celebrated Poincaré uniformization theorem and the fact that any hyperbolic structure on $S$ is complete (\autoref{subsec:Fricke}) is the 1:1 correspondence between complex structures on $S$ and hyperbolic structures on $S$. Concretely, one can associate to each
complex structure $X$ the hyperbolic structure given by the unique conformal metric on $X$ with 
constant curvature $-1$.
This correspondence passes to the quotient as a real-analytic diffeomorphism:
\begin{equation} \label{eq:uniformizationmap}
F: \TS \isomap \FS
\end{equation}
where $\FS$ is the Fricke-Klein space (\autoref{subsec:Fricke}). 
For this reason, $\FS$ itself is sometimes called ``Teichmüller space'', but we will maintain the distinction between $\TS$ and $\FS$.

The \emph{Weil-Petersson product} of two holomorphic quadratic differentials $\phi$ and $\psi$ is defined by
\begin{equation}
 \left<\phi,\psi\right>_{\textnormal{WP}} = -\frac{1}{4} \int_X \phi \cdot {\sigma}^{-1} \cdot \overline{\psi} 
\end{equation}
where ${\sigma}^{-1}$ is the dual current of the Poincaré area form $\sigma$. 
This Hermitian inner product on $ \upT _X^* \TS = H^0(X, K_X^2)$ defines by duality a Hermitian metric $\left<\cdot,\cdot\right>_{\textnormal{WP}}$ on the complex manifold $\TS$, 
which turns out to be Kähler, as was first shown by Ahlfors \cite{MR0204641}. It is called the \emph{Weil-Petersson metric} on $\TS$ and we shall denote it $h_{\textnormal{WP}} = g_{\textnormal{WP}} - i \omega_{\textnormal{WP}}$
where $g_{\textnormal{WP}}$ and $\omega_{\textnormal{WP}}$ are the Weil-Petersson Riemannian metric and Kähler form respectively.

\subsection{\texorpdfstring{Deformation space of complex projective structures $\CPS$}{Deformation space of complex projective structures CP(S)}} \label{subsec:CPS}

\subsubsection{Deformation space and holonomy}

We now consider the geometry $(X,G)$, where $X = \CP^1$ and 
$G = \PSL_2(\C)$ acting projective linearly. An $(X,G)$-structure on $S$ is a \emph{complex projective structure}.
An excellent reference is \cite{MR2497780}.

The deformation space $\CPS$ is a complex manifold. As in the case of Teichmüller space $\TS$,
one can use deformation theory to express the tangent space of $\CPS$ in terms of \v{C}ech cohomology,
but the description of tangent vectors is not as explicit as with $\TS$. On the other hand, $\CPS$ has a remarkable parametrization as a holomorphic affine bundle over $\TS$. We briefly describe this so-called \emph{Schwarzian parametrization} in \autoref{subsubsec:SchwarzianParam}.

The holonomy of a complex projective structure on $S$ is always irreducible, so the holonomy map
$\mathit{hol}: \CPS \rightarrow  {\cal X}(\pi, \PSL_2(\C))$
lands in the smooth locus of the character variety. Moreover, it is a local biholomorphism \cite{MR0463429, MR624807, MR624819} (but not a covering map \cite{MR0463429}). Gallo, Kapovich, Marden \cite{MR1765706} precisely determined its image: 
a representation $\rho: \pi \rightarrow \PSL_2(\C)$ is the holonomy of a complex projective structure if and only if $\rho$ is nonelementary
and lifts to $\SL_2(\C)$. Here, $\rho$ is called \emph{elementary} if its image is bounded or virtually abelian.

\subsubsection{Schwarzian parametrization} \label{subsubsec:SchwarzianParam}

Since a complex projective atlas is in particular a holomorphic atlas, we have a ``forgetful projection'' 
\begin{equation} \label{eq:ForgetfulProjection}
 p \colon \CPS \to \TS~. 
\end{equation}
The map $p$ is onto $\TS$, because the Fuchsian section $\sigma_{0} : \TS \to \CPS$ is a right inverse (\autoref{subsec:FuchsianAndQF}).
It is also not hard to show that $p$ is holomorphic. It turns out that for any $X \in \TS$, the fiber $p^{-1}(X)$ can be equipped with the structure of an affine space
modeled on the vector space $H^0(X, K_X^2)$, we explain this soon. For now, recall that $H^0(X, K_X^2) = \upT _X^* \TS$,
so that globally $\CPS$ is an affine bundle modeled on the vector bundle $\upT ^* \TS$. 
The choice of a ``zero section'' $\sigma : \TS \to \CPS$ defines an isomorphism $\tau_\sigma : \CPS \to \upT ^*\TS$, characterized by the fact that $\tau_\sigma \circ \sigma$ is the zero section of $\upT ^*\TS$.

We now explain why, for $X \in \TS$, the fiber $p^{-1}(X) \subset \CPS$ is an affine space
modeled on $H^0(X, K_X^2)$. How can one associate to $Z_1, Z_2 \in p^{-1}(X)$ a holomorphic
quadratic differential $\phi \in H^0(X, K_X^2)$ representing the affine subtraction $Z_2 - Z_1$?  The \emph{Schwarzian derivative} holds the answer. We describe it quickly and refer to \cite{MR2497780} and \cite{MR2698860} 
for more details. 

The Schwarzian derivative of a locally injective holomorphic function $f : U \subseteq \C \to \C$ is
given by the explicit formula: $ Sf(z) = \frac{f'''(z)}{f'(z)} - \frac{3}{2}{\left(\frac{f''(z)}{f'(z)}\right)}^2$.
% \begin{equation}
%  Sf(z) = \frac{f'''(z)}{f'(z)} - \frac{3}{2}{\left(\frac{f''(z)}{f'(z)}\right)}^2~.
% \end{equation}
One should think of this as a holomorphic quadratic differential $Sf = Sf(z)\,dz^2$ rather than a function, given 
how it transforms under a projective change of coordinate on $\CP^1$.
In a sense that can be made precise, the Schwarzian derivative of $f$ measures how far $f$ deviates from being a projective linear transformation\footnote{The \emph{osculating map to $f$} is the map $\mathit{Osc}(f) : U \to \PSL_2(\C)$ whose value at a point $z$ is the projective linear transformation that
best approximates $f$ at $z$. The Darboux derivative of $\mathit{Osc}(f)$ is equal to the Schwarzian derivative of $f$, suitably interpreted. 
In particular, it is clear that $f$ is projective linear if and only if $\mathit{Osc}(f)$ is constant \ie $Sf = 0$.}.

Given a locally injective holomorphic map $f : Z_1 \to Z_2$ where $Z_1, Z_2 \in p^{-1}(X)$,
one can take the Schwarzian derivative of $f$ in local projective coordinates on 
$Z_1$ and $Z_2$, which yields a holomorphic quadratic differential $\phi \in H^0(X, K_X^2)$. When $f = \mathrm{id}_X$,
$\varphi$ is the affine subtraction $Z_2 - Z_1\in H^{0}(X, K_{X}^{2}).$ The fact that it respects the axioms of affine subtraction (Weyl's axioms) is a consequence of the classical properties of the Schwarzian derivative.

\subsection{Fuchsian and quasi-Fuchsian structures} \label{subsec:FuchsianAndQF}

\subsubsection{Fuchsian and quasi-Fuchsian deformation spaces}

A \emph{Kleinian group} is a discrete subgroup of $\PSL_2(\C)$\footnote{Kleinian groups
were classically required to have a nonempty domain of discontinuity (\eg{} for Thurston \cite{ThurstonNotes}, Maskit \cite{MR959135} or Kapovich \cite{MR2553578}), 
but modern usage tends to allow any discrete group.}. The group $\PSL_2(\C)$ can be regarded 
as the group of automorphisms of the complex projective line $\CP^1$ acting projectively linearly, or the group of the automorphisms the Riemann sphere $\hat{\C} = \C \cup \{\infty\} \approx S^2$
acting by Möbius transformations, or the group of orientation-preserving isometries of hyperbolic 3-space $\H^3$. 
% These points of view are consistent with each other considering that:
% \begin{itemize}
%  \item Stereographic projection identifies the round $2$-sphere $S^2$ with the Riemann sphere $\hat{\C}$, which can also be identified to the complex projective line $\CP^1$ 
%  (via the ``affine patch'' $z \mapsto [z : 1]$).
%  \item $S^2$ is also the \emph{ideal boundary} of $\H^3$ (in other words its Gromov boundary), and an isometry of $\H^3$ is uniquely determined by its continuous extension
%  on the ideal boundary.
% \end{itemize}

Let $\Gamma$ be a Kleinian group. A point $x \in \CP^1$ is called a \emph{point of discontinuity} for $\Gamma$ if $x$ has a neighborhood $U \subset \CP^1$ such that
$gU \cap U = \emptyset$ for all but finitely many $g \in \Gamma$. The \emph{domain of discontinuity} of $\Gamma$ is the set $\Omega \subset \CP^1$
of all points of discontinuity. One can show that $\Gamma$ acts properly on $\Omega$, and $\Omega$ is the largest open set with that property.
The \emph{limit set} of $\Gamma$ is the set $\Lambda = \CP^1 - \Omega$. One can show that $\Lambda$ is finite if and only if $\Gamma$ is elementary (virtually abelian).
When $\Gamma$ is not elementary, its limit set $\Lambda$ can be characterized as the smallest $\Gamma$-invariant closed subset of $\CP^1$, or the set of accumulation points
of the orbit of any point $x \in \H^3 \cup \CP^1$.

A \emph{Fuchsian group} is a Kleinian group
conjugate to a subgroup of $\PSL_2(\R) \subset \PSL_2(\C)$. It is characterized as a Kleinian group whose limit set $\Lambda$ is contained 
in some round circle $C$, and which preserves both components of $\CP^1 - C$.
More generally, a \emph{quasi-Fuchsian group} is a Kleinian group $\Gamma$ whose limit set $\Lambda$ is contained in some topological circle $C$, and which preserves both components
of $\CP^1 - C$. A [quasi-]Fuchsian group $\Gamma$ is called \emph{of the first kind} or \emph{of the second kind}
depending on whether $\Lambda = C$ or $\Lambda \subsetneq C$. If $\Gamma$ is a finitely generated torsion-free Kleinian group, 
then $\Gamma$ is quasi-Fuchsian of the first kind if and only if its action on $\CP^1$ is topologically conjugate to the action of a Fuchsian group of the first kind, \ie{} there exists a homeomorphism $f : \CP^1 \to \CP^1$ such that $f \Gamma f^{-1}$ is Fuchsian
(\cite[Theorem 4]{MR0297992}). Moreover, $f$ must be quasiconformal.

Let $S$ be a closed oriented surface with fundamental group $\pi$. A representation $\rho \colon \pi \to \PSL_2(\C)$ is called
[quasi-]Fuchsian  if $\rho$ is injective and its image $\Gamma = \rho(\pi) \subset \PSL_2(\C)$ is a [quasi-]Fuchsian group of the first kind.
The \emph{deformation space of [quasi-]Fuchsian} structures denoted $\FS$ [$\QFS$] is the corresponding subset of the character variety $\mathcal{X}(\pi, \PSL_2(\C))$. Remarkably, $\QFS$ is an open subset of the character variety\footnote{This fact is essentially due to Bers' simultaneous uniformization theorem discussed in the next
paragraph. It is moreover true
that $\QFS$ is the interior of the subset $\mathcal{AH}(S) \subset \mathcal{X}(\pi, \PSL_2(\C))$ of equivalence classes of
discrete and faithful representations. This theorem
is due to the work of Ahlfors, Bers, Maskit, Kra, Marden, Thurston and 
Sullivan \cite{MR0115006, MR0167618, MR910353, MR0306485, MR0349992, MR806415}. 
We refer to \cite[Chapter 7]{MR2096234} for an exposition of this theory, which holds more generally for geometrically finite representations.
The positive answer to the Bers-Sullivan-Thurston \emph{density conjecture} \cite{MR3001608, MR2821565} shows that furthermore, 
$\mathcal{AH}(S)$ is the closure of $\QFS$ in the character variety.}, and lies in the smooth locus.

\subsubsection{Bers slices, Bers sections, Bers embeddings}

Let $\rho$ be quasi-Fuchsian and $\Gamma \coloneqq \rho(\pi) \subset \PSL_2(\C)$. The domain
of discontinuity $\Omega$ of $\Gamma$ is the disjoint union of two invariant topological disks $\Omega^+$ and $\Omega^-$ on which $\Gamma$ acts freely properly,
therefore $\Omega^+ / \Gamma$ and $\Omega^- / \Gamma$ are both closed surfaces diffeomorphic to $S$. 
Let us denote $S^+ = S$ and $S^-$ the same surface with the opposite orientation. 
We can assume that $\Omega^\pm$ has the same orientation as $S^\pm$ in the sense that there exists an orientation-preserving homeomorphism $f^\pm \colon S^\pm \to \Omega^\pm / \Gamma$ whose induced map at the level of fundamental groups is $\rho$. Both quotient surfaces $\Omega^\pm / \Gamma$ inherit complex projective structures, called
\emph{standard quasi-Fuchsian projective structures}\footnote{There are infinitely many complex projective structures with same quasi-Fuchsian holonomy $\rho$.
The standard one is the only one that is an embedded projective structure (whose developing map is an embedding), all the others have infinite-to-one developing maps. 
They are called \emph{exotic} quasi-Fuchsian structures and Goldman \cite{MR882826}  showed
that any one of them can be obtained from the standard one by grafting along an rational lamination. Shinpei Baba gave a generalization of this result
for an arbitrary holonomy fiber $\operatorname{hol}^{-1}(\rho) \subset \CPS$ \cite{MR3447103, BabaII}.}. 

This defines a map
$(Z^+, Z^-) \colon \QFS \to \mathcal{CP}(S^+) \times \mathcal{CP}(S^-)$, and composing with the forgetful projection yields a second map
$\QFS \to \mathcal{T}(S^+) \times \mathcal{T}(S^-)$. Bers's \emph{simultaneous uniformization} theorem \cite{MR0111834}
says that the latter is a biholomorphism. We denote its inverse by
\begin{equation} \label{eq:QFmap}
 \mathit{QF}( \cdot, \cdot) ~\colon~ \mathcal{T}(S^+) \times \mathcal{T}(S^-) \isomap \QFS\,.
\end{equation}
For a fixed $X^- \in \mathcal{T}(S^-)$, the map $\mathit{QF}(\cdot, X^-) \colon \TS \to \QFS$ is called a \emph{horizontal Bers slice}.
Similarly, for $X^+ \in \mathcal{T}(S)$, the map $\mathit{QF}(X^+, \cdot)\colon \mathcal{T}(S^-) \to \QFS$ is called a \emph{vertical Bers slice}.
The quasi-Fuchsian space $\QFS$ is thus foliated by horizontal Bers slices and by vertical Bers slices, and these two foliations are transverse.

After composing with $Z^+ \colon \QFS \to \mathcal{CP}(S)$, a horizontal Bers slice turns into a map $\sigma_{X^-} \colon \TS \to \CPS$ that is a section of $p \colon \CPS \to \TS$, called a \emph{Bers section}. On the other hand, a vertical Bers slice becomes a map
$B_{X^+} \colon \mathcal{T}(S^-) \to \CPS$ that lands in the fixed fiber $p^{-1}(X^+) \subset \CPS$, which is an affine space modeled on $H^0(X, K_{X^+}^2)$  (\autoref{subsubsec:SchwarzianParam}). 
% One defines the \emph{Bers embedding} $b_{X^+}$ using affine subtraction:
% \begin{equation} \label{eq:BersEmbedding}
% \begin{split}
%  b_{X^+} \colon \mathcal{T}(S^-) &\to H^0(X^+, K_{X^+}^2)\\
%  X^- &\mapsto B_{X^+}(X^-) - Z_0
%  \end{split}
% \end{equation}
The \emph{Bers embedding} $b_{X^+} \colon \mathcal{T}(S^-) \to H^0(X^+, K_{X^+}^2)$
is defined as the affine subtraction $b_{X^+}(X^-) = B_{X^+}(X^-) - Z_0$
where $Z_0$ is the standard Fuchsian projective structure $Z_0 = \mathit{QF}(X^+, \overline{X^+})$.
Bers showed that $b_{X^+}$ is a holomorphic embedding of Teichmüller space $\mathcal{T}(S^-) \approx \overline{\TS}$
in $H^0(X, K_{X^+}^2) \approx \C^{3g-3}$, and moreover has bounded image \cite{MR0111835, MR0130972}.

In the special case where $\rho$ is Fuchsian, $\Lambda$ is a round circle and 
$\Omega^+$ and $\Omega^-$ are round disks. After conjugating $\rho$, one can assume that $\Gamma \subset \PSL_2(\R)$, $\Lambda = \R \cup \{\infty\}$,
$\Omega^+ = \H$, and $\Omega^- = \overline{\H}$. Since $\H$ can be identified to the hyperbolic plane
and $\Gamma$ acts freely properly on $\H$ by isometries, the quotient surface $Z^+ = \H/\Gamma$ inherits a hyperbolic structure, in addition to a complex projective structure,
and $\rho$ is the holonomy of both these geometric structures. We have just described
an identification between the Fuchsian deformation space and the Fricke-Klein deformation space of $S$, 
which is the reason why we somewhat abusively use the same
notation $\FS$ for both.
Clearly, the Fuchsian space $\FS \subset \QFS$ arises as the image of the map $X \in \TS \mapsto \mathit{QF}(X,\overline{X})$
called the \emph{Fuchsian slice}. Seen as map from $\TS$ to the Fricke-Klein space $\FS$, this is nothing else than the uniformization map \eqref{eq:uniformizationmap}.
After composing the Fuchsian slice with the map $Z^+ \colon \QFS \to \mathcal{CP}(S)$, one obtains a map $\sigma_0 \colon \TS \to \CPS$ called the \emph{Fuchsian section}.

\section{Applications to Teichmüller theory} \label{sec:TeichmullerApplications}

In this section, we present some applications of the notions developed in sections \ref{sec:LagrangianAffine}, \ref{sec:BiLagrangian}, \ref{sec:BiLagrangianComplexification} to Teichmüller theory, more specifically to the deformation spaces introduced in section \ref{sec:TeichmullerBackground}.

Throughout this section again, $S$ is a connected, oriented, smooth, closed surface of negative Euler characteristic and fundamental group $\pi$.

\subsection{Symplectic structure of deformation spaces} \label{subsec:SymplecticStructureOfDeformationSpaces}

Deformation spaces associated to a closed surface tend to have a natural symplectic structure.
Goldman established the main reason for that:

\begin{theorem}[Goldman, \cite{MR762512}] \label{thm:GoldmanSymp}
Let $G$ be a complex semisimple algebraic group. The character variety ${\cal X}(\pi, G)$ enjoys a natural complex symplectic structure $\omega_G$.
\end{theorem}

\begin{remark}
The $2$-form $\omega_G$ is an algebraic tensor on the character variety, defining in particular a complex symplectic structure on the smooth locus.
If $G$ is merely a real or complex semisimple Lie group, $\omega_G$ is still well-defined as a real or complex symplectic structure on the smooth locus.
 If $G$ is only assumed reductive, the theorem also holds with the trade-off that the symplectic structure is not univocal, its definition requires a choice.
\end{remark}

Let us explain briefly how this symplectic structure is constructed. The Zariski tangent space to the character variety is given by $\upT _{[\rho]} {\cal X}(\pi, G) = H^1(\pi, \mathfrak{g}_{\Ad \circ \rho})$
where $\mathfrak{g}$ is the Lie algebra of $G$. The notation $\mathfrak{g}_{\Ad \circ \rho}$ signifies that $\mathfrak{g}$ is a $\pi$-module via 
$\Ad \circ \rho \colon \pi \to \Aut(\mathfrak{g})$, so it is possible to define the group cohomology of $\pi$ with coefficients in $\mathfrak{g}$.
Let $B$ denote the Killing form of $\mathfrak{g}$, which is nondegenerate because $\mathfrak{g}$ is semisimple.
Given two elements $\alpha$, $\beta \in H^1(\pi, \mathfrak{g}_{\Ad \circ \rho})$, one can take their cup product with $B$ as coefficient pairing to obtain an element
$\alpha \cup_B \beta \in H^2(\pi, \C)$. Since $S$ is a closed oriented surface, $H^2(\pi, \C)$ is isomorphic to the singular cohomology $H^2(S,\C)$ which is isomorphic to $\C$ 
by Poincaré duality. Therefore we have described how to assign a complex number to the pair $(\alpha, \beta)$. This globally
defines an algebraic nondegenerate $2$-form on ${\cal X}(\pi, G)$. It is not easy to show that $\omega_G$ is closed, 
but it follows from a beautiful argument of symplectic reduction due to Atiyah and Bott \cite{MR702806} that Goldman adapted to this setting.

A consequence of Goldman's theorem is that any deformation space of $(X,G)$-structures on $S$ also enjoys a natural symplectic structure: 
just pull back $\omega_G$ from the character variety by the holonomy map. Somewhat abusively, we still denote by $\omega_G$ the resulting symplectic structure. In \cite{MR762512}, Goldman showed that in the case of Fricke-Klein space $\FS$, $\omega_G$ coincides with the Weil-Petersson Kähler form $\omega_{\textnormal{WP}}$ on Teichmüller space. More precisely, the uniformization map
$F \colon (\TS, \omega_{\textnormal{WP}}) \to (\FS, \omega_G)$ is a real symplectomorphism.

In the case of complex projective structures, $\omega_G$ is a complex symplectic structure in $\CPS$. The symplectic geometry of $\CPS$ was carefully studied in 
\cite{MR3352248, MR3323643}. In particular, Loustau proved that the Schwarzian parametrization is a symplectomorphism, improving a result of Kawai \cite{MR1386110}.
More precisely:
\begin{theorem}[Loustau \cite{MR3352248}] \label{thm:LoustauA}
Let $\sigma$ be a section to the projection $p : \CPS \to \TS$ \eqref{eq:ForgetfulProjection}. Let $\tau_\sigma \colon \CPS \to \upT ^*\TS$ denote the isomorphism of affine bundles 
given by the Schwarzian parametrization using $\sigma$ as the ``zero section'' (see \autoref{subsubsec:SchwarzianParam}). The following are equivalent:
\begin{enumerate}[(i)]
 \item $\sigma$ is Lagrangian: $\sigma^* \omega_G = 0$.
 \item $(\tau_\sigma)^* \omega = -i \, \omega_G$, where $\omega$ is the canonical complex
 symplectic structure on $\upT ^*\TS$.
 \item $d(\sigma - \sigma_0) = i \, \omega_{\textnormal{WP}}$, where $\sigma_0 : \TS \to \CPS$ is the Fuchsian section.
\end{enumerate}
\end{theorem}
Moreover, it is known that Bers sections (see \autoref{subsec:FuchsianAndQF}) and their generalizations including Schottky sections satisfy the conditions of this theorem, it is a consequence
of the work of McMullen \cite{MR1745010}, Takhtajan and Teo \cite{MR1997440}, and Loustau \cite{MR3352248}. A consequence \autoref{thm:LoustauA} is that  $p \colon \CPS \to \TS$ is 
a Lagrangian fibration, a fact already known to Goldman \cite{MR2094117}. Steven Kerckhoff also proved part or all of these facts in unpublished work.

\begin{remark}
Other recent activity in this subject are due to Bertola-Korotkin-Norton \cite{MR3735629} and Takhtajan \cite{TakhtajanKawai}.
\end{remark}

\subsection{Complex bi-Lagrangian geometry of quasi-Fuchsian space}

As an open subset in the smooth
locus of $\mathcal{X}(S, \PSL_2(\C))$, the quasi-Fuchsian space $\QFS$ is a complex manifold of dimension $3g-3$. Moreover $\QFS$
comes with a natural complex symplectic structure $\omega_G$, which is just the restriction of Goldman's.
Let us mention that the complex Fenchel-Nielsen coordinates introduced by Kourouniotis
\cite{MR1288062} and Tan \cite{MR1266284} are Darboux coordinates for the complex symplectic structure $\omega_G$ \cite{MR3352248}.

As a complex symplectic manifold, $(\QFS, \omega_G)$ is the complexification of the real-analytic symplectic manifold
$(\TS, \omega_{\textnormal{WP}})$:
\begin{proposition} \label{prop:FuchsianComplexification}
There is a commutative diagram
\begin{equation}
\begin{tikzcd}
\TS \arrow{d}{\Delta} \arrow{r}{F}
& \mathcal{F}(S) \arrow{d}{\textit{hol}} \\
\TS\times \overline{\TS} \arrow{r}{\mathit{QF}}
& \mathcal{QF}(S).
\end{tikzcd}
\end{equation}
where the vertical arrows are complexifications, and the horizontal arrows are (simultaneous) uniformization.  
Moreover, all maps are (complex) symplectomorphisms.
\end{proposition}

\begin{remark}
Given the above proposition, we see that \emph{simultaneous uniformization is the complexification of uniformization}.  
By the uniqueness of complexification, this predicts the existence of a unique biholomorphism 
between a neighborhood of the diagonal in $\TS\times \overline{\TS}$ and a neighborhood of $\FS$ in $\QFS$, independently from
Bers's simultaneous uniformization theorem.
\end{remark}

\begin{proof}
 The commutativity of the diagram is obvious from the discussions in section \ref{subsec:FuchsianAndQF}. 
 The diagonal map $\Delta: \TS \to  \TS\times \overline{\TS}$ is a complexification by definition of the canonical complexification (\autoref{prop:CanonicalComplexification}).  
 Moreover, the construction of the complex symplectic structure on $\TS\times \overline{\TS}$ is via analytic continuation of the K\"{a}hler form of $\TS$ along the diagonal, therefore $\Delta$ is a symplectic embedding by construction.
 
By the commutativity of the diagram, and the fact that the \emph{Fuchsian slice} $QF\circ \Delta: \TS\rightarrow \QFS$ is the composition of a maximal totally real embedding followed by a biholomorphism,
the map $\textnormal{hol}: \FS\rightarrow \QFS$ is a complexification.
It was proved by Goldman in \cite{MR762512}
  (as mentioned in \autoref{subsec:SymplecticStructureOfDeformationSpaces}) that $\textnormal{hol}$ is a symplectic embedding.
  
Finally, since $F: \TS\rightarrow \FS$ is a symplectomorphism, the uniqueness of the complex symplectic structure in the complexification implies that $\mathit{QF}$ is a symplectomorphism.
\end{proof}

Recall that Teichmüller space $\TS$ equipped with the Weil-Petersson metric is a Kähler manifold.
We can thus use \autoref{thm:BiLagComplexificationKahler} to show that $\QFS$, as the complexification of $\TS$, enjoys a natural complex bi-Lagrangian structure.
\autoref{thm:BiLagComplexificationKahler} only predicts that this structure exists in some neighborhood of the Fuchsian slice, 
but we show that it is actually well-defined everywhere in $\QFS$. Moreover, it is remarkable that the complex symplectic structure and 
the two transverse Lagrangian foliations that define the bi-Lagrangian structure
are respectively Goldman's symplectic structure and the two foliations of $\QFS$ by Bers slices:

\begin{theorem} \label{thm:QFBL}
There exists a complex bi-Lagrangian structure $(\omega, {\cal F}_1, {\cal F}_2)$ in $\QFS$ such that:
\begin{enumerate}[(i)]
 \item The complex symplectic structure $\omega$ is equal to Goldman's symplectic structure $\omega_G$.
 \item The complex Lagrangian foliations ${\cal F}_1$ and ${\cal F}_2$ are the horizontal and vertical foliations of $\QFS$ by Bers slices.
\end{enumerate}
Moreover, this bi-Lagrangian structure coincides with the complex bi-Lagrangian structure predicted by \autoref{thm:BiLagComplexificationKahler},
where $\QFS$ is seen as a complexification of Teichmüller space $\TS$ via the Fuchsian slice 
$QF\circ\Delta=\textnormal{hol}\circ F: \TS \rightarrow \QFS$ (cf \autoref{prop:FuchsianComplexification}).
\end{theorem}

\begin{proof}
We have seen that $\QFS$ can be identified holomorphically to $\TS \times \overline{\TS}$
by simultaneous uniformization. Under this identification, the foliations of $\QFS$ by horizontal and vertical Bers slices correspond by definition to the horizontal and vertical foliations
of the product $\TS \times \overline{\TS}$, which are the canonical foliations of the complexification (see \autoref{def:CanonicalFoliations}), 
and we showed in \autoref{thm:BiLagComplexificationKahler}
that these two foliations are the foliations of the complex bi-Lagrangian structure of the complexification. The fact that the complex symplectic structure
coincides with Goldman's was proved in \autoref{prop:FuchsianComplexification}. The fact that the bi-Lagrangian structure exists everywhere
in $\QFS$ is an easy consequence of the fact that Goldman's symplectic structure, as well as the two foliations of $\QFS$ by Bers slices, exist everywhere in $\QFS$.
\end{proof}

An immediate consequence of \autoref{thm:QFBL} is the existence of a natural holomorphic metric in $\QFS$, 
namely the complex bi-Lagrangian metric. As far as we are aware, this holomorphic metric has not been introduced prior to this paper, 
and suggests that the most natural "metric" structure on $\mathcal{QF}(S)$ is neutral pseudo-Riemannian rather than Riemannian.

\begin{theorem} \label{thm:QFHoloMetric}
 There exists a holomorphic metric $g$ on quasi-Fuchsian space $\QFS$ such that:
 \begin{enumerate}[(i)]
  \item \label{theoremQFHoloMetrici} $g$ restricts to $-i g_{\textnormal{WP}}$ on the Fuchsian slice $\FS$,
  where $g_{\textnormal{WP}}$ is the Weil-Petersson metric on Teichmüller space $\TS$, identified to $\FS$ via uniformization. 
  \item \label{theoremQFHoloMetricii} The two foliations of quasi-Fuchsian space by Bers slices are isotropic for $g$, furthermore they are totally geodesic and flat
  for the real and imaginary parts of $g$.
  \item \label{theoremQFHoloMetriciii} The Levi-Civita connection of $g$ parallelizes Goldman's complex symplectic structure $\omega_G$.
  \item \label{theoremQFHoloMetriciv} Let $u$, $v$ be tangent vectors at some point of $\QFS$. Then
     \begin{equation} \label{eq:explicitQFmetric}
  g(u,v) = i ( \langle q_{u_1}, \mu_{v_2} \rangle + \langle q_{v_1}, \mu_{u_2} \rangle )
 \end{equation}
 where $\langle \, , \rangle$ is the duality pairing between $\upT^* \TS$ and $\upT \TS$, and the notations $q_{u_i}$,  $\mu_{v_i}$ are explained right below this theorem.
 \end{enumerate}
\end{theorem}

Before writing the proof of \autoref{thm:QFHoloMetric}, let us introduce the notations that we use in \eqref{eq:explicitQFmetric}. 
First recall that for a tangent vector $u$ to $\QFS$, we write $u = u_1 + u_2$, where $u_1$ (resp. $u_2$) is tangent to the vertical (resp. horizontal) Bers foliation. 
Now observe that:
\begin{itemize}
 \item Since $u_2$ is tangent to a horizontal Bers slice which is a copy of Teichmüller space, one can identify $u_2$ to a tangent vector
to Teichmüller space. We denote it $\mu_{u_2} \in \upT \TS$.
 \item Since $u_1$ is tangent to a vertical Bers slice which can be embedded in a fiber of the bundle $p \colon \CPS \to \TS$ 
 (see \autoref{subsec:FuchsianAndQF}),
 one can identify $u_1$ to a tangent vector to that fiber. Recall that fibers of $p$ are affine spaces modeled on holomorphic quadratic differentials (see
 \autoref{subsec:CPS}), so one can identify $u_1$ to a cotangent vector to Teichmüller space.
 We denote it $q_{u_1} \in \upT^* \TS$.
\end{itemize}

\begin{proof}[Proof of \autoref{thm:QFHoloMetric}]
Since $g$ is the complex bi-Lagrangian metric associated to the complex bi-Lagrangian structure of \autoref{thm:QFBL}, the proof of
\ref{theoremQFHoloMetrici}, \ref{theoremQFHoloMetricii}, and \ref{theoremQFHoloMetriciii} is a direct
application of \autoref{thm:ComplexBiLagConnection}, \autoref{thm:ComplexAndRealBL}, \autoref{thm:BiLagComplexificationKahler}
and \autoref{thm:QFBL}.

Let us now prove \ref{theoremQFHoloMetriciv}. By definition of the complex bi-Lagrangian metric $g$ (cf \autoref{subsec:ComplexBiLagrangianMetricAndConnection}),
 \begin{equation}
 \begin{split}
  g(u,v) &= \omega(F u , v)\\
         &= \omega(u_1 - u_2, v_1 + v_2)
 \end{split}
 \end{equation}
 where $\omega = \omega_G$ is Goldman's symplectic structure. Since both foliations are isotropic for $\omega$, the expression above reduces to 
 $g(u,v) = \omega(u_1 , v_2) - \omega(u_2, v_1)$. The conclusion follows from the general fact that $\omega_G(u, v) = i \langle q_u, \mu_v \rangle$
 whenever $u$ is vertical and $v$ is horizontal. This fact is a consequence of \autoref{thm:LoustauA}, see \cite[Corollary 6.12]{MR3352248} for details.
\end{proof}

\subsection{Weil-Petersson metric on Teichmüller space} \label{subsec:WPredef}

A consequence of \autoref{thm:QFHoloMetric} is that the Weil-Petersson metric can be defined purely through complex
symplectic geometry.
Indeed, it is sufficient to know that quasi-Fuchsian space is a complex bi-Lagrangian manifold (see below) in order to define the Weil-Petersson metric as $-i$ times 
the restriction of the complex bi-Lagrangian metric to the Fuchsian slice. 

This definition of the Weil-Petersson metric is arguably simpler than the classical definition using
the Poincaré metric (see \autoref{subsec:TSWP}), which requires the celebrated but highly nontrivial uniformization theorem\footnote{
On the other hand, we define the Weil-Petersson metric on the Fuchsian space $\FS$ (or the Fricke-Klein deformation space) rather than on Teichmüller space $\TS$. 
Of course, uniformization is required 
if one wishes to identify $\TS$ to $\FS$. Let us mention that there also exists a ``uniformization-free'' definition
of the Weil-Petersson metric on $\FS$ purely in terms of Riemannian geometry, see \cite[\S 2.6]{MR1164870}.}. 

Let us recap
how one argues that $\QFS$ is a complex bi-Lagrangian:
\begin{itemize}
 \item $\QFS$ is a complex symplectic manifold as an open subset of  $\mathcal{X}(\pi, \PSL_2(\C))$, which enjoys
 a complex symplectic structure by Goldman's algebraic construction (\autoref{subsec:SymplecticStructureOfDeformationSpaces}).
 \item The fact that the two transverse foliations of $\QFS$ by horizontal and vertical Bers slices are Lagrangian 
 is part of \autoref{thm:QFBL}, but it can also be seen as a weaker version
 of \autoref{lemma:LagrangianFibration}.
\end{itemize}

\subsection{\texorpdfstring{Affine bundle structure of $\CPS$}{Affine bundle structure of CP(S)}} \label{subsec:AffineCPS}

One of the striking features of the deformation space of complex projective structures $\CPS$ is that 
the ``forgetful projection'' $p \colon \CPS \to \TS$ is a holomorphic affine bundle. 
The affine structure in the fibers of $p^{-1}(X)$ is not obvious: it is constructed from a differential operator called the Schwarzian derivative 
(see \autoref{subsec:CPS}). We prove in \autoref{thm:SchwarzianAffine} below that this affine structure can instead be defined 
easily using symplectic geometry: is just the Bott affine structure in the leaves of a Lagrangian foliation (\autoref{thm:ComplexLagrangianAffine}).

\begin{lemma} \label{lemma:LagrangianFibration}
 The forgetful projection $p \colon \CPS \to \TS$ is a holomorphic Lagrangian fibration.
\end{lemma}

This lemma is a direct consequence of \autoref{thm:LoustauA} (see \cite[Corollary 6.11]{MR3352248}) but it was formerly known: 
it is a consequence of Kawai's work \cite{MR1386110}, and was also proven directly by Goldman \cite{MR2094117}
with a concise algebraic argument.

% \TheoremSchwarzianAffine
\begin{theorem} \label{thm:SchwarzianAffine}
The complex affine structure in the fibers of the projection $p \colon \CPS \to \TS$ constructed via the Schwarzian derivative
coincides with the natural complex affine structure in the leaves of a complex Lagrangian foliation.
\end{theorem}

\begin{proof}
 Choose any complex Lagrangian section $\sigma \colon \TS \to \CPS$, such as a Bers section, 
 and denote by $\tau_\sigma \colon \CPS \to \upT ^*\TS$ the associated isomorphism
 of affine bundles over $\TS$ given by the Schwarzian parametrization using $\sigma$ as the ``zero section''. 
 By definition, the restriction of $\tau^{\sigma}$ to any fiber $p^{-1}(X)$ is an isomorphism of complex affine spaces ${\tau^{\sigma}}\evalat{p^{-1}(X)} \colon p^{-1}(X) \to \upT _X^* \TS$.
 On the other hand, $\tau_\sigma \colon \CPS \to \upT ^*\TS$ is additionally a complex symplectomorphism by \autoref{thm:LoustauA}. In particular it is an isomorphism of complex Lagrangian foliations,
 so that the map ${\tau^{\sigma}}\evalat{p^{-1}(X)}$ identifies the complex affine structures of $p^{-1}(X)$ and $\upT _X^* \TS$
 given by their respective Bott connections.
 By \autoref{thm:BottCotangent}, that the Bott affine structure in $\upT _X^* \TS$ coincides with its affine structure as a vector space. The conclusion follows.
\end{proof}

\subsection{Affine structures on Teichmüller space}

It is well-known that Teichmüller space is a Stein manifold. In fact, any Bers embedding $b_{\overline{X}} \colon \TS \to H^0(\overline{X}, K_{\overline{X}}^2)$ 
(cf \autoref{subsec:FuchsianAndQF}) is a holomorphic embedding of $\TS$ in a complex vector space of the same dimension, and its image is a bounded domain of holomorphy \cite{MR0168800}. 
In particular, Bers embeddings define a family of (incomplete) complex affine structures
on $\TS$ (in the sense of \autoref{def:AffineStructure}), parametrized by $X \in \TS$.
Remarkably, this family of affine structures is precisely the one predicted by \autoref{corollary:CKA}:
\begin{theorem} \label{thm:TSAffine}
 The family of affine structures on $\TS$ parametrized by $X\in \TS$ of \autoref{corollary:CKA} is equal to the family of affine structures on $\TS$ provided by the 
 Bers embeddings $b_{\overline{X}}$, $X \in \TS$.
\end{theorem}
Note that \autoref{corollary:CKA} applies in this situation because the complex symplectic structure $\omega_{\textnormal{WP}}^c = \omega_G$ exists everywhere
in the canonical complexification $\TS \times \overline{\TS} \approx \QFS$. 

\begin{proof}
 The map $Z^+ \colon \QFS \to \CPS$
 defined in \autoref{subsec:FuchsianAndQF}
 is a complex symplectomorphism which sends the vertical foliation of $\QFS$ by Bers slices to (an open subset of) the foliation of $\CPS$ given by the fibers of the projection
 $p \colon \CPS \to \TS$. It follows that the Bott affine structure in the leaves of the vertical foliation of $\QFS$, \ie{} the family of affine structures
 of \autoref{corollary:CKA}, agrees with the Bott affine structure in the fibers of the projection
 $p \colon \CPS \to \TS$. By  \autoref{thm:SchwarzianAffine}, this also coincides with the affine structure in the fibers of the cotangent bundle $\upT^* \TS$, which in turn is
 precisely the family of affine structures given by the Bers embeddings.
\end{proof}

\begin{remark}
Any manifold equipped with an affine structure can be locally embedded in a vector space of the same dimension in an essentially unique way. 
\autoref{thm:TSAffine} could thus be used to define the Bers embeddings themselves, but more work is required to show that they are injective.
\end{remark}

\subsection{Derivative of the Bers embedding at the origin}

Fix a point $X^+ \in \TS$ and consider the Bers embedding $b_{X^+} \colon \mathcal{T}(S^-) \to H^0(X^+, K_{X^+}^2)$ defined in \autoref{subsec:FuchsianAndQF}. 
Its derivativ at a point $X^- \in \mathcal{T}(S^-)$ is a linear map
${\upd b_{X^+}}\evalat{X^-} : \upT _{X^-}\mathcal{T}(S^-) \to H^0(X^+, K_{X^+}^2)$.
% \begin{equation}
% {\upd b_{X^+}}\evalat{X^-} : \upT _{X^-}\mathcal{T}(S^-) \to H^0(X^+, K_{X^+}^2)~.
% \end{equation}
Recall that $X \mapsto \overline{X}$ yields an identification $\mathcal{T}(S^-) \approx \overline{\TS}$,
and that $H^0(X^+, K_{X^+}^2)$ can be identified as the complex cotangent space $\upT _{X^+}^* \TS$. Thus the derivative of the Bers embedding $b_{X^+}$ at 
$X^- = \overline{X^+}$ may be seen as a complex antilinear map
\begin{equation} \label{eq:DerivativeBersEmbedding}
 {\upd b_{X^+}}\evalat{\overline{X^+}} : \upT_{X^+} \TS \to \upT _{X^+}^* \TS~.
\end{equation}

\begin{theorem} \label{thm:DerivativeBersEmbedding}
The derivative of the Bers embedding at the origin \eqref{eq:DerivativeBersEmbedding} is equal to -1/2 times the musical isomorphism induced by the Weil-Petersson metric:
\begin{equation}
\begin{split}
  {\upd b_{X^+}}\evalat{\overline{X^+}} &= -\frac{1}{2} \flat\\
  v &\mapsto -\frac{1}{2} v^\flat :=  {h_{\textnormal{WP}}}\left(\cdot, v\right)\,.
\end{split}
\end{equation}
\end{theorem}

\begin{remark}
\autoref{thm:DerivativeBersEmbedding} is by no means a new result: the derivative of the Bers embedding at the origin is a standard calculation in Teichmüller theory. Nevertheless, we emphasize that we recover the result using symplectic geometry instead of complex analysis.

The classical proof on the following
 \emph{self-reproducing formula} (\cite{MR0192046}, \cite[\S 5.7]{MR903027}).
Let $\varphi$ be a holomorphic function on $\mathbb{H} \subset \C$, denote by $\mathbb{L} \subset \C$ the lower half-plane and
let $\vert \sigma \vert = \frac{\vert \upd w \vert^2}{\Imag(w)^2}$ denote the area density of the Poincaré metric on $\mathbb{L}$. Then 
\begin{equation}
 \varphi(z) = \frac{12}{\pi} \int_{\mathbb{L}} \frac{\varphi(\overline{w})}{(w-z)^4} \vert \sigma \vert~.
\end{equation}
An concise exposition of this calculation is in \cite[proof of Theorem 7.1]{MR1745010}.
We thank Curtis McMullen for pointing this out to us and the fact that a stronger result is proved in \cite{MR1047763}.
\end{remark}

\begin{proof}[Proof of \autoref{thm:DerivativeBersEmbedding}]
 Let $\nu$ be a tangent vector to $\TS$ at $X^+$, denote by $\overline{\nu}$ the corresponding tangent vector to $\mathcal{T}(S^-)$ at $X^- = \overline{X^+}$. 
 We need to show that for every $\mu \in \upT_{X^+} \TS$,
 \begin{equation}
 \langle (b_{X^+})_* \overline{\nu}, \mu \rangle = -\frac{1}{2} h_{\textnormal{WP}}(\mu, \nu)~.
 \end{equation}
 Let $\rho = \mathit{QF}(X^+, X^-) \in \FS \subset \QFS$ and consider the following tangent vectors:
 \begin{equation}
 \begin{aligned}
  a &= (\mu, \overline{\mu}) \in \upT_{(X^+, X^-)} \mathcal{T}(S^+) \times \mathcal{T}(S^-)\\
  u &= (\mathit{QF})_* a \in \upT_\rho \FS \subset \upT_\rho \QFS
 \end{aligned}\qquad
 \begin{aligned}
  b &= (\nu, \overline{\nu}) \in \upT_{(X^+, X^-)} \mathcal{T}(S^+) \times \mathcal{T}(S^-)\\
  v &= (\mathit{QF})_* b \in \upT_\rho \FS \subset \upT_\rho \QFS~.
 \end{aligned}
 \end{equation}
Using the notations introduced below \autoref{thm:QFHoloMetric}, observe that we have by construction:
\begin{equation}
\begin{aligned}
 u_1 &= (\mathit{QF})_* (0, \overline{\mu})\\
 u_2 &= (\mathit{QF})_* (\mu, 0)
\end{aligned} \qquad
\begin{aligned}
 q_{u_1} &=  (b_{X^+})_* \overline{\mu}\\
 \mu_{u_2} &= \mu
\end{aligned} \qquad
\begin{aligned}
 v_1 &= (\mathit{QF})_* (0, \overline{\nu})\\
 v_2 &= (\mathit{QF})_* (\nu, 0)
\end{aligned} \qquad
\begin{aligned}
 q_{v_1} &=  (b_{X^+})_* \overline{\nu}\\
 \mu_{v_2} &= \nu
\end{aligned}
\end{equation}
By \autoref{thm:QFHoloMetric} \ref{theoremQFHoloMetriciv}, one can thus express $g(u,v)$ as:
\begin{equation} \label{eq:guv1}
 g(u,v) = i ( \langle (b_{X^+})_* \overline{\mu}, \nu \rangle + \langle (b_{X^+})_* \overline{\nu}, \mu \rangle )~.
\end{equation}
On the other hand, since $u$ and $v$ are both tangent to $\FS \subset \QFS$ and \autoref{thm:QFHoloMetric} \ref{theoremQFHoloMetrici} guarantees that $g$ restricts 
to $-i g_{\textnormal{WP}}$ on $\FS$, we have
\begin{equation} \label{eq:guv2}
 g(u,v) = -i g_{\textnormal{WP}}(\mu, \nu)~.
\end{equation}
Equating \eqref{eq:guv1} and \eqref{eq:guv2} yields:
\begin{equation} \label{eq:gwpuv}
  g_{\textnormal{WP}}(\mu, \nu) = - \langle (b_{X^+})_* \overline{\mu}, \nu \rangle - \langle (b_{X^+})_* \overline{\nu}, \mu \rangle~.
\end{equation}
The conclusion follows, recalling that $h_{\textnormal{WP}} = g_{\textnormal{WP}} - i \omega_{\textnormal{WP}}$ and 
$\omega_{\textnormal{WP}}(\mu, \nu) = g_{\textnormal{WP}}(i\mu, \nu)$.
\end{proof}

\section{Hyper-Hermitian structure in the complexification of a Kähler manifold} \label{sec:HyperKahler}

In this final section, we construct a natural almost hyper-Hermitian structure in the complexification of a real-analytic Kähler manifold,
relying on the canonical complex bi-Lagrangian structure constructed in section \ref{sec:BiLagrangianComplexification}. This almost hyper-Hermitian structure
is unfortunately not integrable in general; in particular it is not the same as the Feix-Kaledin hyper-Kähler structure in the cotangent bundle
of a real-analytic Kähler manifold which we review in \autoref{subsec:FeixKaledin}.

\subsection{Almost hyper-Hermitian structures}

The algebra of \emph{quaternions} $\H$ is the unital associative algebra over the real numbers generated by three elements $i$, $j$, and $k$ satisfying
the \emph{quaternionic relations}:
\begin{equation} \label{eq:QuaternionicRelations}
  \begin{gathered}
    i^2 = j^2 = k^2 = -1\\
    ij = -ji = k
  \end{gathered}
\end{equation}
$\H$ is a $4$-dimensional algebra over $\R$: a generic quaternion is written $q = a + ib + jc + kd$ with $(a,b,c,d) \in \R^4$.
We refer to \cite[A.2]{LoustauAsheville} for basic notions of quaternionic linear algebra relevant for quaternionic differential geometry.

\begin{definition} \label{def:AlmostHH}
 An \emph{almost hyper-Hermitian} structure on $M$ is the data of $(g, I, J, K)$ where:
 \begin{itemize}
  \item $g$ is a Riemannian metric on $M$.
  \item $I$, $J$, and $K$ are three almost complex structures on $M$ which are compatible with $g$ (\textit{i.e.} $g$-orthogonal as endomorphisms of $\upT M$)
  and satisfy the quaternionic relations \eqref{eq:QuaternionicRelations}.
 \end{itemize}
 If moreover $I$, $J$, and $K$ are parallel for the Levi-Civita connection of $g$ (\emph{integrability condition}) 
 then $(g, I, J, K)$ is called a \emph{hyper-Kähler} structure.
\end{definition}

Given an almost hyper-Hermitian manifold $(M, g, I, J, K)$, we denote by $\omega_I$, $\omega_J$, and $\omega_K$ the three $2$-forms on $M$
defined by $\omega_I = g(I \cdot, \cdot)$, $\omega_J = g(J \cdot, \cdot)$, and $\omega_K = g(K \cdot, \cdot)$.

\begin{remark}
By definition, an almost hyper-Hermitian structure equips a Riemannian manifold $(M,g)$ with an isometric action of the algebra of quaternions $\H$ in its tangent bundle $\upT M$.
 In the language of $G$-structures, an almost hyper-Hermitian structure is equivalent to a $\Sp(m)$-structure, where $\Sp(m) = \U(m, \mathbb{H})$ is the quaternionic unitary group.
 \end{remark}
 
\begin{remark}
Some authors use \emph{almost hyper-Kähler} instead of \emph{almost hyper-Hermitian}, \eg{} Joyce \cite{MR1787733} and Bryant \cite{MR1338391}.
It seems more consistent to us to call \emph{almost hyper-Kähler} an almost hyper-Hermitian structure whose three Kähler forms
$\omega_I$, $\omega_J$, $\omega_K$ are closed, although almost hyper-Kähler structures
in this sense are in fact always hyper-Kähler (\cite[Lemma 6.8]{MR887284}).
\end{remark}

\subsection{The Feix-Kaledin hyper-Kähler structure} \label{subsec:FeixKaledin}

Feix and Kaledin independently discovered a canonical hyper-Kähler structure in a neighborhood of the zero section of the holomorphic cotangent bundle of a real-analytic Kähler
manifold:

\begin{theorem}[Feix \cite{FeixThesis,MR1817502}, Kaledin \cite{MR1815021, MR1848662}] \label{thm:FeixKaledin}
Let $(N, g_0, I_0, \omega_0)$ be a real-analytic Kähler manifold. There exists a unique\footnote{Kaledin \cite{MR1848662} states 
that the hyper-Kähler structure is only unique up to symplectic fiber-wise automorphisms of $\upT^*$N. 
However, unless we are mistaken, there are no such automorphisms besides the identity map.} hyper-Kähler structure $(g, I, J, K)$ in a neighborhood of the zero section in $\upT^* N$ 
such that:
\begin{enumerate}[(i)]
 \item $g$ restricts to $g_0$ along the zero section.
 \item $I$ is the usual almost complex structure in $\upT^* N$.
 \item $\omega_J + i \omega_K = \omega$ is the canonical complex symplectic form in $\upT^*N$ (\autoref{ex:CotangentBundle})
 \item The $\U(1)$-action in $\upT^*N$ by multiplication in the fibers is $g$-orthogonal.
\end{enumerate}
\end{theorem}

In the case where $N = \CP^n$ with the Fubini-Study Kähler structure, the Feix-Kaledin metric coincides with the Calabi metric \cite{MR543218}, which is a complete hyper-Kähler metric
defined everywhere in $\upT^*N$. The Eguchi-Hanson metric described in \autoref{subsubsec:FKCP1} is a special case of the Calabi metric when $n=1$.
On the other hand, Feix shows that when $N$ is a compact Riemann surface of genus $>1$, the hyper-Kähler metric in a neighborhood of the zero section of $\upT^*N$
cannot be extended everywhere and is incomplete (\cite[Example 5.14]{FeixThesis}).

The proofs of Feix and Kaledin are both difficult and are very different in nature. Let us make a couple of heuristic comments about the proof of Feix \cite{FeixThesis}
because it is relevant to this paper. Feix does not define the hyper-Kähler structure on $\upT^*N$ directly, instead she constructs the \emph{twistor space} of the hyper-Kähler structure,
which is a holomorphic object encoding the hyper-Kähler data (see \cite{MR877637}). Interestingly, she constructs the twistor space 
for the complexification $N^c$ rather than $\upT^*N$\footnote{Feix then recovers $\upT^*N$ by considering the fiber
over $0$ instead of the fiber over $1$ in the twistor space $Z \to \CP^1$, we refer to \cite{FeixThesis} for details.}. The bi-Lagrangian structure of $N^c$ (see section 
\ref{sec:BiLagrangianComplexification}) is key in the in her construction, especially the Bott affine structure in the leaves of the Lagrangian foliations, 
though she does not use this terminology. This initially led us to believe that our construction would recover the same hyper-Kähler structure,
but we soon realized that was not the case (\autoref{subsec:CP1HK}). 

In this paper, we are interested in hyper-Kähler structures in the complexification $N^c$ of a Kähler manifold rather than in the cotangent bundle $\upT^*N$.
As we mentioned above, Feix defines a hyper-Kähler structure in $N^c$, but no theorem is given to characterize the existence and uniqueness of such a structure.
There is however at least one proper way to define a ``canonical'' hyper-Kähler structure in the complexification $N^c$:

\begin{theorem} \label{thm:StraightMap}
Let $(N, g_0, I_0, \omega_0)$ be a real-analytic Kähler manifold. Let $(\tilde{g}, \tilde{I}, \tilde{J}, \tilde{K})$ denote the Feix-Kaledin hyper-Kähler structure 
in a neighborhood $V$ of the zero section of the cotangent bundle $\upT^*N$.
Let $f \colon N \to M$ be a complexification of $N$. \newline
There exists a unique hyper-Kähler structure $(g, I, J, K)$ in a sufficiently small neighborhood $U$ of $f(N)$ in $M$ such that
there exists an embedding $T \colon U \to V$ (necessarily unique) such that:
\begin{enumerate}[(i)]
 \item $T$ is an isomorphism of hyper-Kähler structures: $T^* (\tilde{g}, \tilde{I}, \tilde{J}, \tilde{K}) = (g, I, J, K)$.
 \item $J$ is the almost complex structure of $M$.
 \item $T \circ f$ is the zero section $N \to \upT^* N$.
\end{enumerate} 
\end{theorem}

\begin{proof}
The zero section of the cotangent bundle is complex Lagrangian for the canonical complex symplectic structure $\omega$, \textit{a fortiori} it is real Lagrangian for 
$\omega_{\tilde{J}} = \Real(\omega)$ (cf \autoref{thm:FeixKaledin}). Since $\omega_{\tilde{J}} = \tilde{g}(\tilde{J} \cdot, \cdot)$, it follows that $\tilde{J}$ sends the tangent
space to the zero section to its $\tilde{g}$-orthogonal complement. In particular, $\upT^*N$ equipped with the complex structure $\tilde{J}$
is a complexification of the zero section. By uniqueness of complexification (cf \autoref{thm:Complexification}), there exists a unique map $T \colon U \to V$ (for $U$ and $V$ sufficiently small)
such that $T$ is holomorphic as map $(U, J) \to (V, \tilde{J})$, and $T \circ f$ is the zero section $N \to \upT^* N$.
Now just take $g = T^* \tilde{g}$, $I = T^*\tilde{I}$, and $K = T^* \tilde{K}$.
\end{proof}

\begin{definition} \label{def:FKComplexification}
 We shall call $(g, I, J, K)$ the \emph{Feix-Kaledin hyper-Kähler structure in the complexification $M$} (defined in a neighborhood of $N \hookrightarrow M$).
\end{definition}

We would like to characterize the Feix-Kaledin hyper-Kähler structure in the complexification as the unique hyper-Kähler \emph{admissible extension} of the Kähler structure 
in the following sense.

\begin{definition} \label{def:AdmissibleExtension}
 Let $(N, g_0, I_0, \omega_0)$ be a real-analytic Kähler manifold. An almost hyper-Hermitian structure $(g, I, J, K)$ in a neighborhood of $N$ 
 in a complexification $M$ is called an \emph{admissible extension} of the Kähler metric off $N$ when:
 \begin{enumerate}[(i)]
 \item $J$ is the almost complex structure of $M$.
 \item $g$, $I$, and $\omega_I$ extend $g_0$, $I_0$, and $\omega_0$ respectively.
 \item $\omega_I - i\omega_K$ is the complexification $\omega_0^c$ of $\omega_0$.
 \end{enumerate}
\end{definition}

The next theorem is another straightforward application of analytic continuation.
\begin{theorem}
 The Feix-Kaledin hyper-Kähler structure in a neighborhood of $N$ in a complexification $M$ (cf \autoref{def:FKComplexification}) is an admissible extension of the Kähler structure
 off $N$.
\end{theorem}
In the next subsection, we construct another almost hyper-Hermitian admissible extension, but it is typically not integrable.
It is unclear to us whether uniqueness of the admissible extension holds amongst hyper-Kähler structures. 
Note that in the tangent space at points of $N \hookrightarrow N^c$, it is easy to check that uniqueness does indeed hold:
it is just linear algebra (\autoref{lem:linearHKcomplexification}).

\begin{question} \label{qu:FKComp}
Is the Feix-Kaledin hyper-Kähler structure in a sufficiently small neighborhood of $N$ in a complexification $M$ the unique admissible hyper-Kähler extension?
\end{question}

\subsection{Construction of the almost hyper-Hermitian structure} \label{subsec:HHconstruction}

Let $(N, g_0, I_0, \omega_0)$ be a real-analytic Kähler manifold. Let $M$ be a complexification of $N$, we can assume without loss of generality that $M = N^c = N \times \overline{ N}$
is the canonical complexification (cf \autoref{subsec:ComplexificationComplex}). We are going to construct an almost hyper-Hermitian structure in a neighborhood of $N$ in $N^c$ 
which is admissible in the sense of 
\autoref{def:AdmissibleExtension}. We start by pointing out that the tensors $(g, I, J, K)$ are uniquely defined 
in the tangent spaces at points of $N$ inside $N^c$ by the following lemma.
\begin{lemma} \label{lem:linearHKcomplexification}
Let $(V, g_0, I_0, \omega_0)$ be a real vector space with a linear Hermitian structure. Denote by $V^c := V \times \overline{V}$ the complexification of $(V, I_0)$,
\textit{i.e.} the real vector space $V \times V$ equipped with the linear complex structure $(I_0, -I_0)$.
There exists a unique linear hyper-Hermitian structure $(g, I, J, K)$ in $V^c$ which is admissible in the sense of \autoref{def:AdmissibleExtension}.
\end{lemma}
The definitions of a linear Hermitian structure and linear hyper-Hermitian structure should be obvious so we do not repeat them, but they may be found in \eg{} \cite[A.2]{LoustauAsheville}.
The proof of \autoref{lem:linearHKcomplexification} is elementary and spared to the reader.

Recall that $N^c$ enjoys a canonical pair of transverse foliations $(\mathcal{F}_1, \mathcal{F}_2)$ (cf \autoref{subsec:ComplexificationComplex})
and a canonical complex bi-Lagrangian structure (cf \autoref{thm:BiLagComplexificationKahler}). We are now ready to state the main theorem of this section.
\begin{theorem} \label{thm:MainThmHK}
There exists a unique almost hyper-Hermitian structure $(g, I, J, K)$ in a sufficiently small neighborhood of $N$ in $N^c$ such that:
\begin{enumerate}[(i)]
 \item $(g, I, J, K)$ is admissible in the sense of \autoref{def:AdmissibleExtension}.
 \item $(g, I, J, K)$ is parallel with respect to the complex bi-Lagrangian connection of $N^c$ (\autoref{thm:BicomplexInComplexification})
 along the foliation $\mathcal{F}_1$.
\end{enumerate}
Moreover, if $N$ is simply connected and $\omega_0$ extends holomorphically throughout $N^c$, then the almost hyper-Hermitian structure
$(g, I, J, K)$ exists throughout $N^c$.
\end{theorem}

\autoref{thm:MainThmHK} follows fairly easily \autoref{cor:PathIndependence} and \autoref{thm:ComplexBiLagConnection}:
\begin{proof}
Let $U$ be a neighborhood of $N$ in $N^c$ sufficiently small so that:
\begin{itemize}
 \item The leaves of $\mathcal{F}_1$ are simply connected.
 \item The complex symplectic form $\omega = \omega_0^c$ extending $\omega_0$ is well-defined.
\end{itemize}
In particular, the canonical complex bi-Lagrangian structure $(\omega, \mathcal{F}_1, \mathcal{F}_2)$ is well-defined.
Since the tensor fields $g$, $I$, $J$, and $K$ must be parallel with respect to the complex bi-Lagrangian connection $\nabla$ along the vertical foliation $\mathcal{F}_1$, 
they are completely determined by their values at one point that can be chosen freely in every leaf: their values everywhere else can be obtained by parallel transport along paths
contained in the leaves. However the values of $g$, $I$, $J$, and $K$ are uniquely defined at points $p \in N  \subset U$ by \autoref{lem:linearHKcomplexification}. Thus we have showed uniqueness.
Conversely, parallel transport along vertical paths of the tensors $g_p$, $I_p$, $J_p$, and $K_p$ at points $p \in N$ yields tensor fields
$g$, $I$, $J$, $K$ that are well-defined in $U$ by virtue of \autoref{cor:PathIndependence} and the simple-connectedness of the leaves.
Of course, the linear algebraic identities verified by $g_p$, $I_p$, $J_p$, and $K_p$ are preserved by parallel transport, so that $(g, I, J, K)$ is a almost hyper-Hermitian structure in $U$.
It remains to argue that it is admissible in the sense of \autoref{def:AdmissibleExtension}:
 \begin{enumerate}[(i)]
 \item $J$ is the almost complex structure of $U \subset N^c$: this is because we know that the almost complex structure of $N^c$ is parallel with respect to the bi-Lagrangian connection
 by \autoref{thm:ComplexBiLagConnection}. 
 \item $g$, $I$, and $\omega_I$ extend $g_0$, $I_0$, and $\omega_0$ respectively: this is clearly the case by construction.
 \item $\omega_I - i\omega_K$ is the complexification $\omega = \omega_0^c$ of $\omega_0$: this is because we know 
 that $\omega$ is parallel with respect to the bi-Lagrangian connection by \autoref{thm:ComplexBiLagConnection}, therefore the identity $\omega = \omega_I - i\omega_K$ which holds
 at points $p \in N  \subset U$ is preserved under parallel transport.
 \end{enumerate}
\end{proof}

\begin{remark}
The almost hyper-Hermitian structure is typically not integrable (\autoref{subsubsec:CP1AlmostHH}), therefore it is not the same as
the Feix-Kaledin hyper-Kähler structure. Note however that it is ``2/3 integrable'' in the sense that $\omega_I$ and $\omega_K$
are both closed, since $\omega = \omega_I - i \omega_K$ is the complex symplectic structure of $N^c$. We hypothesize that the flatness of 
the Kähler metric $g_0$ is necessary and sufficient for the integrability
of the almost hyper-Hermitian structure, in which case it is equal to Feix-Kaledin.
\end{remark}

\subsection{The biquaternionic structure}

Our construction actually yields more than an almost hyper-Hermitian structure. Indeed, the almost hyper-Hermitian structure 
of \autoref{thm:MainThmHK} is compatible with the bicomplex Kähler structure associated 
to the complex bi-Lagrangian structure (cf \autoref{subsec:BicomplexKahlerInComplexification}). The resulting ``package''  is an almost \emph{biquaternionic Hermitian structure}, let us explain
this in what follows.

The algebra of \emph{(ordinary) biquaternions} $\mathbb{BH}$ is the unital associative algebra over the real numbers generated by four elements $h$, $i$, $j$, $k$ satisfying
the \emph{biquaternionic relations}:
\begin{equation} \label{eq:BiquaternionicRelations}
  \begin{gathered}
    h^2 = i^2 = j^2 = k^2 = -1\\
    ij = -ji = k\\
    hi = ih \quad hj = jh \quad hk = kh~.\\
  \end{gathered}
\end{equation}
One quickly sees that $\mathbb{BH}$ is an $8$-dimensional algebra over $\R$ which can be simply be described as $\H \otimes_\R \C$ by writing 
a generic biquaternion $q = q_1 + q_2 h$, where $q_1 = a_1 + i b_1 + j c_1 + k d_1$ and $q_2 = a_2 + i b_2 + j c_2 + k d_2$ are quaternions.

Let $(N, g_0, I_0, \omega_0)$ be a real-analytic Kähler manifold and let $M$ be a complexification of $N$. We denote by $(g, I, J, K)$ the almost hyper-Hermitian structure
of \autoref{thm:MainThmHK}. On the other hand, let $H := I_0^c$ denote the holomorphic extension of $I_0$. Recall that the triple $(H, J, F:= HJ)$ is 
the holomorphic bicomplex structure in (a neighborhood of $N$ in) $M$
associated to the canonical pair of foliations $(\mathcal{F}_1$, $\mathcal{F}_2)$ by \autoref{thm:BicomplexInComplexification}.
Mind that we denote by $H$ instead of $I$ the almost complex structure $I_0^c$ because $I$ is now a different almost complex structure, and the choice of the letter $H$
is motivated by the following theorem.

\begin{theorem} \label{thm:BiQuaternionicHermitian}
Let $(N, g_0, I_0, \omega_0)$ be a real-analytic Kähler manifold and let $M$ be a complexification of $N$.
The Riemannian metric $g$ and the quadruple of almost (para-)complex structures $(H, I, J, K)$ as above define an \emph{almost biquaternionic Hermitian structure} in (a neighborhood of $N$ in) $M$
in the sense that:
 \begin{itemize}
  \item $H$, $I$, $J$, $K$ satisfy the biquaternionic relations \eqref{eq:BiquaternionicRelations}.
  \item $H$, $I$, $J$, $K$ are compatible with $g$ (\textit{i.e.} $g$-orthogonal as endomorphisms of $\upT M$).
 \end{itemize}
\end{theorem}
\begin{proof}
Let us only sketch the proof of this theorem for brevity, as we have previously detailed similar arguments. The fact that $H$, $I$, $J$, $K$ satisfy the biquaternionic relations 
and are compatible with $g$ can be checked directly at points of $N$ inside $M$ because all these tensors are explicit there--it is ``just linear algebra''. Then we argue that these linear identities
are satisfied everywhere because they are preserved by parallel transport. Indeed, the tensor fields $g$, $H$, $I$, $J$, $K$ are all invariant 
by parallel transport along the vertical foliation $\mathcal{F}_1$ with respect to the complex bi-Lagrangian connection. For $g$, $I$, $J$, $K$, this is true by \autoref{thm:MainThmHK}.
For $H$, this is because the triple $(g, F= HJ, \omega)$ is the holomorphic bicomplex Kähler structure associated to the canonical bi-Lagrangian structure
$(\omega, \mathcal{F}_1$, $\mathcal{F}_2)$ by \autoref{thm:BicomplexKahlerInComplexification}, it follows that $H$ is parallel with respect to the bi-Lagrangian connection.
\end{proof}

Observe that the algebra of biquaternions $\mathbb{BH}$ contains as subalgebras:
\begin{itemize}
 \item The algebra of quaternions $\mathbb{H} \approx \operatorname{span}\{1, i, j, k\}$.
 \item The algebra of bicomplex numbers $\mathbb{BC} \approx \operatorname{span}\{1, h, j, f\}$ where $f := hj$.
 \item The algebras of \emph{para-quaternions} $\mathbb{P} \approx \operatorname{span}\{1, i, f, g\}$ where $f := hj$ and $g := hk$. This is the unital associative algebra
 over the real numbers generated by three elements $i$, $f$, $g$ satisfying the \emph{para-quaternionic relations}:
\begin{equation}
\begin{gathered}
 i^2 = -1 \quad f^2 = +1 \quad g^2 = +1\\
 if = -fi = g~.
\end{gathered}
\end{equation}
\noindent Para-quaternions are also called \emph{split-quaternions}, \emph{coquaternions}, and \emph{quaternions of the second kind} in the context
of differential geometry \cite{MR0046735, MR0079766}.
\end{itemize}

Corresponding to these subalgebras we have an almost hyper-Hermitian structure, an almost bicomplex Hermitian structure and an \emph{almost para-hyper-Hermitian structure} on $M$.
We have already established the presence of the first two: they are respectively the almost hyper-Hermitian structure of \autoref{thm:MainThmHK} and the holomorphic bicomplex Kähler structure
of \autoref{thm:BicomplexKahlerInComplexification}. The third one also deserves to be noted:

\begin{corollary} \label{cor:ParaHH}
Let $(N, g_0, I_0, \omega_0)$ be a real-analytic Kähler manifold and let $M$ be a complexification of $N$.
The biquaternionic Hermitian structure $(g, H, I, J, K)$ of \autoref{thm:BiQuaternionicHermitian} induces an 
underlying \emph{almost para-hyper-Hermitian structure} $(g, I, G, F)$ where $F := HJ$ and $G := HK$.
\end{corollary}
Para-hyper-Hermitian structures have been studied by several authors under a variety of names such as \emph{para-hyperhermitian} and \emph{para-hyperkähler} \cite{AMTParaK, MR2554487, MR2868386}, 
\emph{hyper-para-Kähler} \cite{MR2104536, Merker2012}, \emph{hypersymplectic}--a terminology coined
by Hitchin \cite{HitchinHyperSymplectic} and followed by many authors, \emph{pseudo-hyperKähler}
\cite{MR2436230, MR2775123}, and \emph{neutral hyperKähler} \cite{MR1715481}.

\subsection{\texorpdfstring{Example: $\CP^1$}{Example: CP1}} \label{subsec:CP1HK}

Let us compute and compare the Feix-Kaledin hyper-Kähler structure and the almost hyper-Hermitian structure of \autoref{thm:MainThmHK}
in the cotangent bundle and complexification of $\CP^1$. We continue using the same notations as in \autoref{subsec:CP1BL}.

\subsubsection{Feix-Kaledin hyper-Kähler structure} \label{subsubsec:FKCP1}

The Feix-Kaledin hyper-Kähler structure in the total space of the cotangent bundle of $\CP^1$ coincides with the metric discovered by Eguchi-Hanson \cite{MR540896}
and is a special case of \emph{Calabi metric} studied in \cite{MR543218}. However it is typically not expressed in the standard holomorphic coordinates of the cotangent bundle
in the mathematics and physics literature. Let us find its expression in such coordinates, it will be instructive to see how to recover it only using the conditions of \autoref{thm:FeixKaledin}.

Let $z$ be the usual complex coordinate in the affine patch $\C = \CP - \{[1 : 0]\}$. 
Let $(z,u)$ be the corresponding coordinates on the holomorphic cotangent bundle
$M = \upT^*\CP$: a covector $\alpha \in \upT^*\C$ having coordinates $(z,u)$ means that $\alpha = u \,dz\evalat{z}$.

\begin{lemma} \label{lem:MAHK}
The Feix-Kaledin hyper-Kähler structure $(g, I, J, K)$ in the cotangent bundle $M = \upT^*N$ of a Riemann surface $N$ satisfies 
$\omega \wedge \overline{\omega} = 2 \omega_I \wedge \omega_I$, where
$\omega$ is the canonical complex symplectic structure in $M = \upT^* N$.
\end{lemma}
\begin{proof}
In any Kähler (in fact Hermitian) manifold $(M, g, I, \omega_I)$ of complex dimension $n$, the identity $\omega_I^n = n! \vol_g$ holds. Thus we have here 
\begin{equation}
 \omega_I \wedge \omega_I = \omega_J \wedge \omega_J = \omega_K \wedge \omega_K = 2 \vol_g~.
\end{equation}
On the other hand, since $\omega = \omega_J + i\omega_K$ (this is a requirement of the Feix-Kaledin hyper-Kähler structure), we have:
\begin{equation}
 \omega \wedge \overline{\omega} = \omega_J \wedge \omega_J + \omega_K \wedge \omega_K = 4 \vol_g~.
\end{equation}
The conclusion follows.
\end{proof}
Note that more generally, a similar argument shows that a hyper-Kähler manifold admits a holomorphic volume form, and thus is \emph{Calabi-Yau} (see \eg{} \cite{YauScholarpedia} for details).
Coming back to $M = \upT^*\CP^1$, \autoref{lem:MAHK} means that a \emph{Kähler potential} 
$\varphi \colon M \to \R$ (such that $\frac{i}{2}\partial \overline{\partial} \varphi = \omega_I$) 
must satisfy the \emph{Monge-Ampère equation}
\begin{equation} \label{eq:MA}
\varphi_{z \bar{z}} \varphi_{u \bar{u}} - \varphi_{u \bar{z}} \varphi_{z \bar{u}} = 1~.
\end{equation}
Let $r \colon M \to \mathbb{\R}$ be the function given by 
$r(\alpha) = \Vert \alpha \Vert^2$, where $\Vert \cdot \Vert^2$ is the metric in the cotangent bundle induced by $g_0$, explicitly:
$ r(z,u) = 4 u \overline{u}{(1+ z \overline{z})}^2$.
Since the Feix-Kaledin hyper-Kähler metric $g$ is invariant under the $\U(1)$-action in $M = \upT^*M$ which acts transitively in the level sets of $r$, we 
can look for a Kähler potential of the form $\varphi(z, u) = y (r)$, where $y \colon [0, +\infty) \to \R$. The Monge-Ampère equation \eqref{eq:MA} yields the
ordinary differential equation
\begin{equation}
8 r^2 y' y'' + 8 r (y')^2 = 1~.
\end{equation}
This ODE is easily solved as it is rewritten $4 (Y^2)' = 1$ where $Y = ry'$. The general solution up to an additive constant is 
\begin{equation}
 y = \sqrt{r + a} - \sqrt{a} \operatorname{arcoth}\left(\sqrt{1 + \frac{r}{a}}\right)
\end{equation}
where $a$ is some constant of integration, which we choose $a = 1$ in order to recover the Fubini-Study metric $g_0$ on the zero section at the end of the calculation.
We can now proceed to compute the Kähler form $\omega_I = \frac{i}{2} \partial \overline{\partial} \varphi$, the metric $g = -\omega_I(I\cdot, \cdot)$, etc.
\begin{proposition}
The Feix-Kaledin hyper-Kähler structure in $\upT^*\CP^1$ in the coordinates $(z,u)$ is:
\begin{equation}
\begin{split}
g &= \varphi_{z \bar{z}}\, \dz \dzbar + \varphi_{u \bar{z}}\, \upd u \dzbar + \varphi_{z \bar{u}}\, \dz \dubar + \varphi_{u \bar{u}}\, \upd u \dubar  \\
I &= i \dd{z} \otimes \dz - i \dd{\overline{z}} \otimes \dzbar + i \dd{u} \otimes \du - i \dd{\overline{u}} \otimes \dubar\\
J &=  \left(\varphi_{z \bar{u}} \dd{\overline{z}} - \varphi_{z \bar{z}} \dd{\overline{u}} \right) \otimes \dz
      + \left(\varphi_{u \bar{z}} \dd{z} - \varphi_{z \bar{z}} \dd{u} \right) \otimes \dzbar \\
      &\quad + \left(\varphi_{u \bar{u}} \dd{\overline{z}} - \varphi_{u \bar{z}} \dd{\overline{u}} \right) \otimes \du 
      + \left(\varphi_{u \bar{u}} \dd{z} - \varphi_{z \bar{u}} \dd{u} \right) \otimes \dubar\\
K &= -i\left(\varphi_{z \bar{u}} \dd{\overline{z}} - \varphi_{z \bar{z}} \dd{\overline{u}} \right) \otimes \dz
      + i\left(\varphi_{u \bar{z}} \dd{z} - \varphi_{z \bar{z}} \dd{u} \right) \otimes \dzbar \\
      &\quad  -i\left(\varphi_{u \bar{u}} \dd{\overline{z}} - \varphi_{u \bar{z}} \dd{\overline{u}} \right) \otimes \du 
      + i\left(\varphi_{u \bar{u}} \dd{z} - \varphi_{z \bar{u}} \dd{u} \right) \otimes \dubar\\
\omega_I &= \frac{i}{2}\left(\varphi_{z \bar{z}}\, \dz \wedge \dzbar + \varphi_{u \bar{z}}\, \upd u \wedge \dzbar 
+ \varphi_{z \bar{u}}\, \dz \wedge \dubar + \varphi_{u \bar{u}}\, \upd u \wedge \dubar \right) \\
\omega_J &= -\frac{1}{2}\left(\dz \wedge \du + \dzbar \wedge \dubar\right)\\
\omega_K &= \frac{i}{2}\left(\dz \wedge \du - \dzbar \wedge \dubar\right)\\
\end{split}
\end{equation}
where:
\begin{equation}
\begin{aligned}
 \varphi_{z \bar{z}} &= \frac{1 + r(1+ z \overline{z})}{\sqrt{1+r}(1 + z \overline{z})^2}\\
 \varphi_{u \bar{z}} &= \frac{2\overline{u}z(1+ z \overline{z})}{\sqrt{1+r}}
\end{aligned}\qquad
\begin{aligned}
 \varphi_{z \bar{u}} &= \frac{2u\overline{z}(1+ z \overline{z})}{\sqrt{1+r}}\\
 \varphi_{u \bar{u}} &= \frac{(1 + z \overline{z})^2}{\sqrt{1+r}}
\end{aligned}
\end{equation}
\end{proposition}
It is straightforward to check that this hyper-Kähler structure satisfies all the requirements of \autoref{thm:FeixKaledin} as expected.

One would like to proceed to compute the Feix-Kaledin hyper-Kähler structure of \autoref{def:FKComplexification} in the complexification of $\CP^1$.
This problem seems difficult in general: the natural approach is to find $J$-holomorphic coordinates in the cotangent bundle extending 
real coordinates off the zero section, but such a solution of the Newlander-Nirenberg theorem
is not explicit in general. An alternative approach is to look for extra symmetries in the case at hand
and predict the hyper-Kähler isomorphism $T \colon \CP^1 \times \overline{\CP^1} \to \upT^*\CP^1$, as Thomas Hodge does in the case of 
$\mathbb{H}^2$ in \cite[Lemma 3.1.4]{HodgePHD}. However we were not able to recover an explicit expression of $T$ by using Hodge's ansatz, 
though it seems likely that a small adjustment of it could work. 
% In any case, we shall see in 
% \autoref{subsubsec:CP1AlmostHH} below that on the other hand we are able to explicitly compute the almost hyper-Hermitian structure of 
% \autoref{thm:MainThmHK} in the complexification of $\CP^1$, and that it must be different from this Feix-Kaledin hyper-Kähler structure.
% 

\subsubsection{Almost hyper-Hermitian structure} \label{subsubsec:CP1AlmostHH}

We carry out the explicit computation of the almost hyper-Hermitian structure of \autoref{thm:MainThmHK} in the complexification of $\CP^1$. 
The bulk of the work was already executed in \autoref{subsec:CP1BL}.

\begin{lemma}
Let $\gamma : [0,1] \to N^c$ be a path contained in a vertical leaf. Denote $\gamma(0) = (z_0, w_0)$ and $\gamma(1) = (z_0, w_1)$.
The parallel transport $P_\gamma : T_{(z_0, w_0)} N^c \to T_{(z_0, w_1)} N^c$ is
given by
\begin{equation}
\begin{split}
   \dd{z} &\mapsto \dd{z}\\
   \dd{w} &\mapsto \left(\frac{1 + z_0 w_1}{1 + z_0 w_0}\right)^2 \dd{w}~.
\end{split} 
\end{equation}
\end{lemma}

\begin{proof}
Let $u(t) = a(t) \dd{z} +  b(t) \dd{w}$ be a vector field along $\gamma$. Then
$u(t)$ is parallel along $\gamma$ if and only if $\nabla_t u(t) = 0$, which gives the equations:
\begin{equation}
\left\{\begin{array}{l}
a'(t) = 0\\
b'(t) - \frac{2 z_0 w'(t)}{1 + z_0 w(t)} b(t) = 0~.
\end{array}\right.
\end{equation}
The second equation is easily solved noting that it is an ODE of the form $y'(t) - 2 \frac{\alpha'(t)}{\alpha(t)}y(t) = 0$, solutions to this
are of the form $y(t) = \left(\frac{\alpha(t)}{\alpha(0)}\right)^2\,y_0$. The conclusion follows.
\end{proof}
Note that $P_\gamma$ only depends on the endpoints of $\gamma$ as expected.

We are now ready to compute the almost hyper-Hermitian structure by using vertical parallel transport along the leaves. 
\begin{theorem}
The almost hyper-Hermitian structure in $N^c$ is given by:
\begin{equation}
\begin{split}
g &= \frac{1}{2} \left[\frac{1}{\left(1 + |z|^2\right)^2}\dz \dzbar
+ \frac{\left(1 + |z|^2\right)^2}{\left|1 + z w\right|^4}\dw \dwbar\right]\\
I &= i \left[\overline{\eta}\, \dwbar \otimes \dd{z} - \eta \,\dw \otimes \dd{\overline{z}}
- \frac{1}{\eta}\, \dzbar \otimes \dd{w} + \frac{1}{\overline{\eta}} \,\dz \otimes \dd{\overline{w}}\right]\\
J &= i \left[\dz \otimes \dd{z} - \dzbar \otimes \dd{\overline{z}}
+ \dw \otimes \dd{w} - \dwbar \otimes \dd{\overline{w}}\right]\\
K &= \overline{\eta}\, \dwbar \otimes \dd{z} + \eta \,\dw \otimes \dd{\overline{z}}
- \frac{1}{\eta}\, \dzbar \otimes \dd{w} - \frac{1}{\overline{\eta}} \,\dz \otimes \dd{\overline{w}}\\
\omega_I &= \frac{i}{4} \left[\frac{\dz \wedge \dw}{\left(1 + z w\right)^2} 
- \frac{\dzbar \wedge \dwbar}{\left(1 + \overline{z} \overline{w}\right)^2}\right]\\
\omega_J &= \frac{i}{4} \left[\frac{1}{\left(1 + |z|^2\right)^2}\dz \wedge \dzbar 
+ \frac{\left(1 + |z|^2\right)^2}{\left|1 + z w\right|^4}\dw \wedge \dwbar\right]\\
\omega_K &= \frac{-1}{4} \left[\frac{\dz \wedge \dw}{\left(1 + z w\right)^2} 
+ \frac{\dzbar \wedge \dwbar}{\left(1 + \overline{z} \overline{w}\right)^2}\right]
\end{split}
\end{equation}
where we have written $\eta = \left(\frac{1 + |z|^2}{1 + z w}\right)^2$ in the expressions of $I$, $J$, $K$.
\end{theorem}

It is clear from the expressions of $\omega_I$, $\omega_J$, and $\omega_K$ that $\omega_I$ and $\omega_K$ are both closed, 
in fact $\omega_0^c = \frac{i \, \dz \wedge \dw}{2\left(1 + zw\right)^2} = \omega_I -i \omega_K$ as expected.
One can also check that all the properties of an admissible extension in the sense of \autoref{def:AdmissibleExtension} are indeed satisfied.

On the other hand, we observe that $\omega_J$ is not closed, meaning that the hyper-Hermitian structure is not integrable. Thus the Levi-Civita connection
of $g$ does not parallelize $J$, unlike the bi-Lagrangian connection. In particular, the metric $g$ (and hence the hyper-Hermitian structure)
differs from the Feix-Kaledin hyper-Kähler structure.

\cleardoublepage\phantomsection % So that the hyperref thumbnail is correct
\bibliographystyle{alpha}
\newgeometry{top=0.09\paperheight, bottom=0.09\paperheight, left=0.10\paperwidth, right=0.10\paperwidth}
{\small \bibliography{biblio}}

\begin{thebibliography}{DGMY09}

\bibitem[AB60]{MR0115006}
Lars Ahlfors and Lipman Bers.
\newblock Riemann's mapping theorem for variable metrics.
\newblock {\em Ann. of Math. (2)}, 72:385--404, 1960.

\bibitem[AB83]{MR702806}
M.~F. Atiyah and R.~Bott.
\newblock The {Y}ang-{M}ills equations over {R}iemann surfaces.
\newblock {\em Philos. Trans. Roy. Soc. London Ser. A}, 308(1505):523--615,
  1983.

\bibitem[Ahl61]{MR0204641}
Lars~V. Ahlfors.
\newblock Some remarks on {T}eichm\"uller's space of {R}iemann surfaces.
\newblock {\em Ann. of Math. (2)}, 74:171--191, 1961.

\bibitem[Ahl64]{MR0167618}
Lars~V. Ahlfors.
\newblock Finitely generated {K}leinian groups.
\newblock {\em Amer. J. Math.}, 86:413--429, 1964.

\bibitem[AMT08]{AMTParaK}
D.~V. Alekseevski\u\i, K.~Medori, and A.~Tomassini.
\newblock Homogeneous para-{K}\"ahlerian {E}instein manifolds.
\newblock {\em Preprint:
  \href{http://arxiv.org/abs/0806.2272}{arXiv:0806.2272}}, 2008.

\bibitem[And98]{MR2698860}
Charles~Gregory Anderson.
\newblock {\em Projective structures on {R}iemann surfaces and developing maps
  to {H}(3) and {CP}(n)}.
\newblock ProQuest LLC, Ann Arbor, MI, 1998.
\newblock Thesis (Ph.D.)--University of California, Berkeley.

\bibitem[Bab13]{BabaII}
Shinpei Baba.
\newblock {$2\pi$}-grafting and complex projective structures, {II}, 2013.

\bibitem[Bab15]{MR3447103}
Shinpei Baba.
\newblock {$2\pi$}-grafting and complex projective structures, {I}.
\newblock {\em Geom. Topol.}, 19(6):3233--3287, 2015.

\bibitem[Bau14]{MR3289710}
Oliver Baues.
\newblock The deformation of flat affine structures on the two-torus.
\newblock In {\em Handbook of {T}eichm\"uller theory. {V}ol. {IV}}, volume~19
  of {\em IRMA Lect. Math. Theor. Phys.}, pages 461--537. Eur. Math. Soc.,
  Z\"urich, 2014.

\bibitem[BE64]{MR0168800}
Lipman Bers and Leon Ehrenpreis.
\newblock Holomorphic convexity of {T}eichm\"uller spaces.
\newblock {\em Bull. Amer. Math. Soc.}, 70:761--764, 1964.

\bibitem[Ber60a]{MR0111834}
Lipman Bers.
\newblock Simultaneous uniformization.
\newblock {\em Bull. Amer. Math. Soc.}, 66:94--97, 1960.

\bibitem[Ber60b]{MR0111835}
Lipman Bers.
\newblock Spaces of {R}iemann surfaces as bounded domains.
\newblock {\em Bull. Amer. Math. Soc.}, 66:98--103, 1960.

\bibitem[Ber61]{MR0130972}
Lipman Bers.
\newblock Correction to ``{S}paces of {R}iemann surfaces as bounded domains''.
\newblock {\em Bull. Amer. Math. Soc.}, 67:465--466, 1961.

\bibitem[Ber66]{MR0192046}
Lipman Bers.
\newblock A non-standard integral equation with applications to quasiconformal
  mappings.
\newblock {\em Acta Math.}, 116:113--134, 1966.

\bibitem[Ber70]{MR0297992}
Lipman Bers.
\newblock On boundaries of {T}eichm\"uller spaces and on {K}leinian groups.
  {I}.
\newblock {\em Ann. of Math. (2)}, 91:570--600, 1970.

\bibitem[Ber87]{MR910353}
Lipman Bers.
\newblock On {S}ullivan's proof of the finiteness theorem and the eventual
  periodicity theorem.
\newblock {\em Amer. J. Math.}, 109(5):833--852, 1987.

\bibitem[BG05]{MR2181958}
Oliver Baues and William~M. Goldman.
\newblock Is the deformation space of complete affine structures on the 2-torus
  smooth?
\newblock In {\em Geometry and dynamics}, volume 389 of {\em Contemp. Math.},
  pages 69--89. Amer. Math. Soc., Providence, RI, 2005.

\bibitem[BKN17]{MR3735629}
Marco Bertola, Dmitry Korotkin, and Chaya Norton.
\newblock Symplectic geometry of the moduli space of projective structures in
  homological coordinates.
\newblock {\em Invent. Math.}, 210(3):759--814, 2017.

\bibitem[Bry95]{MR1338391}
Robert~L. Bryant.
\newblock An introduction to {L}ie groups and symplectic geometry.
\newblock In {\em Geometry and quantum field theory ({P}ark {C}ity, {UT},
  1991)}, volume~1 of {\em IAS/Park City Math. Ser.}, pages 5--181. Amer. Math.
  Soc., Providence, RI, 1995.

\bibitem[BW97]{MR1806388}
Sean Bates and Alan Weinstein.
\newblock {\em Lectures on the geometry of quantization}, volume~8 of {\em
  Berkeley Mathematics Lecture Notes}.
\newblock American Mathematical Society, Providence, RI; Berkeley Center for
  Pure and Applied Mathematics, Berkeley, CA, 1997.

\bibitem[BW11]{MR2746980}
Paul Baird and John~C. Wood.
\newblock Harmonic morphisms and bicomplex manifolds.
\newblock {\em J. Geom. Phys.}, 61(1):46--61, 2011.

\bibitem[Cal79]{MR543218}
E.~Calabi.
\newblock M\'etriques k\"ahl\'eriennes et fibr\'es holomorphes.
\newblock {\em Ann. Sci. \'Ecole Norm. Sup. (4)}, 12(2):269--294, 1979.

\bibitem[CFG96]{MR1386154}
V.~Cruceanu, P.~Fortuny, and P.~M. Gadea.
\newblock A survey on paracomplex geometry.
\newblock {\em Rocky Mountain J. Math.}, 26(1):83--115, 1996.

\bibitem[CM04]{MR2096234}
Richard~D. Canary and Darryl McCullough.
\newblock Homotopy equivalences of 3-manifolds and deformation theory of
  {K}leinian groups.
\newblock {\em Mem. Amer. Math. Soc.}, 172(812):xii+218, 2004.

\bibitem[Cor88]{MR965220}
Kevin Corlette.
\newblock Flat {$G$}-bundles with canonical metrics.
\newblock {\em J. Differential Geom.}, 28(3):361--382, 1988.

\bibitem[Cor10]{MR2605651}
Vicente Cort\'es, editor.
\newblock {\em Handbook of pseudo-{R}iemannian geometry and supersymmetry},
  volume~16 of {\em IRMA Lectures in Mathematics and Theoretical Physics}.
\newblock European Mathematical Society (EMS), Z\"urich, 2010.

\bibitem[DGMY09]{MR2554487}
Johann Davidov, Gueo Grantcharov, Oleg Mushkarov, and Miroslav Yotov.
\newblock Para-hyperhermitian surfaces.
\newblock {\em Bull. Math. Soc. Sci. Math. Roumanie (N.S.)},
  52(100)(3):281--289, 2009.

\bibitem[Don87]{MR887285}
S.~K. Donaldson.
\newblock Twisted harmonic maps and the self-duality equations.
\newblock {\em Proc. London Math. Soc. (3)}, 55(1):127--131, 1987.

\bibitem[Don03]{MR2039989}
S.~K. Donaldson.
\newblock Moment maps in differential geometry.
\newblock In {\em Surveys in differential geometry, {V}ol.\ {VIII} ({B}oston,
  {MA}, 2002)}, volume~8 of {\em Surv. Differ. Geom.}, pages 171--189. Int.
  Press, Somerville, MA, 2003.

\bibitem[Dum09]{MR2497780}
David Dumas.
\newblock Complex projective structures.
\newblock In {\em Handbook of {T}eichm\"uller theory. {V}ol. {II}}, volume~13
  of {\em IRMA Lect. Math. Theor. Phys.}, pages 455--508. Eur. Math. Soc.,
  Z\"urich, 2009.

\bibitem[DW08]{MR2436230}
Maciej Dunajski and Simon West.
\newblock Anti-self-dual conformal structures in neutral signature.
\newblock In {\em Recent developments in pseudo-{R}iemannian geometry}, ESI
  Lect. Math. Phys., pages 113--148. Eur. Math. Soc., Z\"urich, 2008.

\bibitem[Ear81]{MR624807}
Clifford~J. Earle.
\newblock On variation of projective structures.
\newblock In {\em Riemann surfaces and related topics: {P}roceedings of the
  1978 {S}tony {B}rook {C}onference ({S}tate {U}niv. {N}ew {Y}ork, {S}tony
  {B}rook, {N}.{Y}., 1978)}, volume~97 of {\em Ann. of Math. Stud.}, pages
  87--99. Princeton Univ. Press, Princeton, N.J., 1981.

\bibitem[EE69]{MR0276999}
Clifford~J. Earle and James Eells.
\newblock A fibre bundle description of {T}eichm\"uller theory.
\newblock {\em J. Differential Geometry}, 3:19--43, 1969.

\bibitem[EH79]{MR540896}
Tohru Eguchi and Andrew~J. Hanson.
\newblock Self-dual solutions to {E}uclidean gravity.
\newblock {\em Ann. Physics}, 120(1):82--106, 1979.

\bibitem[EST06]{MR2193747}
Fernando Etayo, Rafael Santamar{\'{\i}}a, and Uju{\'e}~R. Tr{\'{\i}}as.
\newblock The geometry of a bi-{L}agrangian manifold.
\newblock {\em Differential Geom. Appl.}, 24(1):33--59, 2006.

\bibitem[Fei99]{FeixThesis}
Birte Feix.
\newblock {\em Hyperk\"ahler metrics on Cotangent Bundles}.
\newblock {Ph.D.} in {M}athematics, University of Cambridge, 1999.

\bibitem[Fei01]{MR1817502}
Birte Feix.
\newblock Hyperk\"ahler metrics on cotangent bundles.
\newblock {\em J. Reine Angew. Math.}, 532:33--46, 2001.

\bibitem[FY13]{MR3130569}
Michael Forger and Sandra~Z. Yepes.
\newblock Lagrangian distributions and connections in multisymplectic and
  polysymplectic geometry.
\newblock {\em Differential Geom. Appl.}, 31(6):775--807, 2013.

\bibitem[Gar87]{MR903027}
Frederick~P. Gardiner.
\newblock {\em Teichm\"uller theory and quadratic differentials}.
\newblock Pure and Applied Mathematics (New York). John Wiley \& Sons, Inc.,
  New York, 1987.
\newblock A Wiley-Interscience Publication.

\bibitem[GKM00]{MR1765706}
Daniel Gallo, Michael Kapovich, and Albert Marden.
\newblock The monodromy groups of {S}chwarzian equations on closed {R}iemann
  surfaces.
\newblock {\em Ann. of Math. (2)}, 151(2):625--704, 2000.

\bibitem[GL11]{MR2775123}
Malin G\"oteman and Ulf Lindstr\"om.
\newblock Pseudo-hyperk\"ahler geometry and generalized {K}\"ahler geometry.
\newblock {\em Lett. Math. Phys.}, 95(3):211--222, 2011.

\bibitem[Gol80]{MR2630832}
William~Mark Goldman.
\newblock {\em D{ISCONTINUOUS} {GROUPS} {AND} {THE} {EULER} {CLASS}}.
\newblock ProQuest LLC, Ann Arbor, MI, 1980.
\newblock Thesis (Ph.D.)--University of California, Berkeley.

\bibitem[Gol84]{MR762512}
William~M. Goldman.
\newblock The symplectic nature of fundamental groups of surfaces.
\newblock {\em Adv. in Math.}, 54(2):200--225, 1984.

\bibitem[Gol87]{MR882826}
William~M. Goldman.
\newblock Projective structures with {F}uchsian holonomy.
\newblock {\em J. Differential Geom.}, 25(3):297--326, 1987.

\bibitem[Gol88]{MR957518}
William~M. Goldman.
\newblock Geometric structures on manifolds and varieties of representations.
\newblock In {\em Geometry of group representations ({B}oulder, {CO}, 1987)},
  volume~74 of {\em Contemp. Math.}, pages 169--198. Amer. Math. Soc.,
  Providence, RI, 1988.

\bibitem[Gol04]{MR2094117}
William~M. Goldman.
\newblock The complex-symplectic geometry of {${\rm SL}(2,\C)$}-characters over
  surfaces.
\newblock In {\em Algebraic groups and arithmetic}, pages 375--407. Tata Inst.
  Fund. Res., Mumbai, 2004.

\bibitem[Gol10]{MR2827816}
William~M. Goldman.
\newblock Locally homogeneous geometric manifolds.
\newblock In {\em Proceedings of the {I}nternational {C}ongress of
  {M}athematicians. {V}olume {II}}, pages 717--744. Hindustan Book Agency, New
  Delhi, 2010.

\bibitem[Gro]{GrothendieckTechniques}
Alexander Grothendieck.
\newblock Techniques de construction en g{\'e}om{\'e}trie analyique, i--vi.
\newblock volume~13.

\bibitem[Hat16]{HatcherMO}
Allen Hatcher.
\newblock Homotopically trivial vs isotopically trivial diffeomorphisms
  (answer).
\newblock MathOverflow, 2016.
\newblock \url{http://mathoverflow.net/q/248447} (version: 2016-08-28).

\bibitem[Hej75]{MR0463429}
Dennis~A. Hejhal.
\newblock Monodromy groups and linearly polymorphic functions.
\newblock {\em Acta Math.}, 135(1):1--55, 1975.

\bibitem[He{\ss}80]{MR607691}
Harald He{\ss}.
\newblock Connections on symplectic manifolds and geometric quantization.
\newblock In {\em Differential geometrical methods in mathematical physics
  ({P}roc. {C}onf., {A}ix-en-{P}rovence/{S}alamanca, 1979)}, volume 836 of {\em
  Lecture Notes in Math.}, pages 153--166. Springer, Berlin, 1980.

\bibitem[He{\ss}81]{HessPhD}
Harald He{\ss}.
\newblock {\em Symplectic Connections in Geometric Quantization and Factor
  Orderings}.
\newblock {Ph.D.} in {M}athematics, Freie Universität, Berlin, 1981.

\bibitem[Hit87]{MR887284}
N.~J. Hitchin.
\newblock The self-duality equations on a {R}iemann surface.
\newblock {\em Proc. London Math. Soc. (3)}, 55(1):59--126, 1987.

\bibitem[Hit90]{HitchinHyperSymplectic}
Nigel Hitchin.
\newblock Hypersymplectic quotients.
\newblock {\em Atti Accad. Sci. Torino cl. Sci. Fis. Mat. Natur.},
  124:169--180, 1990.

\bibitem[HKLR87]{MR877637}
N.~J. Hitchin, A.~Karlhede, U.~Lindstr{\"o}m, and M.~Ro{\v{c}}ek.
\newblock Hyper-{K}\"ahler metrics and supersymmetry.
\newblock {\em Comm. Math. Phys.}, 108(4):535--589, 1987.

\bibitem[Hod05]{HodgePHD}
Thomas W.~S. Hodge.
\newblock {\em Hyperk\"ahler Geometry and {T}eichm\"uller Space}.
\newblock {Ph.D.} in {M}athematics, Imperial College London, University of
  London, 2005.

\bibitem[Hub81]{MR624819}
John~H. Hubbard.
\newblock The monodromy of projective structures.
\newblock In {\em Riemann surfaces and related topics: {P}roceedings of the
  1978 {S}tony {B}rook {C}onference ({S}tate {U}niv. {N}ew {Y}ork, {S}tony
  {B}rook, {N}.{Y}., 1978)}, volume~97 of {\em Ann. of Math. Stud.}, pages
  257--275. Princeton Univ. Press, Princeton, N.J., 1981.

\bibitem[Hub06]{MR2245223}
John~Hamal Hubbard.
\newblock {\em Teichm\"uller theory and applications to geometry, topology, and
  dynamics. {V}ol. 1}.
\newblock Matrix Editions, Ithaca, NY, 2006.
\newblock Teichm{\"u}ller theory, With contributions by Adrien Douady, William
  Dunbar, Roland Roeder, Sylvain Bonnot, David Brown, Allen Hatcher, Chris
  Hruska and Sudeb Mitra, With forewords by William Thurston and Clifford
  Earle.

\bibitem[Joy00]{MR1787733}
Dominic~D. Joyce.
\newblock {\em Compact manifolds with special holonomy}.
\newblock Oxford Mathematical Monographs. Oxford University Press, Oxford,
  2000.

\bibitem[JP13]{MR3021551}
Lizhen Ji and Athanase Papadopoulos.
\newblock Historical development of {T}eichm\"uller theory.
\newblock {\em Arch. Hist. Exact Sci.}, 67(2):119--147, 2013.

\bibitem[Kal99]{MR1848662}
D.~Kaledin.
\newblock A canonical hyperk\"ahler metric on the total space of a cotangent
  bundle.
\newblock In {\em Quaternionic structures in mathematics and physics ({R}ome,
  1999)}, pages 195--230. Univ. Studi Roma ``La Sapienza'', Rome, 1999.

\bibitem[Kam99]{MR1715481}
Hiroyuki Kamada.
\newblock Neutral hyper-{K}\"ahler structures on primary {K}odaira surfaces.
\newblock {\em Tsukuba J. Math.}, 23(2):321--332, 1999.

\bibitem[Kap90]{MR1043223}
M.~Kapovich.
\newblock Deformation spaces of flat conformal structures.
\newblock In {\em Proceedings of the {S}econd {S}oviet-{J}apan {J}oint
  {S}ymposium of {T}opology ({K}habarovsk, 1989)}, volume~8, pages 253--264,
  1990.

\bibitem[Kap09]{MR2553578}
Michael Kapovich.
\newblock {\em Hyperbolic manifolds and discrete groups}.
\newblock Modern Birkh\"auser Classics. Birkh\"auser Boston, Inc., Boston, MA,
  2009.
\newblock Reprint of the 2001 edition.

\bibitem[Kaw96]{MR1386110}
Shingo Kawai.
\newblock The symplectic nature of the space of projective connections on
  {R}iemann surfaces.
\newblock {\em Math. Ann.}, 305(1):161--182, 1996.

\bibitem[Kli17]{MR3665000}
Bruno Klingler.
\newblock Chern's conjecture for special affine manifolds.
\newblock {\em Ann. of Math. (2)}, 186(1):69--95, 2017.

\bibitem[Kou94]{MR1288062}
Christos Kourouniotis.
\newblock Complex length coordinates for quasi-{F}uchsian groups.
\newblock {\em Mathematika}, 41(1):173--188, 1994.

\bibitem[Kra72]{MR0306485}
Irwin Kra.
\newblock On spaces of {K}leinian groups.
\newblock {\em Comment. Math. Helv.}, 47:53--69, 1972.

\bibitem[KS58]{MR0112154}
K.~Kodaira and D.~C. Spencer.
\newblock On deformations of complex analytic structures. {I}, {II}.
\newblock {\em Ann. of Math. (2)}, 67:328--466, 1958.

\bibitem[LESSV15]{MR3410909}
M.~Elena Luna-Elizarrar\'as, Michael Shapiro, Daniele~C. Struppa, and Adrian
  Vajiac.
\newblock {\em Bicomplex holomorphic functions}.
\newblock Frontiers in Mathematics. Birkh\"auser/Springer, Cham, 2015.
\newblock The algebra, geometry and analysis of bicomplex numbers.

\bibitem[Lib52a]{MR0048893}
Paulette Libermann.
\newblock Sur les structures presque paracomplexes.
\newblock {\em C. R. Acad. Sci. Paris}, 234:2517--2519, 1952.

\bibitem[Lib52b]{MR0046735}
Paulette Libermann.
\newblock Sur les structures presque quaternioniennes de deuxi\`eme esp\`ece.
\newblock {\em C. R. Acad. Sci. Paris}, 234:1030--1032, 1952.

\bibitem[Lib54]{MR0066020}
Paulette Libermann.
\newblock Sur le probl\`eme d'\'equivalence de certaines structures
  infinit\'esimales.
\newblock {\em Ann. Mat. Pura Appl. (4)}, 36:27--120, 1954.

\bibitem[Lib55]{MR0079766}
Paulette Libermann.
\newblock Sur les structures presque complexes et autres structures
  infinit\'esimales r\'eguli\`eres.
\newblock {\em Bull. Soc. Math. France}, 83:195--224, 1955.

\bibitem[Lou15a]{MR3352248}
Brice Loustau.
\newblock The complex symplectic geometry of the deformation space of complex
  projective structures.
\newblock {\em Geom. Topol.}, 19(3):1737--1775, 2015.

\bibitem[Lou15b]{LoustauAsheville}
Brice Loustau.
\newblock Minimal surfaces and quasi-{F}uchsian structures.
\newblock Notes published online, 2015.
\newblock Available at
  \url{http://www.brice.loustau.eu/ressources/Loustau-Asheville2015.pdf}.

\bibitem[Lou15c]{MR3323643}
Brice Loustau.
\newblock Minimal surfaces and symplectic structures of moduli spaces.
\newblock {\em Geom. Dedicata}, 175:309--322, 2015.

\bibitem[Mar74]{MR0349992}
Albert Marden.
\newblock The geometry of finitely generated kleinian groups.
\newblock {\em Ann. of Math. (2)}, 99:383--462, 1974.

\bibitem[Mas88]{MR959135}
Bernard Maskit.
\newblock {\em Kleinian groups}, volume 287 of {\em Grundlehren der
  Mathematischen Wissenschaften [Fundamental Principles of Mathematical
  Sciences]}.
\newblock Springer-Verlag, Berlin, 1988.

\bibitem[McM00]{MR1745010}
Curtis~T. McMullen.
\newblock The moduli space of {R}iemann surfaces is {K}\"ahler hyperbolic.
\newblock {\em Ann. of Math. (2)}, 151(1):327--357, 2000.

\bibitem[Mer12]{Merker2012}
Jochen Merker.
\newblock {On Almost Hyper-Para-K{\"{a}}hler Manifolds}.
\newblock {\em ISRN Geometry}, 2012:1--13, 2012.

\bibitem[Nih04]{MR2104536}
Terumasa Nihonyanagi.
\newblock Notes on 4-dimensional hyper-para-{K}\"ahler manifolds.
\newblock {\em Arab J. Math. Sci.}, 10(1):43--48, 2004.

\bibitem[NS12]{MR3001608}
Hossein Namazi and Juan Souto.
\newblock Non-realizability and ending laminations: proof of the density
  conjecture.
\newblock {\em Acta Math.}, 209(2):323--395, 2012.

\bibitem[Ohs11]{MR2821565}
Ken'ichi Ohshika.
\newblock Realising end invariants by limits of minimally parabolic,
  geometrically finite groups.
\newblock {\em Geom. Topol.}, 15(2):827--890, 2011.

\bibitem[Pap07]{MR2284826}
Athanase Papadopoulos, editor.
\newblock {\em Handbook of {T}eichm\"uller theory. {V}ol. {I}}, volume~11 of
  {\em IRMA Lectures in Mathematics and Theoretical Physics}.
\newblock European Mathematical Society (EMS), Z\"urich, 2007.

\bibitem[Pla01]{MR1866841}
Ioannis~D. Platis.
\newblock Complex symplectic geometry of quasi-{F}uchsian space.
\newblock {\em Geom. Dedicata}, 87(1-3):17--34, 2001.

\bibitem[Rs48]{MR0056351}
P.~K. Ra\v~sevski\u\i.
\newblock The scalar field in a stratified space.
\newblock {\em Trudy Sem. Vektor. Tenzor. Analizu}, 6:225--248, 1948.

\bibitem[Sik12]{MR2931326}
Adam~S. Sikora.
\newblock Character varieties.
\newblock {\em Trans. Amer. Math. Soc.}, 364(10):5173--5208, 2012.

\bibitem[Sim91]{MR1159261}
Carlos~T. Simpson.
\newblock Nonabelian {H}odge theory.
\newblock In {\em Proceedings of the {I}nternational {C}ongress of
  {M}athematicians, {V}ol.\ {I}, {II} ({K}yoto, 1990)}, pages 747--756. Math.
  Soc. Japan, Tokyo, 1991.

\bibitem[Sim92]{MR1179076}
Carlos~T. Simpson.
\newblock Higgs bundles and local systems.
\newblock {\em Inst. Hautes \'Etudes Sci. Publ. Math.}, (75):5--95, 1992.

\bibitem[Sul85]{MR806415}
Dennis Sullivan.
\newblock Quasiconformal homeomorphisms and dynamics. {II}. {S}tructural
  stability implies hyperbolicity for {K}leinian groups.
\newblock {\em Acta Math.}, 155(3-4):243--260, 1985.

\bibitem[Tak17]{TakhtajanKawai}
Leon Takhtajan.
\newblock On {K}awai theorem for orbifold {R}iemann surfaces.
\newblock {\em Preprint:
  \href{http://arxiv.org/abs/1708.09052}{arXiv:1708.09052}}, 2017.

\bibitem[Tan94]{MR1266284}
Ser~Peow Tan.
\newblock Complex {F}enchel-{N}ielsen coordinates for quasi-{F}uchsian
  structures.
\newblock {\em Internat. J. Math.}, 5(2):239--251, 1994.

\bibitem[Thu80]{ThurstonNotes}
William~P. Thurston.
\newblock Geometry and topology of 3-manifolds.
\newblock Notes published online by MSRI:
  \url{http://library.msri.org/books/gt3m/}, 1980.

\bibitem[Thu97]{MR1435975}
William~P. Thurston.
\newblock {\em Three-dimensional geometry and topology. {V}ol. 1}, volume~35 of
  {\em Princeton Mathematical Series}.
\newblock Princeton University Press, Princeton, NJ, 1997.
\newblock Edited by Silvio Levy.

\bibitem[Tra18]{Trautwein2018}
Samuel Trautwein.
\newblock The {D}onaldson hyperkähler metric on the almost-{F}uchisan moduli
  space.
\newblock {\em Preprint:
  \href{https://arxiv.org/abs/1809.00869}{arXiv:1809.00869}}, 2018.

\bibitem[Tro92]{MR1164870}
Anthony~J. Tromba.
\newblock {\em Teichm\"uller theory in {R}iemannian geometry}.
\newblock Lectures in Mathematics ETH Z\"urich. Birkh\"auser Verlag, Basel,
  1992.
\newblock Lecture notes prepared by Jochen Denzler.

\bibitem[TT03]{MR1997440}
Leon~A. Takhtajan and Lee-Peng Teo.
\newblock Liouville action and {W}eil-{P}etersson metric on deformation spaces,
  global {K}leinian reciprocity and holography.
\newblock {\em Comm. Math. Phys.}, 239(1-2):183--240, 2003.

\bibitem[Vai89]{MR1038491}
I.~Vaisman.
\newblock Basics of {L}agrangian foliations.
\newblock {\em Publ. Mat.}, 33(3):559--575, 1989.

\bibitem[VK99]{MR1815021}
Misha Verbitsky and Dmitri Kaledin.
\newblock {\em Hyperkahler manifolds}, volume~12 of {\em Mathematical Physics
  (Somerville)}.
\newblock International Press, Somerville, MA, 1999.

\bibitem[Vl11]{MR2868386}
Gabriel~Eduard V\^\i~lcu.
\newblock Para-hyperhermitian structures on tangent bundles.
\newblock {\em Proc. Est. Acad. Sci.}, 60(3):165--173, 2011.

\bibitem[Vor09]{VoronovLecture}
Theodore Voronov.
\newblock Differential geometry, lecture 11.
\newblock Lecture notes published online, 2009.
\newblock Available at \url{http://www.maths.manchester.ac.uk/~tv/}.

\bibitem[Wei58]{MR0124485}
Andr\'e Weil.
\newblock Modules des surfaces de {R}iemann.
\newblock In {\em S\'eminaire {B}ourbaki; 10e ann\'ee: 1957/1958. {T}extes des
  conf\'erences; {E}xpos\'es 152\`a 168; 2e \'ed.corrig\'ee, {E}xpos\'e 168},
  page~7. Secr\'etariat math\'ematique, Paris, 1958.

\bibitem[Wei71]{MR0286137}
Alan Weinstein.
\newblock Symplectic manifolds and their {L}agrangian submanifolds.
\newblock {\em Advances in Math.}, 6:329--346 (1971), 1971.

\bibitem[Wik]{wiki:paracomplex}
Wikipedia.
\newblock Split-complex number --- {W}ikipedia{,} the free encyclopedia.

\bibitem[Wol82]{MR657237}
Scott Wolpert.
\newblock The {F}enchel-{N}ielsen deformation.
\newblock {\em Ann. of Math. (2)}, 115(3):501--528, 1982.

\bibitem[Wol83]{MR690844}
Scott Wolpert.
\newblock On the symplectic geometry of deformations of a hyperbolic surface.
\newblock {\em Ann. of Math. (2)}, 117(2):207--234, 1983.

\bibitem[Wol85]{MR796909}
Scott Wolpert.
\newblock On the {W}eil-{P}etersson geometry of the moduli space of curves.
\newblock {\em Amer. J. Math.}, 107(4):969--997, 1985.

\bibitem[Wol90]{MR1047763}
Scott~A. Wolpert.
\newblock The {B}ers embeddings and the {W}eil-{P}etersson metric.
\newblock {\em Duke Math. J.}, 60(2):497--508, 1990.

\bibitem[Wol10]{MR2641916}
Scott~A. Wolpert.
\newblock {\em Families of {R}iemann surfaces and {W}eil-{P}etersson geometry},
  volume 113 of {\em CBMS Regional Conference Series in Mathematics}.
\newblock Published for the Conference Board of the Mathematical Sciences,
  Washington, DC; by the American Mathematical Society, Providence, RI, 2010.

\bibitem[Yau09]{YauScholarpedia}
S.~Yau.
\newblock {C}alabi-{Y}au manifold.
\newblock {\em Scholarpedia}, 4(8):6524, 2009.
\newblock revision \#170091.

\end{thebibliography}
\restoregeometry

\end{document}